\documentclass[a4paper]{article}
\usepackage{amsmath}
\usepackage{amssymb}
\usepackage{theorem}
\usepackage{latexsym}
\usepackage{subfig}
\usepackage{subfig}
\usepackage{graphicx,color}
\usepackage{multirow}
\usepackage{colortbl}
\usepackage{rotating}
\usepackage[table]{xcolor}
\usepackage{hhline}
\usepackage{floatrow}
\usepackage{mathrsfs}

\DeclareMathAlphabet{\mathantt}{OT1}{antt}{li}{it}
\DeclareMathAlphabet{\mathpzc}{OT1}{pzc}{m}{it}

\definecolor{lightgray}{gray}{0.95}
\numberwithin{equation}{section}

\newfloatcommand{capbtabbox}{table}[][\FBwidth]

\newtheorem{theorem}{Theorem}
\newtheorem{lemma}[theorem]{Lemma}

\newenvironment{proof}[1][Proof]{\begin{trivlist}
\item[\hskip \labelsep {\bfseries #1}]}{\end{trivlist}}

\input amssym.def
\input amssym.tex

\def\der{{\rm d}}
\def\qed{\raise1pt\hbox{\vrule height5pt width5pt depth0pt}}
\DeclareMathOperator{\sn}{sn}             % Jacobian elliptic sinus
\DeclareMathOperator{\cn}{cn}             % Jacobian elliptic cosinus
\DeclareMathOperator{\dn}{dn}             % elliptic functions
\title{The effects of time-dependent dissipation on the basins of attraction
for the pendulum with oscillating support}
\author{James A. Wright{\small$^1$}, Michele Bartuccelli{\small$^1$}, Guido Gentile{\small$^2$} \\ \\
{\small $^1$Department of Mathematics, University of Surrey, Guildford, GU2 7XH, UK} \\
{\small $^2$Dipartimento di Matematica e Fisica, Universit\`{a} di Roma Tre, 00146 Roma, Italy}}
\date{}
\begin{document}
\maketitle
\begin{center}
\line(24,0){400}
\end{center}
\begin{abstract}
We consider a pendulum with vertically oscillating support and time-dependent
damping coefficient which varies until reaching a finite final value. Although
it is the final value which determines which attractors eventually exist,
however the sizes of the corresponding basins of attraction are found
to depend strongly on the full evolution of the dissipation.
In particular we investigate numerically how dissipation monotonically
varying in time changes the sizes of the basins of attraction.
It turns out that, in order to predict the behaviour of the system, it is essential
to understand how the sizes of the basins of attraction
for constant dissipation depend on the damping coefficient.
For values of the parameters where the systems can be considered
as a perturbation of the simple pendulum, which is integrable, we characterise
analytically the conditions under which the attractors exist and study numerically
how the sizes of their basins of attraction depend on the damping coefficient.
Away from the perturbation regime, a numerical study of the attractors
and the corresponding basins of attraction
for different constant values of the damping coefficient produces a
much more involved scenario: changing the magnitude
of the dissipation causes some attractors to disappear either leaving
no trace or producing new attractors by bifurcation,
such as period doubling and saddle-node bifurcation.
Finally we pass to the case of an initially non-constant damping coefficient,
both increasing and decreasing to some finite final value,
and we numerically observe the resulting effects on the sizes
of the basins of attraction: when the damping coefficient varies slowly from
a finite initial value to a different final value, without changing the set of attractors,
the slower the variation the closer the sizes of the basins of attraction
are to those they have for constant damping coefficient fixed at the initial value.
Furthermore, if during the variation of the damping coefficient attractors
appear or disappear, remarkable additional phenomena may occur.
For instance it can happen that, in the limit of very large variation time,
a fixed point asymptotically attracts the entire phase space,
up to a zero measure set, even though no attractor with such a property
exists for any value of the damping coefficient between the extreme values.
\newline \newline{\small \it Mathematical Subject Classification (2000)
34C60, 34C25, 37C60, 58F12, 70K40, 70K50.}
\end{abstract}
\begin{center}
\line(24,0){400}
\end{center}

%%%%%%%%%%%%%%%%%%%%%%%%%%%%%%%%%%%%%%%%%%%%%%%%%%%%%%%%%%%%%%%%
\section{Introduction}
%%%%%%%%%%%%%%%%%%%%%%%%%%%%%%%%%%%%%%%%%%%%%%%%%%%%%%%%%%%%%%%%

Consider the ordinary differential equations
\begin{equation}
\label{GenEq}
\ddot x + G(x,t) + \gamma \dot x = 0 ,  \qquad
\ddot{\theta} + F(\theta,t) + \gamma\dot{\theta} = 0,
\end{equation}
where $(x,\dot x)\in\mathbb{R}^{2}$ and $(\theta,\dot\theta)\in\mathbb{T}\times\mathbb{R}$,
with $\mathbb{T}=\mathbb{R}/2\pi\mathbb{Z}$. The functions $F$ and $G$ are smooth and $2\pi$-periodic
in time $t$ ($F$ is also $2\pi$-periodic in $\theta$); the dots denote derivatives with respect to time.
Equations to describe the motion of one-dimensional physical systems are often of this form,
in which case the functions $G(x,t)$ and $F(\theta,t)$ can be considered as an external driving force
and the parameter $\gamma$ represents the damping coefficient, which we shall assume positive.

First, for convenience, let us summarise some of the already known ideas regarding systems
of the form \eqref{GenEq}, which can be found in the literature \cite{POaSL,CP,DPP}.
When $\gamma$ is fixed at zero, the system is Hamiltonian and no attractors are present.
For $\gamma > 0$ numerical experiments show that a finite set of attractors exist:
this is consistent with Palis' conjecture \cite{PalisConject,FeudelGrebogiHuntYorke,Rodrigues}.
The number of attractors present and the percentage of phase space covered
by their basins of attraction depend upon the chosen values of the parameters
(perturbation parameter $\varepsilon$ and damping coefficient $\gamma$),
but, for all values of the parameters, the union of the corresponding
basins of attraction completely fill the phase space, up to a set of zero measure.
Moreover if the system is a perturbation of an integrable system (perturbation regime),
all attractors found numerically turn out to be either fixed points or periodic solutions
with periods that are rational multiples  of the forcing period (subharmonic solutions);
we cannot exclude the presence of chaotic attractors \cite{GhilWolansky,FeudelGrebogi},
but apparently they either do not arise or seem to be irrelevant.

Generally in the literature the damping coefficient is taken as constant, but
in many physical systems it changes non-periodically over time. This can be due
to several factors, such as the heating or cooling of a mechanical system and the
wear out or rust on mechanical parts. Despite this, usually models and numerical simulations
of such systems only take the final value of dissipation into account
when calculating basins of attraction. The recent paper \cite{CP} puts forward the idea that,
although the final value of dissipation determines which attractors exist, the relative sizes
of their basins of attraction depend on the evolution of the dissipation.
In particular the effect of dissipation increasing to some constant value
over a given time span induces a significant change to the sizes of
the basins of attraction in comparison to those when dissipation is constant.

Let us illustrate in more detail the phenomenology.
Suppose that for two values $\gamma_{0}$ and $\gamma_{1}$
of the damping coefficient, with $\gamma_0 \neq \gamma_1$,
the same set of attractors exists. Provided the difference
between the two values is sufficiently large, the relative sizes of the basins
of attraction under the two coefficients will in general be appreciably different.
If we allow the damping coefficient $\gamma$ to depend on time, $\gamma=\gamma(t)$,
and vary from $\gamma_0$ to $\gamma_1$ over an initial period of time $T_0$, after which
it remains constant at the value $\gamma_{1}$, then the sizes of the basins of attraction
will be different from those where the system has constant coefficient $\gamma_{1}$ throughout.
Moreover if $T_0$ is taken larger, the sizes of the basins of attraction tend
towards those for the system under constant $\gamma = \gamma_0$: this reflects the fact
that the damping coefficient remains close to $\gamma_0$  for longer periods of time.

Now consider two values $\gamma_0$ and $\gamma_1$ of the damping coefficient
for which the corresponding sets of attractors $\mathcal{A}_0$ and
$\mathcal{A}_{1}$ are not the same.
As a system evolving under dissipation is expected to have only
finitely many attractors, there can only be a finite number of attractors which exist
for one of the two values and not for the other one.
What happens is that, by varying $\gamma(t)$ from $\gamma_0$ to $\gamma_1$,
an attractor can either appear or disappear, and in the latter case it can disappear
either without leaving any trace or being replaced by a new attractor by bifurcation.
Suppose, for instance, that the only difference between $\mathcal{A}_{0}$
and $\mathcal{A}_{1}$ is that the attractor $a_0\in\mathcal{A}_{0}$ simply disappears,
that is $\mathcal{A}_0\setminus\mathcal{A}_1=\{a_0\}$; then, if the time $T_0$
over which $\gamma(t)$ varies is large, each remaining attractor tends to have
a basin of attraction not smaller than that it has for $\gamma$ fixed at $\gamma_0$:
the reason being, again, that the damping coefficient remains close to $\gamma_0$
for a long time and, moreover, the trajectories which would be attracted by $a_0$
at $\gamma=\gamma_0$ will move towards some other attractor when $a_0$ disappears.
If, instead, the only difference between the sets of attractors $\mathcal{A}_0$
and $\mathcal{A}_1$ is that the attractor $a_{0}\in\mathcal{A}_{0}$ is replaced
by an attractor $a_{1}$, say by period doubling bifurcation,
then, letting $\gamma(t)$ vary from $\gamma_0$
to $\gamma_1$ over a sufficiently large time $T_0$ causes the
size of the basin of attraction of $a_{1}$ to tend towards that of
the basin of attraction that $a_{0}$ has for $\gamma=\gamma_0$.

We summarise our results by the following statements.
\begin{enumerate}
\item If $\mathcal{A}_0$, the set of attractors at $\gamma = \gamma_0$,
is a subset of $\mathcal{A}_1$, the set of attractors which exist at $\gamma=\gamma_1$,
that is $\mathcal{A}_0 \subseteq \mathcal{A}_1$, then, as the time $T_0$
over which $\gamma(t)$ is varied from $\gamma_0$ to $\gamma_1$
is taken larger, the basins of attraction
tend towards those when $\gamma$ is kept constant at $\gamma = \gamma_0$.
In particular, if an attractor belongs to $\mathcal{A}_1\setminus\mathcal{A}_0$,
then the larger $T_{0}$ the more negligible is the corresponding basin of attraction.
\item If the set of attractors  at fixed $\gamma = \gamma_1$ is a proper subset of those which
exist at $\gamma = \gamma_0$, that is $\mathcal{A}_{1} \subset \mathcal{A}_0$,
then, as $T_0$ is taken larger, the basins of attraction for the attractors
which exist at both $\gamma_0$ and $\gamma_1$ change so that for
$\gamma(t)$ varying from $\gamma_0$ to $\gamma_{1}$
they tend to become greater than or equal to those for constant $\gamma = \gamma_0$.
\item If an attractor $a_0$ exists for $\gamma = \gamma_0$ but is destroyed as $\gamma(t)$
tends towards $\gamma_1$, and a new attractor $a_1$ is created from it by bifurcation
(we will explicitly investigate the case of saddle-node or period doubling bifurcations),
then the size of the basin of attraction of $a_{1}$, as $T_0$ is taken larger,
tends towards that of $a_0$ at constant $\gamma = \gamma_0$.
\item If $\mathcal{A}_{01}$ is the set of attractors which exist at both
$\gamma=\gamma_0$ and $\gamma=\gamma_1$, that is $\mathcal{A}_{01} = \mathcal{A}_0
\cap \mathcal{A}_1$, and none of the elements in $\mathcal{A}_{0} \setminus
\mathcal{A}_{01}$ are linked by bifurcation to elements in $\mathcal{A}_{1} \setminus \mathcal{A}_{01}$,
then, as $T_0$  is taken larger, the phase space covered by the basins of attraction
of the attractors which belong to $\mathcal{A}_{01}$ tends towards 100\%.
Moreover, all such attractors have a basin of attraction larger than or equal to that
they have when the coefficient of dissipation is fixed at $\gamma=\gamma_0$.
\end{enumerate}

The main model used in \cite{CP} to convey some of the ideas above is a
version of the forced cubic oscillator, which is of the form of the first equation
in \eqref{GenEq}, with $G(x, t) = (1 + \varepsilon \cos{t})x^3$.
This system, considered in the perturbation regime
(both $\varepsilon$ and $\gamma$ small), apart from the
fixed point and as far as the numerics fortells, exhibits only oscillatory attractors
with different periods depending on the parameter values.
Also discussed in \cite{CP} is the relevance to the spin-orbit problem,
describing an asymmetric ellipsoidal satellite moving in a Keplerian elliptic orbit
around a planet \cite{spinorbit}:
the corresponding equations of motion
are of the form of the second equation in \eqref{GenEq}, with the tidal friction term
$\gamma(t)\,(\dot\theta-1)$ instead of $\gamma\dot\theta$,
with $\gamma(t)$ slowly increasing in time because of the the cooling of the satellite.

In the present paper we wish to extend the discussion to the pendulum with
periodically oscillating support \cite{LL,EngAAV}. The latter is a system which
has been already extensively studied in the literature (we refer to \cite{DPP}
for a list of references): it offers a wide variety of dynamics and, because of the separatrix
of the unperturbed system, in the perturbation regime, unlike the cubic oscillator,
also includes rotatory attractors in addition to the oscillatory attractors.
An important difference with respect to the results in \cite{POaSL} is the following.
In \cite{POaSL}, if an attractor exists for some value
of $\gamma$, it is found to exist for smaller values of $\gamma$ too.
This is not always true for the pendulum considered in the present paper,
where we will see that, at least for some values of the parameters,
both increasing and decreasing $\gamma$ can destroy attractors as well as create new ones.
However, this occurs away from the perturbation regime, where the system
can no longer be considered as a perturbation of an integrable one:
the appearance and disappearance of attractors would occur also
in the case of the cubic oscillator for larger values of the forcing.
In  addition to the case of increasing dissipation studied in \cite{CP},
here we also include the case where the damping coefficient decreases to a constant value,
which is appropriate for physical systems where joints are initially tight and
require time to loosen. In this case similar phenomena are expected. For instance,
as the value of variation time $T_0$ is taken larger, the amount of phase space
covered by each of the basins of attraction should tend towards that corresponding
to original value $\gamma_0$ of $\gamma$, providing the set of attractors remains the same.

\vspace{.3truecm}

The non-linear pendulum with vertically oscillating support is described by
\begin{equation}
\label{PenEqn}
\ddot{\theta} + f(t)\sin{\theta} + \gamma\dot{\theta} = 0, \qquad
f(t) = \left(\frac{g}{\ell} - \frac{b\omega^2}{\ell} f_{0}{(\omega t)}\right),
\end{equation}
where $f_{0}$ is a smooth $2\pi$-periodic function and
the parameters $\ell$, $b$, $\omega$ and $g$ represent the length,
amplitude and frequency of the oscillations of the support and the gravitational acceleration,
respectively, all of which remain constant; for the sake of simplicity
we shall take $f_{0}(\omega t)=\cos (\omega t)$ in \eqref{PenEqn},
as in \cite{DPP,InvPenStab,BGG3}. As mentioned above,
the parameter $\gamma$ represents the damping coefficient, which, for
analysis where it remains constant, we shall model as
$\gamma = C_{n} \varepsilon^{n}$, where $\varepsilon$ is small and $n$ is an integer.
We shall consider \eqref{PenEqn} as a pair of coupled first order non-autonomous
differential equations by letting $x = \theta$ and $y = \dot{x}$,
such that the phase space is $\mathbb{T}\times\mathbb{R}$
and the system can be written as
\begin{equation} \nonumber
\dot{x} = y , \qquad
\dot{y} = \displaystyle{ -\left(\frac{g}{\ell} - \frac{b\omega^2}{\ell}
\cos{(\omega t)}\right)\sin{x} - \gamma y .}
\end{equation}
The system described by \eqref{PenEqn} can be non-dimensionalised by taking
\[\alpha = \frac{g}{\ell\omega^2}, \qquad \beta = \frac{b}{\ell}, \qquad \tau = \omega t,\]
so that it becomes
\begin{equation}
\label{PenEqnTau}
\theta''  + f(\tau)\sin{\theta} + \gamma\theta' = 0, \qquad f(\tau) = (\alpha - \beta \cos{\tau}),
\end{equation}
or, written as a system of first order differential equations,
\begin{equation} \label{PenEqnTau-xy}
x' = y , \qquad
y' = - f(\tau) \, \sin x - \gamma y ,
\end{equation}
where the dashes represent differentiation with respect to the new time $\tau$
and $\gamma$ has been normalised so as not to contain the frequency.
Linearisation of the system about either fixed point results in a system
of the form of Mathieu's equation, see for instance \cite{HillsEqnBook}.
When the downwards fixed point is linearly stable, it is possible,
for certain parameter values, to prove analytically the conditions for
which the fixed point attracts a full measure set of initial conditions; see Appendix \ref{LinearisedApp}.

In the Sections that follow we shall use the non-dimensionalised version of the system
\eqref{PenEqnTau}, preferable for numerical implementation as it reduces the
number of parameters in the system.
In Section \ref{ThresholdsSect} we detail the calculations of the
threshold values for the attractors, that is the values of constant $\gamma$ below which
periodic attractors exist in the perturbation regime (small $\beta$).
As we shall see, because of the presence of the separatrix for the unperturbed
pendulum, this will be of limited avail for practical purposes:
the persisting periodic solutions found to first order
are in general too close to the separatrix for the perturbation theory to converge.
In Section \ref{PenNumericSect} we present numerical results, in the case of
both constant and non-constant (either increasing or decreasing) dissipation,
for values of the parameters in the perturbation regime. Since for such values
the downwards position turns out to be stable, we shall refer to this case
as the downwards pendulum. Next, in Section \ref{InvPenNumericSect}
we perform the numerical analysis for values of the parameters for which the
upwards position is stable (hence such a case will be referred to as the inverted pendulum).
Of course such parameter values are far away from the perturbation regime:
as a consequence additional phenomena occur, including period doubling and
saddle-node bifurcations.
In Section \ref{NumericsSection} we include a discussion of numerical methods used.
Finally in Section \ref{Conclusions} we draw our conclusions and briefly discuss
some open problems and possible directions for future investigation.

%%%%%%%%%%%%%%%%%%%%%%%%%%%%%%%%%%%%%%%%%%%%%%%%%%%%%%%%%%%%%%%%
\section{Thresholds values for the attractors}
\label{ThresholdsSect}
%%%%%%%%%%%%%%%%%%%%%%%%%%%%%%%%%%%%%%%%%%%%%%%%%%%%%%%%%%%%%%%%

The method used below to calculate the threshold values of $\gamma$ below which
given attractors exist follows that described in \cite{POaSL,CP}, where it was applied to
the damped quartic oscillator and the spin-orbit model.
We consider the system \eqref{PenEqnTau-xy},
with $\beta=\varepsilon$ and $\gamma=C_{1}\varepsilon$, where $\varepsilon,C_{1}>0$.
This approach is well suited to compute the leading order of the threshold values.
In general, it would be preferable to write $\gamma$ as a function of $\varepsilon$ of the form
$\gamma=C_{1}\varepsilon + C_{2}\varepsilon^2+\ldots$ (bifurcation curve), and fix the constants
$C_{k}$ by imposing  formal solubility of the equations to any perturbation order, see \cite{17FromCubic};
however this only produces higher order corrections to the leading order value.

For $\varepsilon=0$ the system reduces to the simple pendulum
$\theta''+\alpha \sin \theta = 0 $, which admits periodic solutions inside the separatrix
(librations or oscillations) and outside the separatrix (rotations).
In terms of the variables $(x,y)$ the equations \eqref{PenEqnTau-xy}
become $x'=y$, $y'=-\sin x$: the librations are described by
\begin{equation}\begin{cases}
x_{\rm osc}(\tau) =
2\arcsin{\left[k_1 \sn{ \left( \sqrt{\alpha}(\tau-\tau_0),k_1\right)}\right]}, \\
y_{\rm osc}(\tau) =
2k_1\sqrt{\alpha} \cn{ \left( \sqrt{\alpha}(\tau-\tau_0),k_1\right)},
\end{cases}
\qquad \qquad k_1 < 1 ,
\label{AAvariablesOsc}
\end{equation}
while the rotations are described by
\begin{equation}\begin{cases}
x_{\rm rot}(\tau) =
2 \arcsin{\left[ \sn{\left( \sqrt{\alpha}(\tau-\tau_0)/k_2 ,k_2\right)}\right]}, \\
y_{\rm rot}(\tau) =
2 k_2^{-1}\sqrt{\alpha} \dn{\left( \sqrt{\alpha}(\tau-\tau_0) / k_2. k_2\right)},
\end{cases}
\qquad \qquad k_2 < 1 ,
\label{AAvariablesRot}
\end{equation}
where $\cn{(\cdot, k)}$, $\sn(\cdot, k)$ and $\dn(\cdot, k)$ are
the Jacobi elliptic functions with elliptic modulus $k$ \cite{EllipOne,EllipTwo,DFLaw,WhittWat},
and $k_{1}$ and $k_{2}$ are such that $k_1^2 = (E+\alpha)/2\alpha$ and $k_2^2 = 1/k_1^2$,
with $E$ being the energy of the pendulum.
From \eqref{AAvariablesOsc} and \eqref{AAvariablesRot}
it can be seen that the solutions are functions of $(\tau - \tau_0)$,
so that the phase of a solution depends on the initial conditions.
We can fix the phase of the solution to zero without loss of generality
by instead writing $f(\tau)$ in equation \eqref{PenEqnTau-xy} as $f(\tau - \tau_0)$.
This moves the freedom of choice in the initial condition to the phase of the forcing.

The dynamics of the simple pendulum can be conveniently written in terms of
action-angle variables $(I,\varphi)$, for which we obtain two sets of variables:
for the librations inside the separatrix one expresses the action as
\begin{equation} \label{I-k-librations}
I = \frac{8}{\pi}\sqrt{\alpha}\Bigl[(k_1^2 - 1){\bf K}(k_1) + {\bf E}(k_1)\Bigr] ,
\end{equation}
where ${\bf K}(k)$ and ${\bf E}(k)$ are the complete elliptic integrals
of the first and second kind, respectively, and writes
\begin{equation} \label{action-angle-librations}
x = 2\arcsin{\left[k_1 \sn{\left(\frac{2{\bf K}(k_1)}{\pi}\varphi,k_1\right)}\right]}, \qquad
y = 2 k_1 \sqrt{\alpha} \cn{\left(\frac{2{\bf K}(k_1)}{\pi}\varphi,k_1\right)} ,
\end{equation}
with $k_{1}$ obtained by inverting \eqref{I-k-librations}, while for the rotations
outside the separatrix one expresses the actions as
\begin{equation} \label{I-k-rotations}
I = \frac{4}{k_2\pi}\sqrt{\alpha} \, {\bf E}(k_2),
\end{equation}
and writes
\begin{equation} \label{action-angle-rotations}
x = 2 \arcsin{\left[ \sn{\left(\frac{{\bf K}(k_2)}{\pi}\varphi,k_2\right)}\right]}, \qquad
y = \frac{2}{k_2} \sqrt{\alpha} \, \dn{\left(\frac{{\bf K}(k_2)}{\pi}\varphi,k_2\right)},
\end{equation}
with $k_{2}$ obtained by inverting \eqref{I-k-rotations};
further details can be found in Appendix \ref{ActionAngleApp}.

For $\varepsilon$ small, in order to compute the thresholds values,
we first write the equations  of motion for the perturbed system
in terms of the action-angle coordinates $(I,\varphi)$ of the simple pendulum,
then we look for solutions in the form of power series expansions in $\varepsilon$,
\begin{equation} \label{perturbedsolution}
I(\tau) = \sum_{n = 0}^{\infty} \varepsilon^n I^{(n)}(\tau) , \qquad
\varphi(\tau) = \sum_{n = 0}^{\infty} \varepsilon^n \varphi^{(n)}(\tau),
\end{equation}
where $I^{(0)}(\tau)$ and $\varphi^{(0)}(\tau)$ are the solutions
to the unperturbed system, that is, see Appendix \ref{ActionAngleApp},
$(I^{(0)}(\tau),\varphi^{(0)}(\tau))=(I_{\rm osc},\varphi_{\rm osc}(\tau))$
and $(I^{(0)}(\tau),\varphi^{(0)}(\tau))=(I_{\rm rot},\varphi_{\rm rot}(\tau))$,
in the case of oscillations and rotations, respectively, with
\begin{equation} \label{I-varphi}
\begin{split}
& I_{\rm osc} = \frac{8}{\pi}\sqrt{\alpha}\Bigl[(k_1^2 - 1){\bf K}(k_1)
+ {\bf E}(k_1)\Bigr], \qquad \varphi_{\rm osc}(\tau) =
\frac{\pi}{2{\bf K}(k_{1})} \sqrt{\alpha} \, \tau , \\
&
I_{\rm rot} = \frac{4}{k_2\pi}\sqrt{\alpha} \, {\bf E}(k_2) ,\qquad
\varphi_{\rm rot}(\tau)= \frac{\pi}{{\bf K}(k_{2})}
\sqrt{\alpha} \, \frac{\tau}{k_{2}} ,
\end{split}
\end{equation}
and with given $k_{1}=k_{1}^{(0)}$ and $k_{2}=k_{2}^{(0)}$.

As the solution \eqref{perturbedsolution} is found using perturbation theory, its validity is restricted
to the system where $\varepsilon$ is comparatively small. In particular this limitation has the result
that the calculations of the threshold values are not valid for the inverted pendulum,
where large $\varepsilon$ is required to stabilise the system. On the other hand the
regime of small $\varepsilon$ has the advantage that we can characterise analytically
the attractors and hence allows a better understanding of the dynamics with respect
to the case of large $\varepsilon$, where only numerical results are available.

%%%%%%%%%%%%%%%%%%%%%%%%%%%%%%%%%%%%%%%%%%%%%%%%%%%%%%%%%%%%%%%%
\subsection{Librations}
\label{LibThresh}
%%%%%%%%%%%%%%%%%%%%%%%%%%%%%%%%%%%%%%%%%%%%%%%%%%%%%%%%%%%%%%%%

We first write the equations of motion \eqref{PenEqnTau-xy} in action-angle variables,
see Appendix \ref{AppPerturbedSys}, as
\begin{equation}\label{PerturbedAAvariablesOscTau}
\begin{split}
\varphi' & = \displaystyle{ \frac{\pi \sqrt{\alpha}}{2{\bf K}(k_1)} -
\frac{\varepsilon \pi}{2{\bf K}(k_1)\,
\sqrt{\alpha}} \left[ \sn^2(\cdot) + \frac{k_1^2\sn^2(\cdot)\cn^2(\cdot)}{1-k_1^2}  -
\frac{{\bf Z}(\cdot) \sn(\cdot)\cn(\cdot)\dn(\cdot)}{1 - k_1^2}\right]\cos(\tau - \tau_0) } \\
& \qquad \qquad \quad \displaystyle{ + \frac{C_{1} \varepsilon \,  \pi \cn(\cdot)}{2{\bf K}(k_1)}
\left[ \frac{\sn(\cdot)}{\dn(\cdot)} + \frac{k_1^2\sn(\cdot)\cn^2(\cdot)}{(1-k_1^2)\dn(\cdot)} -
\frac{ {\bf Z}(\cdot)\cn(\cdot)}{1 - k_1^2}\right], } \\
I' & = \displaystyle{ \frac{8 \varepsilon k_1^2 {\bf K}(k_1)}{\pi}
\cos(\tau - \tau_0)\sn(\cdot)\cn(\cdot)\dn(\cdot) -
\frac{8 C_{1} \varepsilon \, k_1^2 \sqrt{\alpha}\,  {\bf K}(k_1)}{\pi}\cn^2(\cdot) ,}
\end{split}
\end{equation}
where ${\bf Z}(\cdot)$ is the Jacobi zeta function, see \cite{DFLaw}.
Here and throughout Section \ref{LibThresh} to save clutter we define $(\cdot) =
\left(\frac{2{\bf K}(k_1)}{\pi}\varphi ,k_1\right)$. Note that in \eqref{PerturbedAAvariablesOscTau},
the dependence on $I$ of the vector field is through the variable $k_{1}$,
according to \eqref{I-k-librations}.

The coordinates for the unperturbed system ($\varepsilon=0$) satisfy
\begin{equation}\label{WronLinSys}
\varphi' = \frac{\der E}{\der I} := \Omega(I) =
\frac{\pi \sqrt{\alpha}}{2{\bf K}(k_1)} , \qquad I'= 0 .
\end{equation}
Linearising around $(\varphi^{(0)}(\tau), I^{(0)}(\tau)) = (\Omega(I^{(0)})\,\tau, I^{(0)})$, we have
\begin{equation} \label{linearised}
\delta \varphi'  = \frac{\partial \Omega}{\partial I}(I^{(0)}) \, \delta I , \qquad \delta I' = 0,
\end{equation}
where, see Appendix \ref{ActionAngleApp},
\begin{equation} \label{OmegaDer}
\zeta(I) := \frac{\partial \Omega}{\partial I}(I) = - \frac{\pi^2}{16 k_1^2 {\bf K}^3(k_1)}
\left[\frac{{\bf E}(k_1)}{1 - k_1^2} - {\bf K}(k_1)\right].
\end{equation}
Since $I = I(k_1)$, that is the action is a function of $k_{1}$,
setting $I = I^{(0)}$ fixes $k_1 = k_1^{(0)}$, yielding
$\zeta(I^{(0)})=\zeta^{(0)}$, with $\zeta^{(0)}$ given by \eqref{OmegaDer}
with $k_{1}=k_{1}^{(0)}$.

The linearised system \eqref{linearised} can by written in compact form as
\begin{equation} \label{linearisedvector}
\begin{pmatrix} \delta\varphi ' \\ \delta I' \end{pmatrix} =
\begin{pmatrix} 0 & \zeta^{(0)} \\ 0 & 0 \end{pmatrix}
\begin{pmatrix} \delta\varphi \\ \delta I \end{pmatrix} .
\end{equation}
The Wronskian matrix $W(\tau)$ is defined as the solution of the
unperturbed linear system
\[W'(\tau) = \begin{pmatrix} 0 & \zeta^{(0)} \\ 0 & 0 \end{pmatrix}
W(\tau), \qquad W(0) = \mathbb{I},\]
where $\mathbb{I}$ is the $2 \times 2$ identity matrix. Hence
\begin{equation} \label{wronskian}
W(\tau) = \begin{pmatrix} 1 & \zeta^{(0)} \tau \\ 0 & 1 \end{pmatrix} ,
\end{equation}
with $(1,0)$ and $(\zeta^{(0)}\tau,1)$ two linearly independent solutions to \eqref{linearisedvector}.

We now look for periodic solutions $(\varphi(\tau), I(\tau) )$ to
\eqref{PerturbedAAvariablesOscTau} with period
$T= 2\pi \mathpzc{q}=4{\bf K}(k_{1})\mathpzc{p}/\sqrt{\alpha}$,
with $\mathpzc{p/q} \in \mathbb{Q}$, of the form \eqref{perturbedsolution};
see also \cite{17FromCubic,18FromCubic} for a more general discussion.
A solution of this kind will be referred to as a $\mathpzc{p}\,$:$\,\mathpzc{q}$ resonance.

The functions $(\varphi^{(n)}(\tau),I^{(n)}(\tau))$ are formally obtained
by introducing the expansions \eqref{perturbedsolution} into
the equations \eqref{PerturbedAAvariablesOscTau} and equating the coefficients of  order $n$.
This leads to the equations
\begin{equation} \label{zdiffEqn}
\begin{pmatrix} (\varphi^{(n)})' \\ (I^{(n)})' \end{pmatrix} =
\begin{pmatrix} \zeta^{(0)} \, I^{(n)} \\ 0\end{pmatrix} +
\begin{pmatrix} F_1^{(n)}(\tau) \\ F_2^{(n)}(\tau) \end{pmatrix}
\end{equation}
with $F_1^{(n)}(\tau)$ and $F_2^{(n)}(\tau)$ given by
\begin{equation} \nonumber
\begin{split}
F_1^{(n)}(\tau) & = \left[ \frac{\pi\sqrt{\alpha}}{2{\bf K}(k_1)} -
\zeta^{(0)} I\right]^{(n)} + \Biggl[ -\frac{\pi}{2{\bf K}(k_1)\,\sqrt{\alpha}}
\left[ \sn^2(\cdot) + \frac{k_1^2 }{1-k_1^2} \sn^2(\cdot) \cn^2(\cdot)\right. \\
& \qquad \qquad \qquad \left.  -
\frac{{\bf Z}(\cdot) }{1 - k_1^2} \sn(\cdot)\cn(\cdot)\dn(\cdot) \right] \cos(\tau - \tau_0)
\\ & \qquad \qquad \qquad
+ \frac{C_{1} \pi \cn(\cdot)}{2{\bf K}(k_1)}\left[ \frac{\sn(\cdot)}{\dn(\cdot)}
+ \frac{k_1^2\sn(\cdot)\cn^2(\cdot)}{(1-k_1^2)\dn(\cdot)} -
\frac{ {\bf Z}(\cdot)\cn(\cdot)}{1 - k_1^2} \right] \Biggr]^{(n-1)} , \\
F_2^{(n)}(\tau) & = \Biggl[ \frac{8 k_1^2 {\bf K}(k_1)}{\pi}
\cos(\tau - \tau_0) \sn(\cdot) \cn(\cdot) \dn(\cdot) -
\frac{8 C_{1} k_1^2\sqrt{\alpha} \, {\bf K}(k_1)}{\pi} \cn^2(\cdot) \Biggr]^{(n-1)}.
\end{split}
\end{equation}
The notation $[\dots]^{(n)}$ means that one has to take all terms of order $n$ in $\varepsilon$
of the function inside $[\ldots]$. 
By construction, $F_1^{(n)}(\tau)$ and $F_2^{(n)}(\tau) $ depend
only  on the coefficients $\varphi^{(p)}(\tau)$ and $I^{(p)}(\tau)$,
with $p<n$, so that \eqref{zdiffEqn} can be solved recursively.

Then, by using the Wronskian matrix \eqref{wronskian}, see \cite{DiffEqnsBook},
one can integrate \eqref{zdiffEqn} so as to obtain
\begin{equation}\label{zEqn}
\begin{pmatrix} \varphi^{(n)}(\tau) \\ I^{(n)}(\tau) \end{pmatrix} =
W(\tau) \begin{pmatrix} \bar{\varphi}^{(n)} \\ \bar{I}^{(n)} \end{pmatrix} +
W(\tau)\int_0^{\tau} \der \sigma \,  W^{-1}(\sigma) \begin{pmatrix} F_1^{(n)}(\sigma) \\
F_2^{(n)}(\sigma) \end{pmatrix}
\end{equation}
where $\bar{\varphi}^{(n)}$ and $\bar{I}^{(n)}$ are the $n^{\rm th}$ order in
the $\varepsilon$-expansion of the initial conditions for $\varphi$ and $I$, respectively.
In the last term of \eqref{zEqn} we have
\begin{equation} \nonumber
W(\tau) \int_0^{\tau}  \der \sigma \,
W^{-1}(\sigma) \begin{pmatrix} F_1^{(n)}(\sigma) \\ F_2^{(n)}(\sigma) \end{pmatrix}
= \int_0^\tau \der \sigma \,
W(\tau - \sigma) \begin{pmatrix} F_1^{(n)}(\sigma) \\ F_2^{(n)}(\sigma) \end{pmatrix} .
\end{equation}
This yields
\begin{equation}\label{phinEqnLib}
\begin{split}
\varphi^{(n)}(\tau) & =
\bar{\varphi}^{(n)} + \zeta^{(0)}\tau \bar{I}^{(n)}  + \int_0^{\tau} \der \sigma \, F_1^{(n)}(\sigma)
+ \zeta^{(0)} \int_0^{\tau} \der\sigma \int_0^{\sigma} \der\sigma' \, F_2^{(n)}(\sigma') , \\
I^{(n)}(\tau) & = \bar{I}^{(n)}(\tau) + \int_0^{\tau} \der\sigma \, F_2^{(n)}(\sigma) .
\end{split}
\end{equation}

For a periodic function $g$, let us denote the average of $g$ with $\langle g \rangle$ and the
zero-average function $g - \langle g \rangle$ with $\breve{g}$. Suppose that
\begin{equation}\label{F2zeroAv}
\langle F_2^{(n)} \rangle := \frac{1}{T}\int_0^T \der\tau \, F_2^{(n)}(\tau) = 0,
\end{equation}
where $T=4{\bf K}(k_{1})\mathpzc{p}$; we will check later on the validity of \eqref{F2zeroAv}.
Then we may write
\begin{equation} \nonumber \begin{split}
\mathscr{F}_1^{(n)}(\tau) & = \int_0^\tau \der \sigma \, F_1^{(n)}(\sigma)
= \langle F_1^{(n)} \rangle \tau + \int_0^\tau \der \sigma \, \breve{F}_1^{(n)}(\sigma) , \\
\mathscr{F}_2^{(n)}(\tau) & = \int_0^\tau F_2^{(n)}(\sigma) \thinspace \der \sigma =
\int_0^\tau \der \sigma \, \breve{F}_2^{(n)}(\sigma) ,
\end{split}
\end{equation}
and subsequently rewrite \eqref{phinEqnLib} as
\begin{equation} \nonumber \begin{split}
\varphi^{(n)}(\tau) & = \bar{\varphi}^{(n)} + \zeta^{(0)}\tau \bar{I}^{(n)} +
\langle F_1^{(n)} \rangle \, \tau +
\int_0^\tau \der \sigma \, \breve{F}_1^{(n)}(\sigma)
+ \zeta^{(0)} \langle \mathscr{F}_2^{(n)} \rangle \, \tau + \zeta^{(0)} \int_0^\tau
\der \sigma \, \mathscr{F}_2^{(n)}(\sigma) , \\
I^{(n)}(\tau) & = \bar{I}^{(n)} + \int_0^\tau \der \sigma \, \breve{F}_2^{(n)}(\sigma) ,
\end{split}
\end{equation}
in which all the terms which are not linear in $\tau$ are periodic.
If we choose our initial conditions $\bar{I}^{(n)}$ such that they satisfy
\[\bar{I}^{(n)} = -\frac{1}{\zeta^{(0)}} \langle F_1^{(n)} \rangle -
\langle \mathscr{F}_2^{(n)} \rangle , \]
the above reduces to
\begin{equation} \nonumber \begin{split}
\varphi^{(n)}(\tau) & = \bar{\varphi}^{(n)}(\tau) + \int_0^\tau \der\sigma \,
\breve{F}_1^{(n)}(\sigma)  + \zeta^{(0)} \int_0^\tau \der\sigma \,
\mathscr{F}_2^{(n)}(\sigma) , \\
I^{(n)}(\tau) & = \bar{I}^{(n)} + \int_0^\tau \der\sigma \, \breve{F}_2^{(n)}(\tilde{\tau}) ,
\end{split}
\end{equation}
so that both $\varphi^{(n)}(\tau)$ and $I^{(n)}(\tau)$ are periodic functions with period $T$,
provided \eqref{F2zeroAv} holds.

%%%%%%%%%%%%%%%%%%%%%%%%%%%%%%%%%%%%%%%%%%%%%%%%%%%%%%%%%%%%%%%%
\begin{lemma} \label{lem:1}
Consider the series \eqref{perturbedsolution}. If $\mathpzc{p/q} = 1/2m$, $m \in \mathbb{N}$
and $C_{1}$ is small enough, then it is possible to fix the initial conditions
$(\bar{\varphi}^{(n)},\bar{I}^{(n)})$ in such a way that \eqref{F2zeroAv}
holds for all $n \ge 1$. If $\mathpzc{p/q} \neq 1/2m$
for all $m\in\mathbb{N}$, then \eqref{F2zeroAv} can be satisfied only for $C_{1} = 0$.
\end{lemma}
%%%%%%%%%%%%%%%%%%%%%%%%%%%%%%%%%%%%%%%%%%%%%%%%%%%%%%%%%%%%%%%%

%%%%%%%%%%%%%%%%%%%%%%%%%%%%%%%%%%%%%%%%%%%%%%%%%%%%%%%%%%%%%%%%
\begin{proof}
For $n = 1$ we have
\begin{equation} \nonumber
\begin{split}
F_2^{(1)}(\tau) & = \frac{8 k_1^2 {\bf K}(k_1)}{\pi} \cos(\tau - \tau_0)
\sn(\sqrt{\alpha}\tau,k_{1})\cn(\sqrt{\alpha}\tau,k_{1}) \, \dn(\sqrt{\alpha}\tau,k_{1}) \\
& - \frac{8 C_{1} k_1^2\sqrt{\alpha} \, {\bf K}(k_1)}{\pi}\cn^2(\sqrt{\alpha}\tau,k_{1}) ,
\end{split}
\end{equation}
with $k_{1}=k_{1}^{(0)}$ here and henceforth. Moreover set,
see Appendix \ref{AppUsefulIntsAndExps},
\begin{equation} \label{Delta}
\begin{split}
\Delta & := \frac{\sqrt{\alpha}}{4{\bf K}(k_1)} \int_0^{4{\bf K}(k_1)/\sqrt{\alpha}}
\der\tau \, \cn^2(\sqrt{\alpha}\tau , k_{1}) \\
&  = \frac{1}{2{\bf K}(k_1)} \int_0^{2{\bf K}(k_1)} \der\tau \,
\cn^2 (\tau , k_{1}) = \frac{1}{k_1^2}\left[\frac{1}{2{\bf K}(k_1)}{\bf E}
\left( 2{\bf K}(k_1),k_1\right) - (1 - k_1^2)\right],
\end{split}
\end{equation}
where ${\bf E}(u,k)$ is the incomplete elliptic integral of the second kind, and
$\Gamma_1 (\tau_0;\mathpzc{p},\mathpzc{q}) :=
\sin(\tau_0)\,G_1(\mathpzc{p},\mathpzc{q})$, with
\begin{equation}\label{GammaEqnLib}
\begin{split}
G_1(\mathpzc{p},\mathpzc{q}) & = \frac{1}{T}
\int_0^{T}\sn(\sqrt{\alpha}\tau,k_1)\cn(\sqrt{\alpha}\tau,k_1)\dn(\sqrt{\alpha}\tau,k_1)\sin(\tau) \\
& = \frac{1}{4 {\bf K}(k_1)\mathpzc{p}}\int_0^{4 {\bf K}(k_1)\mathpzc{p}}
\der\tau \, \sn({\tau},k_1)\cn({\tau},k_1)\dn({\tau},k_1)\sin({\tau}/\sqrt{\alpha}) .
\end{split}
\end{equation}
Under the resonance condition $\pi \alpha/2{\bf K}(k_{1})=\mathpzc{p/q}$, one has
\begin{equation} \nonumber
\sin(\tau /\sqrt{\alpha}) = \sin\left(\frac{\pi\tau}{2{\bf K}(k_1)}
\frac{\mathpzc{q}}{\mathpzc{p}}\right),
\end{equation}
where $\mathpzc{p}$ and $\mathpzc{q}$ are relatively prime.
By expanding the Jacobi elliptic functions in Fourier series,
see Appendix \ref{AppUsefulIntsAndExps},
we find that $\mathpzc{p},\mathpzc{q}$ must also satisfy the condition
\begin{equation} \nonumber
\mathpzc{p}\Bigl(\pm (2m_1 - 1)\pm (2m_2 - 1)\pm 2m_3\Bigr)\pm \mathpzc{q} = 0
\end{equation}
for $G_{1}(\mathpzc{p},\mathpzc{q})$ to be non-zero.
Thus $\mathpzc{q} = 2m\mathpzc{p}$, $m \in \mathbb{N}$, that is
$\mathpzc{p}=1$ and $\mathpzc{q}\in 2\mathbb{N}$, and
$\langle F_2^{(1)} \rangle = 0$ provided $C_{1}$ and $\tau_{0}$ satisfy
\begin{equation}\label{CEqnLib}
C_{1} = \frac{\sin(\tau_0)}{\sqrt{\alpha}\Delta}G_1(\mathpzc{p},\mathpzc{q}) .
\end{equation}
Note that the existence of a value of $\tau_0$ satisfying \eqref{CEqnLib} is possible only if
\begin{equation} \nonumber
|C_{1}| \le C_{1} (\mathpzc{p}/\mathpzc{q})
:= \frac{1}{\sqrt{\alpha}\Delta} \, G_1(\mathpzc{p},\mathpzc{q}).
\end{equation}
Some values of the constants $C_{1} (\mathpzc{p}/\mathpzc{q})$
for $\alpha=0.5$ are listed in Table \ref{ThresholdTableLib}.

%%%%%%%%%%%%%%%%%%%%%%%%%%%%%%%%%%%%%%%%%%%%%%%%%%%%%%%%%%%%%%%%
\begin{table}[H]
\centering
\begin{tabular}{|ccccc|}
\hline
$\mathpzc{q}$ & $k_1$ &
$G_1(1/\mathpzc{q})$ & $\Delta$  & $C_1(1/\mathpzc{q})$ \\ \hline
  2 & 0.885201568846 & 0.172135 & 0.407121 & 0.597944 \\
  4 & 0.998888384493 & 0.077675 & 0.224342 & 0.489649 \\
  6 & 0.999986981343 & 0.051734 & 0.150043 & 0.487616 \\
  8 & 0.999999846887 & 0.038800 & 0.112539 & 0.487578 \\
10 & 0.999999998199 & 0.031040 & 0.090032 & 0.487577 \\
12 & 0.999999999979 & 0.025867 & 0.075026 & 0.487577 \\
\hline
\end{tabular}
\caption{\small{Constants for the oscillating attractors with $\alpha = 0.5$.}}
\label{ThresholdTableLib}
\end{table}
%%%%%%%%%%%%%%%%%%%%%%%%%%%%%%%%%%%%%%%%%%%%%%%%%%%%%%%%%%%%%%%%

For all $n \ge 2$ we can write $F^{(n)}_2(\tau)$ as
\begin{equation} \nonumber
\begin{split}
F^{(n)}_2(\tau) & = \frac{8 k_1^2 {\bf K}(k_{1})}{\pi}
\left.\frac{\partial}{\partial \varphi} \left( \cos(\tau - \tau_0)
\, \sn(\cdot)\cn(\cdot)\dn(\cdot) - \sqrt{\alpha} \, C_{1} \cn^2(\cdot) \right)
\right|_{\varphi = \varphi^{(0)}}
\!\!\!\!\!\!\! \bar{\varphi}^{(n-1)}  + R^{(n)}(\tau),
\end{split}
\end{equation}
where $R^{(n)}(\tau)$ is a suitable function which does not depend on $\bar{\varphi}^{(n-1)}$.
It can be seen that $\langle F_2^{(n)} \rangle = 0$ if and only if
\begin{equation} \nonumber
\begin{split}
\langle R^{(n)} \rangle  & = - \frac{8 k_1^2 {\bf K}(k_1)}{\pi}
\left( \frac{1}{T}\int_0^T \der\tau \, \frac{2{\bf K}(k_1)}{\sqrt{\alpha}\pi}\frac{\partial}{\partial \tau}
\Bigl(\sn(\sqrt{\alpha}\tau)\cn(\sqrt{\alpha}\tau)\dn(\sqrt{\alpha}\tau)\Bigr)
\cos(\tau - \tau_0) \right. \\
& \qquad \qquad \left. - \frac{\sqrt{\alpha}\, C_{1}}{T}
\int_0^T \der\tau \, \frac{2{\bf K}(k_1)}{\sqrt{\alpha}\pi}\frac{\partial}{\partial \tau}
\Bigl(\cn^2(\sqrt{\alpha}\tau)\Bigr) \right)\bar{\varphi}^{(n-1)}.
\end{split}
\end{equation}
This can be rewritten as
\begin{equation} \nonumber
\langle R^{(n)} \rangle = - \frac{16 k_1^2 {\bf K}^2(k_1)}{\sqrt{\alpha}\pi^2}
\cos(\tau_0) \, G_1(\mathpzc{p},\mathpzc{q})\bar{\varphi}^{(n-1)}.
\end{equation}
We refer the reader to Appendix \ref{AppUsefulIntsAndExps}
for more details on the evaluation of the integrals.
The coefficient of $\bar{\varphi}^{(n-1)}$ is non-vanishing for $\tau_0$ chosen
such that \eqref{CEqnLib} is satisfied. Therefore it is possible
to fix the initial conditions $\bar{\varphi}^{(n-1)}$
in such a way that one has $\langle F^{(n)}_2 \rangle= 0$ at all orders,
thus completing the proof of the lemma. $\hfill \Box$
\end{proof}

Lemma \ref{lem:1} implies that the threshold values of the $\mathpzc{p}\,$:$\,\mathpzc{q}$
resonances are $\gamma(\mathpzc{p}/\mathpzc{q})=C_{1}(\mathpzc{p}/\mathpzc{q})
\varepsilon$ for $\mathpzc{p}=1$ and $\mathpzc{q}$ even,
with the constants $C_{1}(\mathpzc{p}/\mathpzc{q})$ in Table \ref{ThresholdTableLib},
while the threshold values of the other resonances are at least $O(\varepsilon^2)$.

%%%%%%%%%%%%%%%%%%%%%%%%%%%%%%%%%%%%%%%%%%%%%%%%%%%%%%%%%%%%%%%%
\subsection{Rotations}
%%%%%%%%%%%%%%%%%%%%%%%%%%%%%%%%%%%%%%%%%%%%%%%%%%%%%%%%%%%%%%%%

Similarly for the rotating scenario, again further details can be found in Appendix \ref{AppPerturbedSys},
the perturbed system can be written as
\begin{equation}\label{PerturbedAAvariablesRotTau}
\begin{split}
\varphi' & = \frac{\pi\sqrt{\alpha}}{k_2{\bf K}(k_2)} +
\frac{\varepsilon \pi k_2}{\sqrt{\alpha}\,{\bf K}(k_2)}
\left[\frac{k^2_2 \sn^2(\cdot)\cn^2(\cdot)}{1 - k_2^2} -
\frac{{\bf Z}(\cdot)\sn(\cdot)\cn(\cdot)\dn(\cdot)}{1 - k_2^2} \right]\cos(\tau - \tau_0) \\
& \qquad \qquad \quad - \frac{C_{1} \varepsilon \pi}{{\bf K}(k_2)}
\left[ \frac{k_2^2 \sn(\cdot) \cn(\cdot)\dn(\cdot)}{1 - k_2^2} -
\frac{ {\bf Z}(\cdot)\dn^2(\cdot)}{1-k_2^2} \right] , \\
I' & = \frac{4 \varepsilon {\bf K}(k_2)}{\pi}\cos(\tau - \tau_0) \sn(\cdot)\cn(\cdot)\dn(\cdot) -
\frac{4 C_{1} \varepsilon \, \sqrt{\alpha} \, {\bf K}(k_2)}{\pi k_2}\dn^2(\cdot) .
\end{split}
\end{equation}
Here and throughout this subsection we set $(\cdot) = \left(\frac{{\bf K}(k_2)}{\pi}
\varphi,k_2\right)$. In this scenario, the Wronskian matrix $W(\tau)$ can be written as
in \eqref{wronskian}, with $\zeta^{(0)}$ given by, see Appendix \ref{ActionAngleApp},
\begin{equation} \label{OmegaDerRot}
\zeta (I)= -\frac{\pi^2\sqrt{\alpha}}{4{\bf K}^3\left(k_2\right)}
\left[\frac{{\bf E}\left(k_2\right)}{1 - k_2^2} - {\bf K}\left(k_2\right)\right] ,
\end{equation}
for $k_{2}=k_{2}^{(0)}$.
We again look for solutions $(\varphi(\tau), I(\tau) )$ with period
$T= 2\pi \mathpzc{q}=4k_2{\bf K}(k_{1})\mathpzc{p}/\sqrt{\alpha}$ corresponding to a resonance
$\mathpzc{p}\,$:$\,\mathpzc{q}$, of the form \eqref{perturbedsolution},
the functions $\varphi^{(n)}(\tau)$ and $I^{(n)}(\tau)$ being defined
as in \eqref{zEqn}, with $F_1^{(n)}(\tau)$ and $F_2^{(n)}(\tau)$ defined as
\begin{equation} \nonumber \begin{split}
F_1^{(n)}(\tau) & = \left[\frac{\pi\sqrt{\alpha}}{k_2{\bf K}(k_2)} - \zeta^{(0)} I \right]^{(n)} +
\Biggl[ \frac{\pi k_2}{\sqrt{\alpha}\,{\bf K}(k_2)}
\left[\frac{k_2^2 \sn^2(\cdot)\cn^2(\cdot)}{1 - k_2^2}  \right. \\
& \qquad \qquad \left.  -
\frac{{\bf Z}(\cdot)\sn(\cdot)\cn(\cdot)\dn(\cdot)}{1 - k_2^2} \right]\cos(\tau - \tau_0) \\
& \qquad \qquad - \frac{C_{1} \pi}{{\bf K}(k_2)} \left[
\frac{k_2^2 \sn(\cdot) \cn(\cdot)}{1 - k_2^2} -
\frac{ {\bf Z}(\cdot)\dn(\cdot)}{1-k_2^2} \right]\Biggr]^{(n-1)}, \\
F_2^{(n)}(\tau) & = \left[\frac{4 {\bf K}(k_2)}{\pi}\cos(\tau - \tau_0) \right.
\left. \sn(\cdot)\cn(\cdot)\dn(\cdot) - \frac{4C_{1} \, \sqrt{\alpha} \, {\bf K}(k_2)}{\pi k_2}\dn^2(\cdot)\right]^{(n-1)}.
\end{split}
\end{equation}
The theory goes through exactly as previously shown for the case of libration and we must show
that $\langle F_2^{(n)} \rangle = 0$.

%%%%%%%%%%%%%%%%%%%%%%%%%%%%%%%%%%%%%%%%%%%%%%%%%%%%%%%%%%%%%%%%
\begin{lemma} \label{lem:2}
Consider the series \eqref{perturbedsolution}.
If $\mathpzc{p}/\mathpzc{q} = 1/2m$, $m \in \mathbb{N}$, and $C_{1}$ is small enough,
then it is possible to fix the initial conditions $(\bar{\varphi}^{(n)},\bar{I}^{(n)})$
in such a way that $\langle F_2^{(n)} \rangle = 0$ for all $n \ge 1$.
If $\mathpzc{p}/\mathpzc{q} \neq 1/2m$ for all $m\in\mathbb{N}$
then $\langle F_2^{(n)} \rangle = 0$ only when $C_{1} = 0$.
\end{lemma}
%%%%%%%%%%%%%%%%%%%%%%%%%%%%%%%%%%%%%%%%%%%%%%%%%%%%%%%%%%%%%%%%

\begin{proof}
One has
\begin{equation} \nonumber \begin{split}
F_2^{(1)}(\tau) & = \frac{4 {\bf K}(k_2)}{\pi}\cos{(\tau - \tau_0)}
\sn\left(\frac{\sqrt{\alpha}}{k_2}\tau , k_2 \right)\cn\left(\frac{\sqrt{\alpha}}{k_2}\tau , k_{2} \right)
\dn\left(\frac{\sqrt{\alpha}}{k_2}\tau , k_2 \right) \\ & \qquad \qquad -
\frac{4C_{1} \, \sqrt{\alpha} {\bf K}(k_2)}{\pi k_2}
\dn^2\left(\frac{\sqrt{\alpha}}{k_2}\tau , k_{2} \right) ,
\end{split}
\end{equation}
with $k_{2}=k_{2}^{(0)}$ here and henceforth. Define,
see Appendix \ref{AppUsefulIntsAndExps},
\begin{equation} \nonumber
\Delta := \frac{\sqrt{\alpha}}{2k_2{\bf K}(k_2)}
\int_0^{2k_2{\bf K}(k_2)/\sqrt{\alpha}} \der\tau \, \dn^2\left(\frac{\sqrt{\alpha}}{k_2}\tau, k_2 \right) =
\frac{1}{2{\bf K}(k_2)}{\bf E}\Bigl( 2{\bf K}(k_2) ,k_2\Bigr) .
\end{equation}
and $\Gamma_{1}(\tau_0;\mathpzc{p},\mathpzc{q}) := \sin(\tau_0)\,
G_1(\mathpzc{p},\mathpzc{q})$, where
\begin{equation} \nonumber
\begin{split}
G_1(\mathpzc{p},\mathpzc{q}) & =
\frac{1}{T} \int_0^{T} \der\tau \, \sn\left(\frac{\sqrt{\alpha}}{k_2}\tau , k_2\right)\cn
\left(\frac{\sqrt{\alpha}}{k_2}\tau , k_2 \right) \dn\left(\frac{\sqrt{\alpha}}{k_2}\tau , k_2 \right)
\sin(\tau)  \\ & = \frac{1}{4 {\bf K}(k_2)\mathpzc{p}}
\int_0^{4 {\bf K}(k_2)\mathpzc{p}} \der\tau \,
\sn(\tau , k_2 )\cn(\tau, k_2)\dn(\tau, k_2)\sin(k_2\tau/\sqrt{\alpha}) ,
\end{split}
\end{equation}
then use the resonance condition to set
\[\sin\left(\frac{k_2 \tau}{\sqrt{\alpha}}\right) =
\sin\left(\frac{\pi\tau}{2{\bf K}(k_2)}\frac{\mathpzc{q}}{\mathpzc{p}}\right).\]
On inspection of the above we see that $\Gamma_{1}(\tau_0;\mathpzc{p},\mathpzc{q})$
can be calculated similarly to the case inside the separatrix.
It follows that the same applies and $\mathpzc{p}/\mathpzc{q} = 1/2m$ for $m \in \mathbb{N}$.
Then $\langle F_2^{(1)} \rangle = 0$ if
\begin{equation}\label{CEqnRot}
C_{1} = \frac{k_2 \sin(\tau_0)}{\sqrt{\alpha}\Delta}G_1(\mathpzc{p},\mathpzc{q}),
\end{equation}
which requires
\begin{equation} \nonumber
|C_{1}| \le C_1(\mathpzc{p}/\mathpzc{q}) := \frac{k_{2}}{\sqrt{\alpha}\Delta} \, G_1(\mathpzc{p},\mathpzc{q}).
\end{equation}
Some values of the constants $C_{1} (\mathpzc{p}/\mathpzc{q})$
for $\alpha=0.5$ are listed in Table \ref{C0valuesRotations}.

%%%%%%%%%%%%%%%%%%%%%%%%%%%%%%%%%%%%%%%%%%%%%%%%%%%%%%%%%%%%%%%%
\begin{table}[H]
\centering
\begin{tabular}{|ccccc|}
\hline
$\mathpzc{q}$ & $k_2$ & $G_1(1/\mathpzc{q})$ & $\Delta$  & $C_1(1/\mathpzc{q})$\\
\hline
  2 & 0.924397052341 & 0.156774 & 0.474414 & 0.432005 \\
  4 & 0.998899257272 & 0.077612 & 0.225808 & 0.485542 \\
  6 & 0.999986983601 & 0.051734 & 0.150063 & 0.487439 \\
  8 & 0.999999846887 & 0.038800 & 0.112540 & 0.487577 \\
10 & 0.999999998199 & 0.031040 & 0.090032 & 0.487577 \\
12 & 0.999999999978 & 0.025867 & 0.075026 & 0.487577 \\ \hline
\end{tabular}
\caption{\small{Constants for the rotating attractors with $\alpha = 0.5$.}}
\label{C0valuesRotations}
\end{table}
%%%%%%%%%%%%%%%%%%%%%%%%%%%%%%%%%%%%%%%%%%%%%%%%%%%%%%%%%%%%%%%%

For $n \ge 2$ one has
\begin{equation} \nonumber \begin{split}
F_2^{(n)}(\tau) = \frac{4 {\bf K}(k_2)}{\pi}
\left. \frac{\partial}{\partial \varphi} \left( \cos(\tau - \tau_0) \,
\sn(\cdot)\cn(\cdot)\dn(\cdot) - \sqrt{\alpha} \, C_{1} \,
\dn^2(\cdot) \right) \right|_{\varphi = \varphi^{(0)}}
\!\!\!\!\!\!\! \bar{\varphi}^{(n-1)}  + R^{(n)}(\tau),
\end{split}
\end{equation}
where again $R^{(n)}(\tau)$ will be a suitable function
which does not depend on $\bar{\varphi}^{(n-1)}$.
Similarly to the case of libration, $\langle F_2^{(n)} \rangle = 0$ if and only if
\begin{equation} \nonumber
\begin{split}
\langle & R^{(n)} \rangle = - \frac{4 {\bf K}(k_2)}{\pi}
\left( \frac{1}{T}\int_0^T \frac{k_2{\bf K}(k_1)}{\sqrt{\alpha}\pi}
\frac{\partial}{\partial \tau}\Bigl(\sn(\sqrt{\alpha}\tau)\cn(\sqrt{\alpha}\tau)
\dn(\sqrt{\alpha}\tau)\Bigr)\cos(\tau - \tau_0)\thinspace \der \tau \right. \\
& \qquad \left. - \frac{\sqrt{\alpha}\,C_{1}}{T}
\int_0^T \frac{k_2{\bf K}(k_1)}{\sqrt{\alpha}\pi}
\frac{\partial}{\partial \tau}\Bigl(\dn^2(\sqrt{\alpha}\tau)\Bigr)
\thinspace \der \tau \right)\bar{\varphi}^{(n-1)} =
- \frac{4 k_2{\bf K}^2(k_2)}{\sqrt{\alpha}\pi^2}\cos(\tau_0) \,
G_1(\mathpzc{p},\mathpzc{q})\bar{\varphi}^{(n-1)} ,
\end{split}
\end{equation}
so that the coefficient of $\bar{\varphi}^{(n-1)}$ turns out to be
non-vanishing for $\tau_0$ chosen such that equation \eqref{CEqnRot} is satisfied.
Therefore it is possible to fix the initial conditions $\bar{\varphi}^{(n-1)}$
in such a way that one has $\langle F^{(n)}_2(\tau) \rangle = 0$ at all orders,
thus completing the proof. $\hfill\Box$
\end{proof}

Lemma \ref{lem:2} implies that the threshold values of the $\mathpzc{p}\,$:$\,\mathpzc{q}$
resonances are $\gamma(\mathpzc{p}/\mathpzc{q})=C_{1}(\mathpzc{p}/\mathpzc{q})
\varepsilon$ for $\mathpzc{p}=1$ and $\mathpzc{q}$ even,
with the constants $C_{1}(\mathpzc{p}/\mathpzc{q})$ in Table \ref{C0valuesRotations},
while the threshold values of the other resonances are at least $O(\varepsilon^2)$.

\vspace{.3truecm}

Note, in Tables \ref{ThresholdTableLib} and \ref{C0valuesRotations} it is apparent that,
for $\alpha=0.5$, increasing $\mathpzc{q}$ causes the value of $C_1(1/\mathpzc{q})$ to
converge to 0.487577 in both cases. However this does not mean that for
$\gamma < 0.487577\,\varepsilon$ there are
infinitely many attracting solutions with increasing period:
this would be a counter-example to Palis' conjecture!
The explanation for this seeming paradox is as follows:
As $\mathpzc{q}$ increases the solutions move closer and closer to the separatrix
(this can be seen by the corresponding values of $k_1$ and $k_2$),
where the power series expansions \eqref{perturbedsolution}
for the solutions $I(\tau)$ and $\varphi(\tau)$ which were constructed with perturbation theory
converge only for very small values of $\varepsilon$: the larger $\mathpzc{q}$,
the smaller must be $\varepsilon$. In particular, for any fixed $\varepsilon$ there
is only a finite number of periodic solutions which can be studied by perturbation theory.
In particular, for the chosen parameters the only periodic solution
corresponds to the resonance $1$:$2$ inside the separatrix.
We also note that the above analysis
applies only to periodic attractors with $\mathpzc{p}=1$ and $\mathpzc{q}$ even.
However we shall see that the numerical simulations provide also rotating
attractors with period $2\pi$, that is the same period as the forcing:
we expect that continuing the analysis to second order and writing
$\gamma=C_{2}\varepsilon^{2}$, see \cite{POaSL},
would give the threshold values for these periodic attractors.

%%%%%%%%%%%%%%%%%%%%%%%%%%%%%%%%%%%%%%%%%%%%%%%%%%%%%%%%%%%%%%%%
\section{Numerics for the downwards pendulum}
\label{PenNumericSect}
%%%%%%%%%%%%%%%%%%%%%%%%%%%%%%%%%%%%%%%%%%%%%%%%%%%%%%%%%%%%%%%%

We shall investigate the system \eqref{PenEqnTau} in the same region of
phase space used in \cite{DPP}, namely $\theta \in [-\pi, \pi]$, $\theta' \in [-4, 4]$
and calculate the relative areas of the basins of attraction, that is the percentage
of phase space they cover relative to this region. 

Throughout this Section
we fix the parameters $\alpha = 0.5$ and $\beta = 0.1$. These parameter values,
also investigated in \cite{DPP}, correspond to a stable region of the stability tongues
for the system linearised around $\theta=0$, see \cite{ODEBook},
so that the downwards configuration is stable.
The chosen values for the damping coefficient span values between
$\gamma = 0.002$ and $0.06$, of which only $\gamma = 0.03$
was previously investigated in \cite{DPP}. For some values of $\gamma$,
the system exhibits three non-fixed-point attractors, examples of which are shown in
Figure \ref{ConstantAttractorsDown}, as well as the downwards fixed point attractor.
Here and henceforth, for brevity, we shall say that a solution has period $\mathpzc{n}$
if it comes back to its initial value after $\mathpzc{n}$ periods of the forcing.
Of course the upwards fixed point also exists as a solution to the system,
however it is unstable and thus does not attract any non-zero measure subset
of phase space. It can also be seen from Figure \ref{ConstantAttractorsDown} that the
attractive solutions are near the separatrix of the unperturbed system:
this is evident as the curves described by the two rotating attractors are close to that
of the oscillating attractor and the separatrix lies between them. This
observation confirms the reasoning as to why the computation of the
threshold values can only produce valid results for the period 2 oscillating attractor
(see the conclusive remarks in Section \ref{ThresholdsSect}).

%%%%%%%%%%%%%%%%%%%%%%%%%%%%%%%%%%%%%%%%%%%%%%%%%%%%%%%%%%%%%%%%
\begin{figure}[H]
\centering
\includegraphics[width=0.5\textwidth]{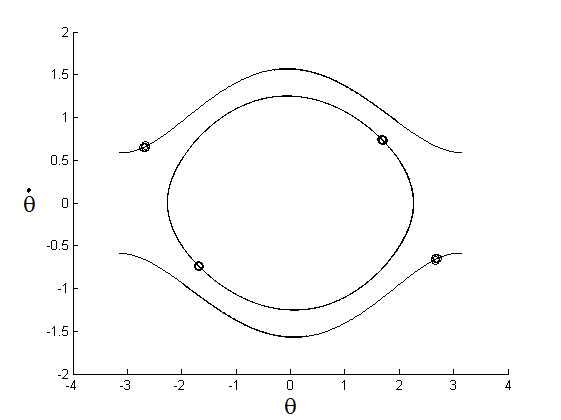}
\caption{\small{Attracting solutions for the system \eqref{PenEqnTau}
with $\alpha=0.5$, $\beta=0.1$ and $\gamma=0.02$,
namely two period 1 rotations and one period 2 solution which oscillates
about the downwards fixed point. Periods can be deduced
by circles corresponding to the Poincar\'e map.}}
\label{ConstantAttractorsDown}
\end{figure}
%%%%%%%%%%%%%%%%%%%%%%%%%%%%%%%%%%%%%%%%%%%%%%%%%%%%%%%%%%%%%%%%

For $\gamma<\bar\gamma_0$, for a suitable $\bar\gamma_0\in(0.4,0.5)$,
the system has three periodic attractors,
in addition to the downwards fixed point: one oscillating
and two rotating attractors. For $\gamma \ge \bar\gamma_{0}$ the two rotating attractors
no longer exist, leaving just the oscillatory attractor and the fixed point.
The basins of attraction for $\gamma = 0.02$,
$0.03$, $0.04$ and $0.05$ are shown in Figure \ref{BasinsConstantDown},
from which we can see that the entire phase space is covered: this suggests that
no other attractors exist, at least for the values of the parameters considered.
The corresponding relative areas, as estimated by the numerical simulations,
are given in Table \ref{TAPCFD} and plotted in Figure \ref{FAPCFD}.
The relative areas of the basins of attraction for positive and negative rotations
have been listed in the same column: numerical simulations found a difference in size
no greater than $10^{-2}\%$ and, due to the symmetries of the system, it is expected
that this difference is numerically induced by the selection of initial conditions
The basins of attraction were estimated using numerical
simulations with 600 000 random initial conditions in phase space.
More notes on the numerics used can be found in Section \ref{NumericsSection}.

It can be seen that the results in Table \ref{TAPCFD} are in agreement
with the calculations for the threshold value for the period 2 oscillatory attractor.
The calculations in Appendix \ref{LinearisedApp} predict that,
for the chosen values $\alpha = 0.5$ and $\beta = 0.1$,
taking $\gamma > \bar{\gamma}_{1}\approx 0.1021$
ensures for the origin to capture a full measure set of initial conditions.
From Table \ref{TAPCFD} a stronger result emerges numerically: 
the fixed point attracts the full phase space, up to a zero-measure set,
for $\gamma \ge \bar\gamma_{2} \approx 0.06$.
Upon comparing results in Table \ref{TAPCFD}, we see that,
essentially, the basin of attraction for the fixed point becomes smaller with an increase
in $\gamma$, up to approximately $0.035$, after which it grows again.
Similarly, by increasing the value of $\gamma$, the basins of attraction
of the oscillating and rotating solutions attractors increase initially, up to some value
(about $0.025$ and $0.035$, respectively), after which they become smaller.
Furthermore, the variations of the relative areas of the basins of attraction are
never monotonic, as one observes slight oscillations for small variations of $\gamma$.
These features seem contrary to systems such as the cubic oscillator,
where decreasing dissipation seems to cause the relative area
of the basin of attraction of the fixed point to decrease monotonically, while
the basins of attraction of the periodic attractors reach a maximum value, after
which their relative areas slightly decrease, see for instance Table III in \cite{CP}.
We note, however, that a more detailed investigation shows that
oscillations occur also in the case of the cubic oscillator. This was already observed
for some values of the parameters (see Table IX in \cite{CP}), but the phenomenon
can also be observed for the parameter values of Table III, simply by considering
smaller changes of the value of $\gamma$ with respect to the values in \cite{CP}.
For instance, by varying slightly $\gamma$
around $0.0005$ (see Table III in \cite{CP} for notations), one finds
for the main attractors the relative areas in Table \ref{table-cubic}.

%%%%%%%%%%%%%%%%%%%%%%%%%%%%%%%%%%%%%%%%%%%%%%%%%%%%%%%%%%%%%%%%
\begin{figure}[H]
\centering
\subfloat[]{\includegraphics[width=0.42\textwidth]{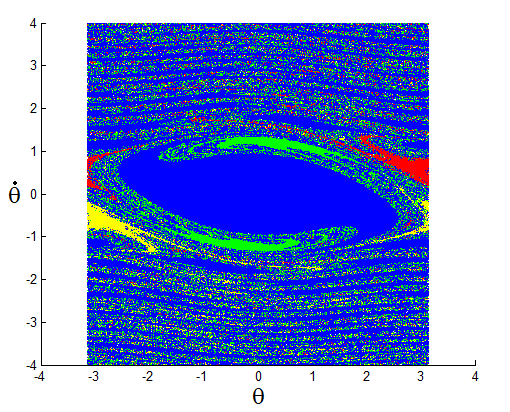}}
\subfloat[]{\includegraphics[width=0.42\textwidth]{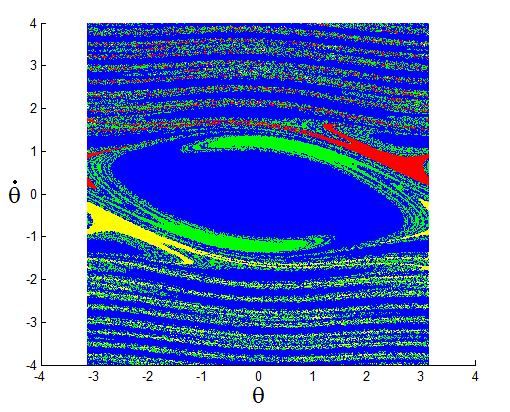}}\\
\subfloat[]{\includegraphics[width=0.42\textwidth]{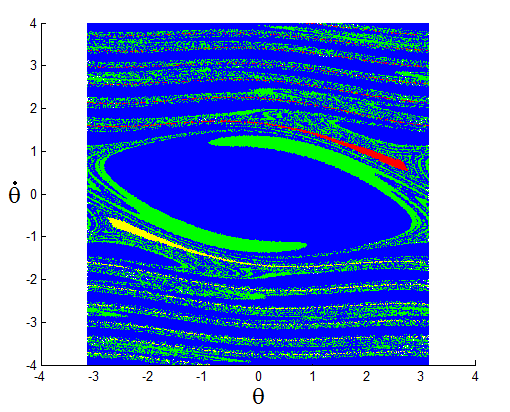}}
\subfloat[]{\includegraphics[width=0.42\textwidth]{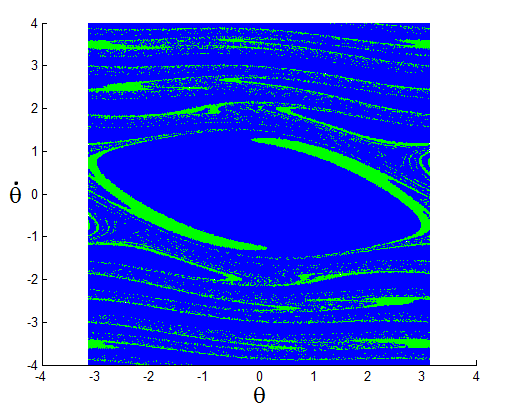}}
\caption{\small{Basins of attraction for the system \eqref{PenEqnTau}
with $\alpha=0.5$, $\beta=0.1$ and (a) $\gamma = 0.02$,
(b) $\gamma = 0.03$, (c) $\gamma = 0.04$ and (d) $\gamma = 0.05$.
The fixed point (FP) is shown in blue, the positive and negative rotating solutions
(PR and NR) are shown in red and yellow, respectively.
and the oscillating solution (OSC) in green.}}
\label{BasinsConstantDown}
\end{figure}
%%%%%%%%%%%%%%%%%%%%%%%%%%%%%%%%%%%%%%%%%%%%%%%%%%%%%%%%%%%%%%%%

In conclusion, for the pendulum, apart from small oscillations,
by decreasing the value of $\gamma$ from $0.06$ to $0.002$,
the basins of attraction of the periodic attractors,
after reaching a maximum value, becomes smaller. A similar phenomenon
occurs also in the cubic oscillator (albeit less pronounced).
However, a new feature of the pendulum,
with respect to the cubic oscillator, is that the basin of attraction
of the origin after reaching a minimum value increases again: the increase
seems to be too large to be ascribed simply to an oscillation, even though
this cannot be excluded. In the case of the cubic oscillator the slight decrease
of the sizes of the basins of attraction of the periodic attractors was due
essentially to the appearance of new attractors and their corresponding basins
of attraction. It would be interesting to investigate further, in the case of the pendulum,
how the basins of attractions,
in particular that of the fixed point,
change by taking smaller and smaller values of $\gamma$.
We intend to come back to this in the future \cite{AnalyticContPen}.

%%%%%%%%%%%%%%%%%%%%%%%%%%%%%%%%%%%%%%%%%%%%%%%%%%%%%%%%%%%%%%%%
\begin{figure}[H]
\begin{floatrow}
\capbtabbox{\small{
\begin{tabular}{cc|c|c|c|c|}
\cline{3-5}
& &\multicolumn{3}{c|}{\multirow{1}{*}{Basin of attraction \%}}\\
\cline{3-5}
& & FP & PR/NR & OSC\\
\hline
\multicolumn{1}{|c|}{\multirow{20}{*}{$\gamma$}}
& 0.0020 & 84.57 & 3.35 & 8.73 \\ \hhline{~----}
\multicolumn{1}{|c|}{} & 0.0050 & 79.91 & 3.88 & 12.32 \\ \hhline{~----}
\multicolumn{1}{|c|}{} & 0.0100 & 72.24 & 4.60 & 18.57 \\ \hhline{~----}
\multicolumn{1}{|c|}{} & 0.0200 & 71.95 & 4.57 & 18.90 \\ \hhline{~----}
\multicolumn{1}{|c|}{} & 0.0230 & 70.73 & 5.18 & 18.90  \\ \hhline{~----}
\multicolumn{1}{|c|}{} & 0.0250 & 69.28 & 5.19 & 20.35  \\ \hhline{~----}
\multicolumn{1}{|c|}{} & 0.0300 & 69.94 & 4.42 & 21.23  \\ \hhline{~----}
\multicolumn{1}{|c|}{} & 0.0330 & 68.92 & 3.75 & 23.59  \\ \hhline{~----}
\multicolumn{1}{|c|}{} & 0.0350 & 68.77 & 3.16 & 24.90 \\ \hhline{~----}
%\multicolumn{1}{|c|}{} & 0.0360 & \cellcolor{yellow} 69.77 & \cellcolor{yellow} 2.78 & \cellcolor{yellow} 24.66  \\ \hhline{~----}
%\multicolumn{1}{|c|}{} & 0.0380 & \cellcolor{yellow} 71.57 & \cellcolor{yellow} 2.17 & \cellcolor{yellow} 24.09 \\ \hhline{~----}
%\multicolumn{1}{|c|}{} & 0.0390 & \cellcolor{yellow} 73.44 & \cellcolor{yellow} 1.86 & \cellcolor{yellow} 22.83 \\ \hhline{~----}
\multicolumn{1}{|c|}{} & 0.0400 & 73.84 & 1.42 & 23.32  \\ \hhline{~----}
\multicolumn{1}{|c|}{} & 0.0500 & 85.61 & 0.00 & 14.39 \\ \hhline{~----}
%\multicolumn{1}{|c|}{} & 0.0530 & \cellcolor{yellow} 88.26 & \cellcolor{yellow} 0.00 & \cellcolor{yellow} 11.74 \\ \hhline{~----}
%\multicolumn{1}{|c|}{} & 0.0550 & \cellcolor{yellow} 91.22 & \cellcolor{yellow} 0.00 & \cellcolor{yellow} 8.78 \\ \hhline{~----}
%\multicolumn{1}{|c|}{} & 0.0570 & \cellcolor{yellow} 93.86 & \cellcolor{yellow} 0.00 & \cellcolor{yellow} 6.14 \\ \hhline{~----}
\multicolumn{1}{|c|}{} & 0.0590 & 96.96 & 0.00 & 3.04 \\ \hhline{~----}
\multicolumn{1}{|c|}{} & 0.0597 & 98.59 & 0.00 & 1.41 \\ \hhline{~----}
\multicolumn{1}{|c|}{} & 0.0600 & $\!\!\!100.00$ & 0.00 & 0.00 \\ \hline
\end{tabular}}}{
\caption{\small{Relative areas of the basins of \newline attraction with
$\alpha=0.5$, $\beta=0.1$ and constant $\gamma$. }}
\label{TAPCFD}}
\ffigbox{
\includegraphics[width=0.5\textwidth]{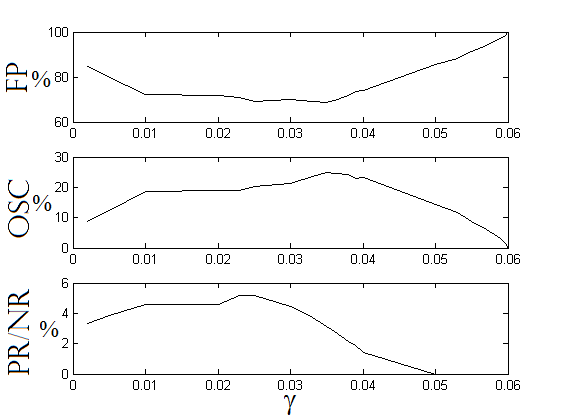}
}{
\caption{\small{Plot of the relative areas of the basins of attraction
with constant $\gamma$ as per Table \ref{TAPCFD}.}}
\label{FAPCFD}}
\end{floatrow}
\end{figure}
%%%%%%%%%%%%%%%%%%%%%%%%%%%%%%%%%%%%%%%%%%%%%%%%%%%%%%%%%%%%%%%%

%%%%%%%%%%%%%%%%%%%%%%%%%%%%%%%%%%%%%%%%%%%%%%%%%%%%%%%%%%%%%%%%
\begin{figure}[H]
\begin{floatrow}
\capbtabbox{\small{
\centering
\begin{tabular}{cc|c|c|c|c|c|c|c|c|}
\cline{3-10}
& &\multicolumn{8}{c|}{\multirow{1}{*}{Basin of attraction \%}}\\
\cline{3-10}
& & 0 & 1/2 & 1/4 & 1a & 1b & 1/6 & 3a & 3b \\
\hline
\multicolumn{1}{|c|}{\multirow{5}{*}{$\gamma$}}
& 0.00052 & 39.03 & 41.73 & 14.72 & 1.22 & 1.22 & 1.59 & 0.25 & 0.25 \\ \hhline{~---------}
\multicolumn{1}{|c|}{}
& 0.00051 & 39.73 & 41.70 & 13.88 & 1.24 & 1.24 & 1.66 & 0.27 & 0.27 \\ \hhline{~---------}
\multicolumn{1}{|c|}{}
& 0.00050 & 38.72 & 41.85 & 14.65 & 1.28 & 1.28 & 1.65 & 0.29 & 0.29 \\ \hhline{~---------}
\multicolumn{1}{|c|}{}
& 0.00049 & 39.26 & 41.96 & 13.81 & 1.29 & 1.29 & 1.77 & 0.27 & 0.32 \\ \hhline{~---------}
\multicolumn{1}{|c|}{}
& 0.00048 & 38.48 & 41.87 & 14.60 & 1.30 & 1.30 & 1.75 & 0.34 & 0.34 \\ \hline
\end{tabular}}}{
\caption{\small{Relative areas of the basins of attraction
of the main attractors for the system $\ddot x + (1+\varepsilon \cos t)x^{3}+
\gamma \dot x =0 $, with $\varepsilon=0.1$ and $\gamma$ around $0.0005$.
The basins of attraction were estimated using numerical simulations with
1000000 random initial conditions in the square $[-1,1]\times[-1,1]$
in phase space.}}
\label{table-cubic}}
\end{floatrow}
\end{figure}
%%%%%%%%%%%%%%%%%%%%%%%%%%%%%%%%%%%%%%%%%%%%%%%%%%%%%%%%%%%%%%%%

%%%%%%%%%%%%%%%%%%%%%%%%%%%%%%%%%%%%%%%%%%%%%%%%%%%%%%%%%%%%%%%%
\subsection{Increasing dissipation}
\label{IncDisSect}
%%%%%%%%%%%%%%%%%%%%%%%%%%%%%%%%%%%%%%%%%%%%%%%%%%%%%%%%%%%%%%%%

In this section we shall investigate the case where dissipation increases with time,
up to a time $T_0$, after which it remains constant. We will consider
a linear increase in dissipation from a value $\gamma_0$ at time $t = 0$ up to
$\gamma_1$ at time $T_0$, that is (see Figure \ref{FricIncFig})
\begin{equation}
\gamma =  \gamma(t) = \left\{ \begin{array}{ll}
\displaystyle{ \gamma_{0} + (\gamma_1 - \gamma_0) \frac{\tau}{T_0} } , &0 \leq \tau < T_0 , \\
\gamma_1 , & T_0 \leq \tau .
\end{array}\right.
\label{T0eqn}
\end{equation}
Although this is a greatly simplified model of what might take place in reality,
it serves the purpose of demonstrating the significant effects of initially
non-constant dissipation on the final basins of attraction.
Below we will consider explicitly the cases $\gamma_0=0.2$ and $\gamma_1=0.3$, $0.4$ and $0.5$.

As previously mentioned, we expect that increasing the value of $T_0$ results
in the relative areas of the basins of attraction moving along the curves plotted
for constant $\gamma$. The movement along these curves starts at $\gamma_1$
and goes towards $\gamma_0$. In particular, any values of the relative area
of a basin of attraction for constant values of the damping coefficient
between $\gamma_1$ and $\gamma_0$ are traced as the value of $T_0$
varies for time-dependent dissipation. This movement
along the curve is not linear with the value of $T_0$ but asymptotic, with
the relative area of the basin of attraction tending towards the value
at $\gamma=\gamma_0$ as $T_0 \to \infty$, providing the attractors existing
at $\gamma=\gamma_0$ also persist at $\gamma_1$. When the attractors which persist
at $\gamma=\gamma_1$ are a proper subset of those which exist at $\gamma_0$,
we expect the persisting attractors to absorb the remaining phase space left by the
attractors which have disappeared: thus their basins of attraction should
be greater than or equal to those at $\gamma=\gamma_0$.
For the values of the parameters in the chosen range,
only these two cases may occur as the solutions
which exist for $\gamma=\gamma_1$ also exist at $\gamma=\gamma_0$,
see Table \ref{TAPCFD}.

%%%%%%%%%%%%%%%%%%%%%%%%%%%%%%%%%%%%%%%%%%%%%%%%%%%%%%%%%%%%%%%%
\begin{figure}[H]
\centering
\includegraphics[width=0.34\textwidth]{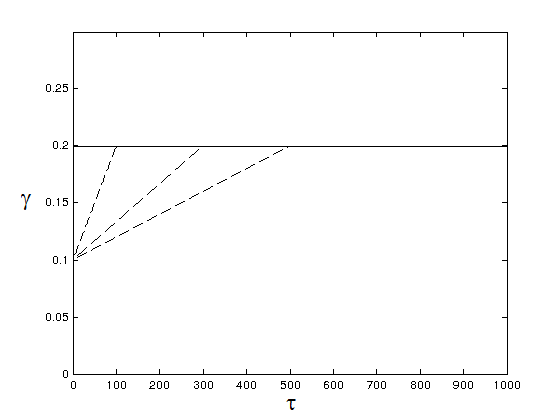}
\caption{\small{Plot of equation \eqref{T0eqn} with $\gamma_0 = 0.1$, $\gamma_1 = 0.2$
and varying $T_0$.}}
\label{FricIncFig}
\end{figure}
%%%%%%%%%%%%%%%%%%%%%%%%%%%%%%%%%%%%%%%%%%%%%%%%%%%%%%%%%%%%%%%%

The results in Tables \ref{TAPD002003},
\ref{TAPD002004} and \ref{TAPD002005} are in agreement
with the expectations above. It can be seen from Tables \ref{TAPD002003}
and \ref{TAPD002004} that the relative areas of the basins of attraction trace
those of constant $\gamma$. In particular, the relative area of the basins of attraction
of the rotating attractors tends towards that at $\gamma=\gamma_0$ from above,
despite having a smaller basin of attraction for the chosen values of $\gamma_1$.
More precisely, the longer $T_{0}$, the closer is the relative area of the basin of attraction
to the value it has at $\gamma=\gamma_{0}$. However, the convergence to the
asymptotic value is rather slow: for instance in Table \ref{TAPD002003},
even $T_{0}=2000$ is not enough to reach the values corresponding to $\gamma=0.02$.
The simulations for time varying dissipation have in general taken 300 000 or 400 000
initial conditions in phase space. In some cases more points were used for additional
accuracy.

%%%%%%%%%%%%%%%%%%%%%%%%%%%%%%%%%%%%%%%%%%%%%%%%%%%%%%%%%%%%%%%%
\begin{figure}[H]
\begin{floatrow}
\capbtabbox{
\begin{tabular}{cc|c|c|c|c|c|c|}
\cline{3-5}
& &\multicolumn{3}{c|}{\multirow{1}{*}{Basin of Attraction \%}}\\
\cline{3-5}
& & FP & PR/NR & OSC \\
\hline
\multicolumn{1}{|c|}{\multirow{7}{*}{$T_0$}}
& 0 & 69.94 & 4.42 & 21.23 \\ \hhline{~----}
\multicolumn{1}{|c|}{} & 25 & 69.80 & 4.42 & 21.36 \\ \hhline{~----}
\multicolumn{1}{|c|}{} & 50 & 69.57 & 4.45 & 21.52  \\ \hhline{~----}
\multicolumn{1}{|c|}{} & 75 & 69.40 & 4.64 & 21.33  \\ \hhline{~----}
\multicolumn{1}{|c|}{} & 100 & 68.84 & 4.85 & 21.47 \\ \hhline{~----}
\multicolumn{1}{|c|}{} & 200 & 68.82 & 5.10 & 20.99 \\ \hhline{~----}
\multicolumn{1}{|c|}{} & 500 & 69.86 & 5.17 & 19.80 \\ \hhline{~----}
\multicolumn{1}{|c|}{} & 1000 & 70.65 & 5.18 & 18.99 \\ \hhline{~----}
\multicolumn{1}{|c|}{} & 2000 & 71.17 & 5.11 & 18.61 \\ \hline
\end{tabular}}{
\hskip-.5truecm
\caption{\small{Relative areas of the basins of attraction with $\gamma_0 = 0.02$,
$\gamma_1 = 0.03$ and $T_0$ varying.}}
\label{TAPD002003}}
\hskip.5truecm
\ffigbox{
\includegraphics[width=0.45\textwidth]{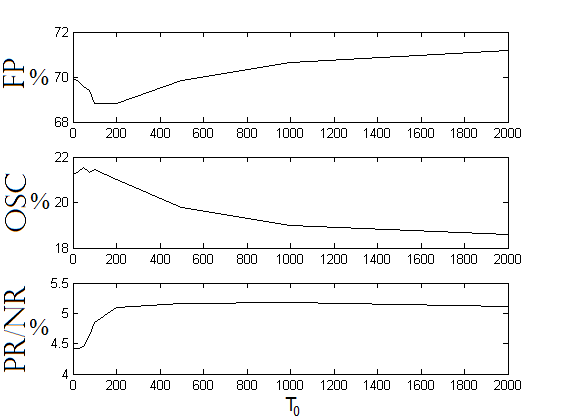}
}{
\caption{\small{Plot of the relative areas of the basins of attraction
as per Table \ref{TAPD002003}}.}
\label{FAPD002003}}
\end{floatrow}
\end{figure}
%%%%%%%%%%%%%%%%%%%%%%%%%%%%%%%%%%%%%%%%%%%%%%%%%%%%%%%%%%%%%%%%

In Table \ref{TAPD002005}, we see that for $\gamma_0=0.02$ and $\gamma_1=0.05$
only two attractors are present: indeed the
oscillating attractors no longer exist for $\gamma=0.05$. Hence,
when $\gamma(t)$ increases and crosses the value at which those attractors disappear,
all the trajectories that up to this time were converging to them, will fall into
the basins of attraction of the persisting attractors, that is the fixed point
and the oscillating solution. In particular the corresponding basins of attraction
will acquire relative areas larger than those they have at constant $\gamma=\gamma_0$,
because of the absorption of all these trajectories.
It is difficult to predict how such trajectories are distributed among the persisting attractors.
In the case of Table \ref{TAPD002005} they seem to be attracted slightly more by
the fixed point, even though the percentage increase is larger for the oscillating solution.

%%%%%%%%%%%%%%%%%%%%%%%%%%%%%%%%%%%%%%%%%%%%%%%%%%%%%%%%%%%%%%%%
\begin{figure}[H]
\begin{floatrow}
\capbtabbox{
\begin{tabular}{cc|c|c|c|c|c|c|}
\cline{3-5}
& &\multicolumn{3}{c|}{\multirow{1}{*}{Basin of Attraction \%}}\\
\cline{3-5}
& & FP & PR/NR & OSC \\
\hline
\multicolumn{1}{|c|}{\multirow{8}{*}{$T_0$}}
& 0 & 73.84 & 1.42 & 23.32 \\ \hhline{~----}
\multicolumn{1}{|c|}{} & 25 & 73.66 & 1.44 & 23.45  \\ \hhline{~----}
\multicolumn{1}{|c|}{} & 50 & 73.37 & 1.50 & 23.63  \\ \hhline{~----}
\multicolumn{1}{|c|}{} & 75 & 72.15 & 2.22 & 23.41  \\ \hhline{~----}
\multicolumn{1}{|c|}{} & 100 & 68.69 & 3.50 & 24.31 \\ \hhline{~----}
\multicolumn{1}{|c|}{} & 200 & 67.46 & 4.76 & 23.03 \\ \hhline{~----}
\multicolumn{1}{|c|}{} & 500 & 69.05 & 5.02 & 20.92 \\ \hhline{~----}
\multicolumn{1}{|c|}{} & 1000 & 69.85 & 5.17 & 19.81 \\ \hhline{~----}
\multicolumn{1}{|c|}{} & 2000 & 70.63 & 5.18 & 19.02 \\ \hline
\end{tabular}}{
\hskip-.5truecm
\caption{\small{Relative areas of the basins of attraction with $\gamma_0 = 0.02$,
$\gamma_1 = 0.04$ and $T_0$ varying.}}
\label{TAPD002004}}
\hskip.5truecm
\ffigbox{
\includegraphics[width=0.45\textwidth]{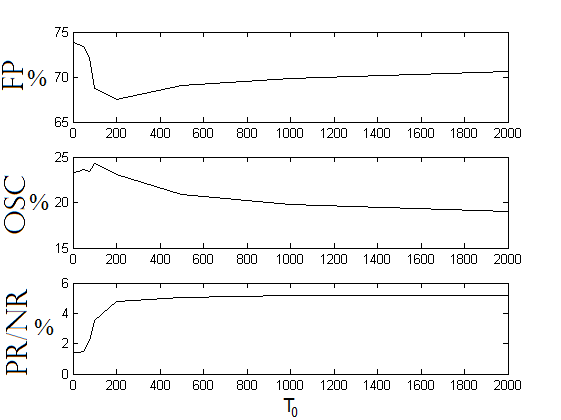}
}{
\caption{\small{Plot of the relative areas of the basins of attraction
as per Table \ref{TAPD002004}.}}
\label{FAPD002004}}
\end{floatrow}
\end{figure}
%%%%%%%%%%%%%%%%%%%%%%%%%%%%%%%%%%%%%%%%%%%%%%%%%%%%%%%%%%%%%%%%

%%%%%%%%%%%%%%%%%%%%%%%%%%%%%%%%%%%%%%%%%%%%%%%%%%%%%%%%%%%%%%%%
\begin{figure}[H]
\begin{floatrow}
\capbtabbox{
\begin{tabular}{cc|c|c|c|c|c|}
\cline{3-5}
& &\multicolumn{2}{c|}{\multirow{1}{*}{Basin of Attraction \%}}\\
\cline{3-5}
& & FP &  OSC \\
\hline
\multicolumn{1}{|c|}{\multirow{8}{*}{$T_0$}}
& 0 & 85.61 & 14.39 \\ \hhline{~----}
\multicolumn{1}{|c|}{} & 25 & 86.01 & 13.99  \\ \hhline{~----}
\multicolumn{1}{|c|}{} & 50 & 86.18 & 13.82  \\ \hhline{~----}
\multicolumn{1}{|c|}{} & 75 & 84.19 & 15.87  \\ \hhline{~----}
\multicolumn{1}{|c|}{} & 100 & 80.42 & 19.58 \\ \hhline{~----}
\multicolumn{1}{|c|}{} & 200 & 75.47 & 24.53 \\ \hhline{~----}
\multicolumn{1}{|c|}{} & 500 & 77.95 & 22.06 \\ \hhline{~----}
\multicolumn{1}{|c|}{} & 1000 & 77.55 & 22.45 \\ \hhline{~----}
\multicolumn{1}{|c|}{} & 1500 & 78.03 & 21.97 \\ \hline
\end{tabular}}{
\hskip-.5truecm
\caption{\small{Relative areas of the basins of attraction with $\gamma_0 = 0.02$,
$\gamma_1 = 0.05$ and $T_0$ varying.}}
\label{TAPD002005}}
\hskip.5truecm
\ffigbox{
\includegraphics[width=0.45\textwidth]{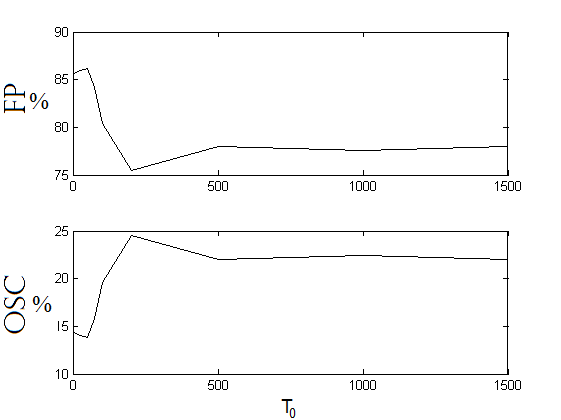}
}{
\caption{\small{Plot of the relative areas of the basins of attraction
as per Table \ref{TAPD002005}.}}
\label{FAPD002005}}
\end{floatrow}
\end{figure}
%%%%%%%%%%%%%%%%%%%%%%%%%%%%%%%%%%%%%%%%%%%%%%%%%%%%%%%%%%%%%%%%

%%%%%%%%%%%%%%%%%%%%%%%%%%%%%%%%%%%%%%%%%%%%%%%%%%%%%%%%%%%%%%%%
\subsection{Decreasing dissipation}
%%%%%%%%%%%%%%%%%%%%%%%%%%%%%%%%%%%%%%%%%%%%%%%%%%%%%%%%%%%%%%%%

In this section we conversely look at the damping coefficient decreasing
from some value $\gamma_0 > \gamma_1$, with different rates of decrease,
see Figure \ref{FricDecFig}. We will consider the cases $\gamma_0=0.04$ and $\gamma_1=0.02$,
$\gamma_0=0.04$ and $\gamma_1=0.03$, $\gamma_0=0.05$ and $\gamma_1=0.02$.

%%%%%%%%%%%%%%%%%%%%%%%%%%%%%%%%%%%%%%%%%%%%%%%%%%%%%%%%%%%%%%%%
\begin{figure}[H]
\centering
\includegraphics[width=0.34\textwidth]{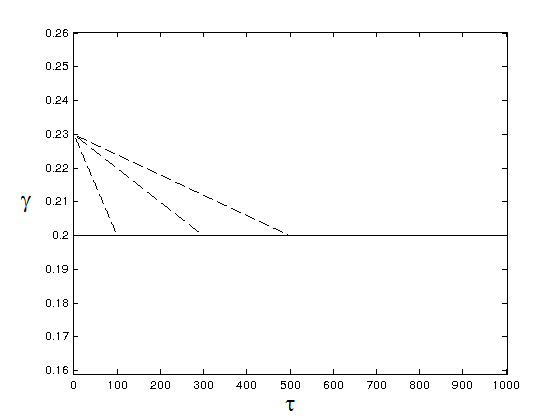}
\caption{\small{Plot of equation \eqref{T0eqn}
with $\gamma_0 = 0.23$, $\gamma_1 = 0.2$ and varying $T_0$.}}
\label{FricDecFig}
\end{figure}
%%%%%%%%%%%%%%%%%%%%%%%%%%%%%%%%%%%%%%%%%%%%%%%%%%%%%%%%%%%%%%%%

In this situation it is possible that more attractors exist at $\gamma_1$ than at $\gamma_0$,
see Table \ref{TAPCFD}.
We again expect that increasing $T_0$ causes the relative areas of the basins of attraction to
tend towards those at $\gamma_0$. The result of this is that solutions which do
not exist at $\gamma_0$ will attract less and less of the phase space as $T_0$
increases, and for $T_0$ large enough their basins of attraction will tend to zero.

%%%%%%%%%%%%%%%%%%%%%%%%%%%%%%%%%%%%%%%%%%%%%%%%%%%%%%%%%%%%%%%%
\begin{figure}[H]
\begin{floatrow}
\capbtabbox{
\begin{tabular}{cc|c|c|c|c|c|c|}
\cline{3-5}
& &\multicolumn{3}{c|}{\multirow{1}{*}{Basin of Attraction \%}}\\
\cline{3-5}
& & FP & PR/NR & OSC \\
\hline
\multicolumn{1}{|c|}{\multirow{7}{*}{$T_0$}}
& 0 & 71.95 & 4.57 & 18.90 \\ \hhline{~----}
\multicolumn{1}{|c|}{} & 25 & 71.85 & 4.60 & 18.94  \\ \hhline{~----}
\multicolumn{1}{|c|}{} & 50 & 72.36 & 4.48 & 18.69  \\ \hhline{~----}
\multicolumn{1}{|c|}{} & 75 & 73.64 & 4.43 & 17.51  \\ \hhline{~----}
\multicolumn{1}{|c|}{} & 100 & 74.10 & 4.32 & 17.27 \\ \hhline{~----}
\multicolumn{1}{|c|}{} & 200 & 72.31 & 2.99 & 21.71 \\ \hhline{~----}
\multicolumn{1}{|c|}{} & 500 & 71.51 & 2.09 & 24.31 \\ \hhline{~----}
\multicolumn{1}{|c|}{} & 1000 & 72.61 & 1.79 & 23.81 \\ \hhline{~----}
\multicolumn{1}{|c|}{} & 2000 & 73.11 & 1.63 & 23.64 \\ \hline
\end{tabular}}{
\hskip-.5truecm
\caption{\small{Relative areas of the basins of attraction with $\gamma_0 = 0.04$,
$\gamma_1 = 0.02$ and $T_0$ varying.}}
\label{TAPD004002}}
\hskip.5truecm
\ffigbox{
\includegraphics[width=0.45\textwidth]{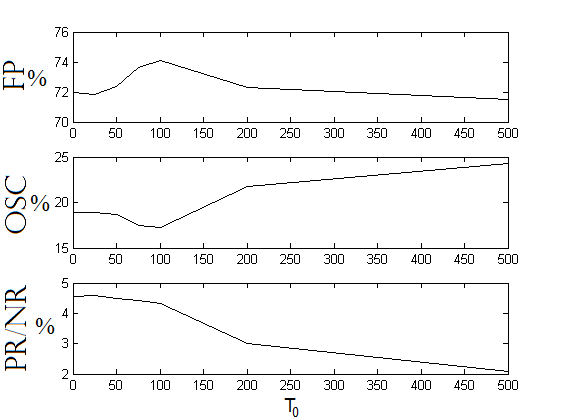}
}{
\caption{\small{Plot of the relative areas of the basins of attraction
as per Table \ref{TAPD004002}.}}
\label{FAPD004002}}
\end{floatrow}
\end{figure}
%%%%%%%%%%%%%%%%%%%%%%%%%%%%%%%%%%%%%%%%%%%%%%%%%%%%%%%%%%%%%%%%

\vspace{-.4cm}

%%%%%%%%%%%%%%%%%%%%%%%%%%%%%%%%%%%%%%%%%%%%%%%%%%%%%%%%%%%%%%%%
\begin{figure}[H]
\begin{floatrow}
\capbtabbox{
\begin{tabular}{cc|c|c|c|c|c|c|}
\cline{3-5}
& &\multicolumn{3}{c|}{\multirow{1}{*}{Basin of Attraction \%}}\\
\cline{3-5}
& & FP & PR/NR & OSC \\
\hline
\multicolumn{1}{|c|}{\multirow{6}{*}{$T_0$}}
& 0 & 69.94 & 4.42 & 21.23 \\ \hhline{~----}
\multicolumn{1}{|c|}{} & 25 & 69.73 & 4.50 & 21.28  \\ \hhline{~----}
\multicolumn{1}{|c|}{} & 50 & 70.78 & 4.23 & 20.77  \\ \hhline{~----}
\multicolumn{1}{|c|}{} & 75 & 72.03 & 3.45 & 21.07  \\ \hhline{~----}
\multicolumn{1}{|c|}{} & 100 & 71.77 & 2.95 & 22.33 \\ \hhline{~----}
\multicolumn{1}{|c|}{} & 500 & 72.61 & 1.79 & 23.81 \\ \hhline{~----}
\multicolumn{1}{|c|}{} & 1000 & 73.11 & 1.63 & 23.64 \\ \hline
\end{tabular}}{
\hskip-.5truecm
\caption{\small{Relative areas of the basins of attraction with $\gamma_0 = 0.04$,
$\gamma_1 = 0.03$ and $T_0$ varying.}}
\label{TAPD004003}}
\hskip.5truecm
\ffigbox{
\includegraphics[width=0.45\textwidth]{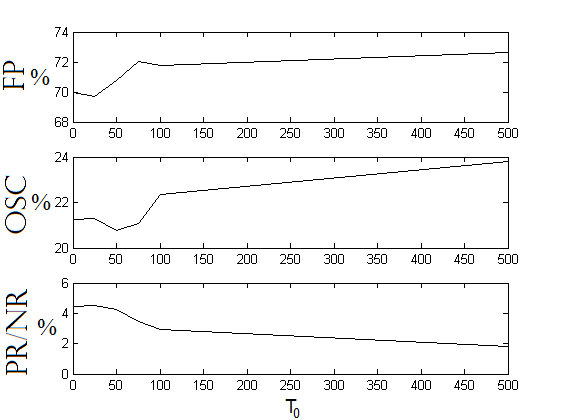}
}{
\caption{\small{Plot of the relative areas of the basins of attraction
as per Table \ref{TAPD004003}.}}
\label{FAPD004003}}
\end{floatrow}
\end{figure}
%%%%%%%%%%%%%%%%%%%%%%%%%%%%%%%%%%%%%%%%%%%%%%%%%%%%%%%%%%%%%%%%

\vspace{-.4cm}

%%%%%%%%%%%%%%%%%%%%%%%%%%%%%%%%%%%%%%%%%%%%%%%%%%%%%%%%%%%%%%%%
\begin{figure}[H]
\begin{floatrow}
\capbtabbox{
\begin{tabular}{cc|c|c|c|c|c|c|}
\cline{3-5}
& &\multicolumn{3}{c|}{\multirow{1}{*}{Basin of Attraction \%}}\\
\cline{3-5}
& & FP & PR/NR & OSC \\
\hline
\multicolumn{1}{|c|}{\multirow{7}{*}{$T_0$}}
& 0 & 71.95 & 4.57 & 18.90 \\ \hhline{~----}
\multicolumn{1}{|c|}{} & 25 & 72.13 & 4.58 & 18.72  \\ \hhline{~----}
\multicolumn{1}{|c|}{} & 50 & 72.97 & 4.24 & 18.55  \\ \hhline{~----}
\multicolumn{1}{|c|}{} & 75 & 76.14 & 3.15 & 17.56  \\ \hhline{~----}
\multicolumn{1}{|c|}{} & 100 & 77.02 & 2.18 & 18.62 \\ \hhline{~----}
\multicolumn{1}{|c|}{} & 200 & 77.71 & 0.31 & 21.67 \\ \hhline{~----}
\multicolumn{1}{|c|}{} & 500 & 81.94 & 0.00 & 18.06 \\ \hline
\end{tabular}}{
\caption{\small{Relative areas of the basins of attraction with $\gamma_0 = 0.05$,
$\gamma_1 = 0.02$ and $T_0$ varying.}}
\label{TAPD005002}}
\hskip.5truecm
\ffigbox{
\includegraphics[width=0.45\textwidth]{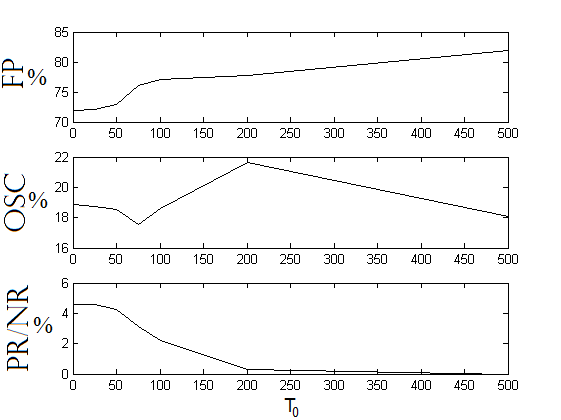}
}{
\caption{\small{Plot of the relative areas of the basins of attraction
as per Table \ref{TAPD005002}.}}
\label{FAPD005002}}
\end{floatrow}
\end{figure}
%%%%%%%%%%%%%%%%%%%%%%%%%%%%%%%%%%%%%%%%%%%%%%%%%%%%%%%%%%%%%%%%

Tables \ref{TAPD004002} and \ref{TAPD004003} illustrate cases in which the system
admits the same set of attractors for both values $\gamma_0$ and $\gamma_1$
of the damping coefficient.
An example of what happens when an attractor exists at $\gamma_1$ but
not at $\gamma_0$ can be seen in the results of Table \ref{TAPD005002},
where $\gamma(t)$ varies from $0.05$ to $0.02$. As the damping coefficient
starts off at a larger value, then decreases to some smaller value,
we also expect the change in the basins of attraction to happen
over shorter values of $T_0$. The reasoning for this is simply that larger values of
dissipation cause trajectories to move onto attractors in less time. Increasing $T_0$
results in the system remaining at higher values of dissipation for more time and
thus trajectories land on the attractors in less time.

%%%%%%%%%%%%%%%%%%%%%%%%%%%%%%%%%%%%%%%%%%%%%%%%%%%%%%%%%%%%%%%%
\section{Numerics for the inverted pendulum}
\label{InvPenNumericSect}
%%%%%%%%%%%%%%%%%%%%%%%%%%%%%%%%%%%%%%%%%%%%%%%%%%%%%%%%%%%%%%%%

The upwards fixed point of the inverted pendulum can be made stable
for large values of $\beta$, i.e when the amplitude of the
oscillations is large relative to the length of the pendulum.
In this section we numerically investigate the system \eqref{PenEqnTau}
for parameter values for which this happens.
For simplicity, as mentioned in the introduction,
we refer to this case as the inverted pendulum.
It can be more convenient to set $x=\pi+\xi$, so as to centre the origin at
the upwards position of the pendulum. Then the equations of motion become
\begin{equation}
\label{PenEqnTauInv}
\xi'' + f(\tau)\sin{\xi} + \gamma\xi' = 0, \qquad f(\tau) = -(\alpha + \beta \cos{\tau}),
\end{equation}
\[\alpha = \frac{g}{\ell\omega^2}, \qquad \beta = \frac{b}{\ell}, \qquad \tau = \omega t.\]
The difference between equations \eqref{PenEqnTau} and \eqref{PenEqnTauInv} is that here
the parameter $\alpha$ has changed sign.

%%%%%%%%%%%%%%%%%%%%%%%%%%%%%%%%%%%%%%%%%%%%%%%%%%%%%%%%%%%%%%%%
\begin{figure}[H]
\centering
\subfloat[]{\includegraphics[width=0.36\textwidth]{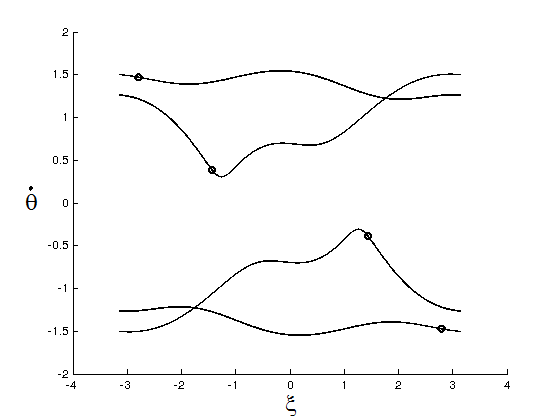}}
\subfloat[]{\includegraphics[width=0.36\textwidth]{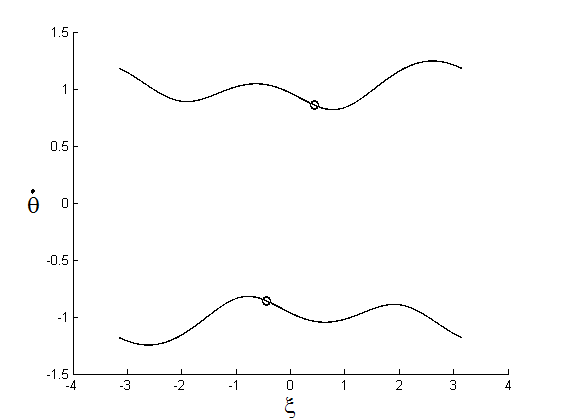}}\\
\subfloat[]{\includegraphics[width=0.36\textwidth]{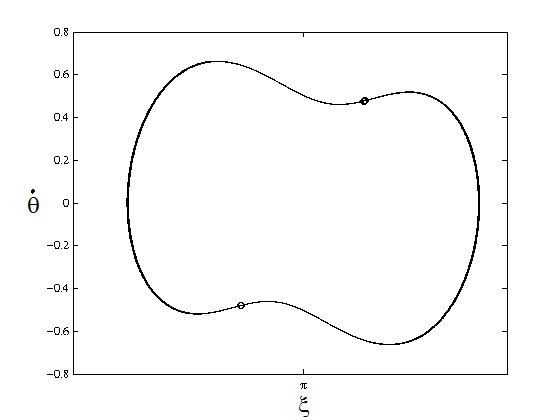}}
\subfloat[]{\includegraphics[width=0.36\textwidth]{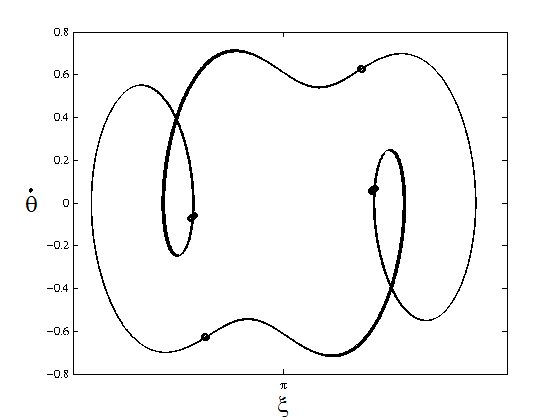}}
\caption{\small{Attracting solutions for the system \eqref{PenEqnTauInv} with
$\alpha=0.1$ and $\beta=0.545$.
Figure (a) shows an example of the positive and negative rotating attractors
with period 2, taken for $\gamma = 0.05$. Figure (b) shows the positive and negative
attractors with period 1 when $\gamma = 0.23$. Figures (c) and (d) show the oscillatory
attractors with periods 2 and 4 respectively when $\gamma$ is taken equal to 0.2725.
These solutions oscillate about the downwards fixed point $\xi = \pi$ and the axis has
been shifted to show a connected curve in phase space.  The period of each solution
can be deduced from the circles corresponding to the Poincar\'e map.}}
\label{ConstantAttractors}
\end{figure}
%%%%%%%%%%%%%%%%%%%%%%%%%%%%%%%%%%%%%%%%%%%%%%%%%%%%%%%%%%%%%%%%

The stability of the upwards fixed point creates interesting dynamics to study numerically,
however it means that the system is no longer a perturbation
of the simple pendulum system. This in turn has the result that the analysis in
Section \ref{ThresholdsSect} to compute the thresholds of friction
cannot be applied. However, we shall see that the very idea that attractors
have a threshold value below which they always exist does not apply
to the inverted pendulum: both increasing and
decreasing the damping coefficient can create and destroy solutions.

For numerical simulations of the inverted pendulum throughout we shall
take parameters $\alpha = 0.1$ and $\beta = 0.545$,
which are within the stable regime for the upwards position.
For these parameter values the function $f(\tau)$ changes sign.
As such, the analysis in Appendix \ref{LinearisedApp} cannot be applied.
Again these particular parameter values
were also investigated in \cite{DPP}, but with a small value for
the damping coefficient, that is $\gamma=0.08$,
where only three attractors appeared in the system:
the upwards fixed point and the left and right rotating solutions.
We have opted to focus on larger dissipation because
the range of values considered for $\gamma$ allows us to incorporate already a
a wide variety of dynamics, in which remarkable phenomena occur, and,
at the same time, larger values of $\gamma$ are better suited to numerical simulation
because of the shorter integration times.
We note that, for the values of the parameters chosen,
no strange attractors arise: numerically, besides the fixed points,
only periodic attractors are found.

For constant dissipation we provide results for $\gamma \in [0.05, 0.6]$.
These values of $\gamma$ are considered to correspond to large dissipation,
however non-fixed-point solutions still persist due to the large coefficient $\beta$
of the forcing term. Some examples of the persisting non-fixed-point solutions
can be seen in Figure \ref{ConstantAttractors}; of course,
the exact form of the curves depends on the particular choices of $\gamma$.
For $\gamma$ varying in the range considered the following attractors arise
(we follow the same convention as in Section \ref{PenNumericSect}
when saying that a solution has period $\mathpzc{n}$):
the upwards fixed point (FP), the downwards fixed point (DFP),
a positively rotating period 1 solution (PR),
a negatively rotating period 1 solution (NR),
a positively rotating period 2 solution (PR2),
a negatively rotating period 2 solution (NR2),
an oscillating period 2 solution (DO2) and
an oscillating period 4 solution (DO4). However,
as we will see, the solution DO2 deserves a separate, more detailed discussion.

%%%%%%%%%%%%%%%%%%%%%%%%%%%%%%%%%%%%%%%%%%%%%%%%%%%%%%%%%%%%%%%%
\begin{figure}[H]
\centering
\subfloat[]{\includegraphics[width=0.33\textwidth]{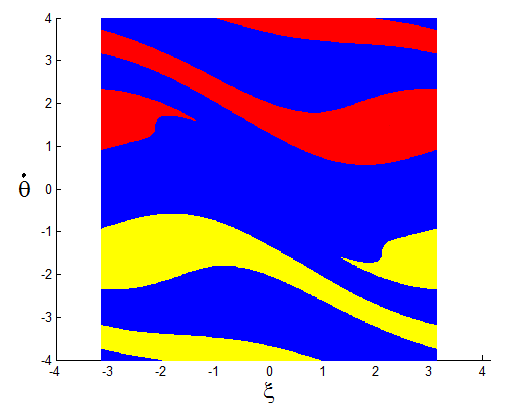}}
\subfloat[]{\includegraphics[width=0.33\textwidth]{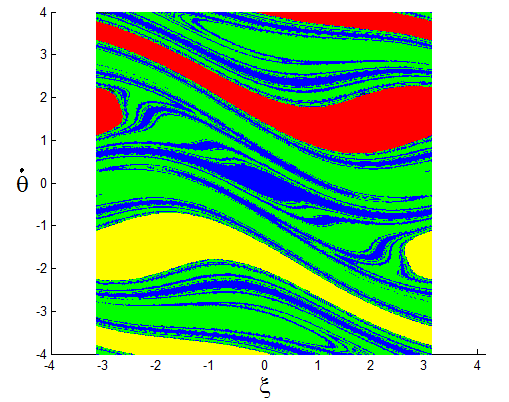}}
\subfloat[]{\includegraphics[width=0.33\textwidth]{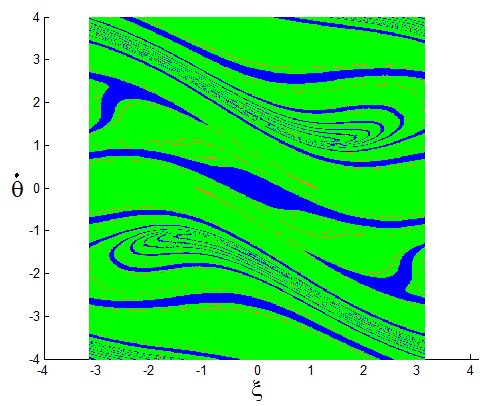}}
\caption{\small{Basins of attraction for constant dissipation with $\gamma = 0.2$, $0.23$ and $0.2725$
from left to right respectively. The fixed point (FP) is shown in blue,
the positively rotating solution (PR) in red,
the negatively rotating solution (NR) in yellow, the downwards oscillation with period 2 (DO2)
in green and finally the downwards oscillation with period 4 (DO4) in orange.}}
\label{BasinsConstant}
\end{figure}
%%%%%%%%%%%%%%%%%%%%%%%%%%%%%%%%%%%%%%%%%%%%%%%%%%%%%%%%%%%%%%%%

The basins of attraction corresponding to the values $\gamma = 0.2$, $0.23$
and $0.2725$ are shown in Figure \ref{BasinsConstant}.
The relative areas of the basins of attraction for $\gamma \in [0.05,0.6]$
are listed in Table \ref{TAPCF}. Again the positive and negative rotations
have been listed together as any difference in the size of their basins
of attraction is expected to be due to numerical inaccuracies.

In Figure \ref{UpwardsConstantDissipFig}, there is a large jump in the relative area of
the basin of attraction for the upwards fixed point (FP) between the values of $\gamma = 0.22$
and $0.225$, approximately:
this is due to the appearance of the oscillatory solution
which oscillates about the downwards pointing fixed point $(\xi = \pi)$. For values of
$\gamma$ slightly larger than $0.22$ large amounts of phase space move close to the
solution DO2, where they remain for long periods of time; however they do not land
on the solution and are eventually attracted to FP. The percentage
of phase space which does this is marked in Figure \ref{UpwardsConstantDissipFig}
by a dotted line, which becomes solid when the trajectories
remain on the solution for all time (however, see comments below).

%%%%%%%%%%%%%%%%%%%%%%%%%%%%%%%%%%%%%%%%%%%%%%%%%%%%%%%%%%%%%%%%
\begin{table}[H]
\centering
\begin{tabular}{cc|c|c|c|c|c|c|c|}
\cline{3-8}
& &\multicolumn{6}{c|}{\multirow{1}{*}{Basin of Attraction \%}}\\
\cline{3-8}
& & FP & DFP & PR/NR & PR2/NR2 & DO2 & DO4 \\
\hline
\multicolumn{1}{|c|}{\multirow{22}{*}{$\gamma$}}
& 0.0500 & 4.30 & 0.00 & 0.00 & 47.85 & 0.00 & 0.00 \\ \hhline{~-------}
\multicolumn{1}{|c|}{} & 0.0750 & 5.08 & 0.00 & 0.00 & 47.46 & 0.00 & 0.00  \\ \hhline{~-------}
\multicolumn{1}{|c|}{} & 0.0900 & 7.41 & 0.00 & 0.00 & 46.30 & 0.00 & 0.00 \\ \hhline{~-------}
\multicolumn{1}{|c|}{} & 0.1000 & 8.51 & 0.00 & 45.74 & 0.00 & 0.00 & 0.00  \\ \hhline{~-------}
\multicolumn{1}{|c|}{} & 0.1700 & 49.65 & 0.00 & 25.17 & 0.00 & 0.00 & 0.00 \\ \hhline{~-------}
\multicolumn{1}{|c|}{} & 0.2000 & 64.31 & 0.00 & 17.84 & 0.00 & 0.00 & 0.00 \\ \hhline{~-------}
\multicolumn{1}{|c|}{} & 0.2230 & 72.09 & 0.00 & 13.95 & 0.00 & 0.00 & 0.00 \\ \hhline{~-------}
\multicolumn{1}{|c|}{} & 0.2250 & 27.60 & 0.00 & 13.59 & 0.00 & 45.22 & 0.00 \\ \hhline{~-------}
\multicolumn{1}{|c|}{} & 0.2300 & 25.00 & 0.00 & 12.68 & 0.00 & 49.61 & 0.00 \\ \hhline{~-------}
\multicolumn{1}{|c|}{} & 0.2500 & 15.87 & 0.00 & 8.49 & 0.00 & 67.16 & 0.00  \\ \hhline{~-------}
\multicolumn{1}{|c|}{} & 0.2690 & 16.50 & 0.00 & 2.13 & 0.00 & 79.25 & 0.00  \\ \hhline{~-------}
\multicolumn{1}{|c|}{} & 0.2694 & 17.26 & 0.00 & 0.00 & 0.00 & 82.74 & 0.00 \\ \hhline{~-------}
\multicolumn{1}{|c|}{} & 0.2700 & 17.28 & 0.00 & 0.00 & 0.00 & 82.73  & 0.00  \\ \hhline{~-------}
\multicolumn{1}{|c|}{} & 0.2725 & 17.21 & 0.00 & 0.00 & 0.00 & 79.44 & 3.35   \\ \hhline{~-------}
\multicolumn{1}{|c|}{} & 0.2800 & 17.30 & 0.00 & 0.00 & 0.00 & 82.70 & 0.00   \\ \hhline{~-------}
\multicolumn{1}{|c|}{} & 0.2900 & 17.30 & 0.00 & 0.00 & 0.00 & 82.70 & 0.00 \\ \hhline{~-------}
\multicolumn{1}{|c|}{} & 0.3000 & 16.97 & 0.00 & 0.00 & 0.00 & 83.03 & 0.00 \\ \hhline{~-------}
\multicolumn{1}{|c|}{} & 0.4600 & 9.61 & 0.00 & 0.00 & 0.00 & 90.39 & 0.00 \\ \hhline{~-------}
\multicolumn{1}{|c|}{} & 0.4700 & 9.80 & 90.20 & 0.00 & 0.00 & 0.00 & 0.00 \\ \hhline{~-------}
\multicolumn{1}{|c|}{} & 0.5000 & 10.06 & 89.94 & 0.00 & 0.00 & 0.00 & 0.00 \\ \hhline{~-------}
\multicolumn{1}{|c|}{} & 0.5500 & 10.32 & 89.68 & 0.00 & 0.00 & 0.00 & 0.00 \\ \hhline{~-------}
\multicolumn{1}{|c|}{} & 0.6000 & 8.79 & 91.21 & 0.00 & 0.00 & 0.00 & 0.00 \\ \hline
\end{tabular}
\caption{\small{Relative areas of the basins of attraction with constant
damping coefficient $\gamma$. The solutions are named as per
Figures \ref{ConstantAttractors} and \ref{BasinsConstant} with the addition
of the downwards fixed point (DFP) and the rotating period 2 solutions
(PR2/NR2). For details on the DO2 solution we refer to the text.}}
\label{TAPCF}
\end{table}
%%%%%%%%%%%%%%%%%%%%%%%%%%%%%%%%%%%%%%%%%%%%%%%%%%%%%%%%%%%%%%%%

\vspace{-.4cm}

%%%%%%%%%%%%%%%%%%%%%%%%%%%%%%%%%%%%%%%%%%%%%%%%%%%%%%%%%%%%%%%%
\begin{figure}[H]
\includegraphics[width=0.95\textwidth]{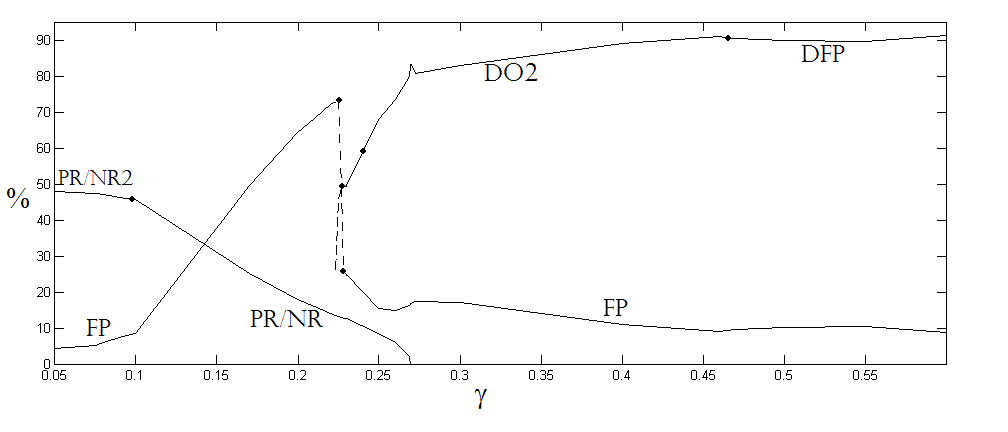}
\caption{\small{Relative sizes of basins of attraction with constant $\gamma$ as per
Table \ref{TAPCF}. The lines are labeled as in Table \ref{TAPCF} and
regions in which a bifurcation takes place are marked with a dot.
The basin of attraction for the oscillatory solution
with period 4 (DO4), has not been included due to its small size and the solutions low
range of persistence with respect to $\gamma$.
The broken lines for FP and DO2 represent areas of transition
just before DO2 (and the solutions created by the period doubling bifurcation) becomes
stable, see text.}}
\label{UpwardsConstantDissipFig}
\end{figure}
%%%%%%%%%%%%%%%%%%%%%%%%%%%%%%%%%%%%%%%%%%%%%%%%%%%%%%%%%%%%%%%%

The solution DO4 listed in Table \ref{TAPCF} is found to persist only in the interval
$[0.272,0.27422]$, where it only attracts a small amount of the phase space.
As such it has not been included in Figure \ref{UpwardsConstantDissipFig}.

%%%%%%%%%%%%%%%%%%%%%%%%%%%%%%%%%%%%%%%%%%%%%%%%%%%%%%%%%%%%%%%%
\begin{figure}[H]
\centering
\subfloat[]{\includegraphics[width=0.25\textwidth]{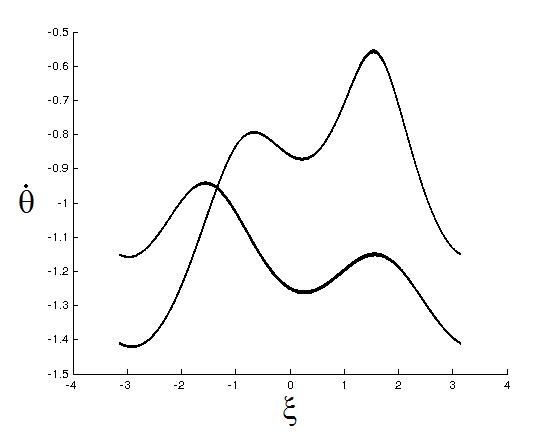}}
\subfloat[]{\includegraphics[width=0.25\textwidth]{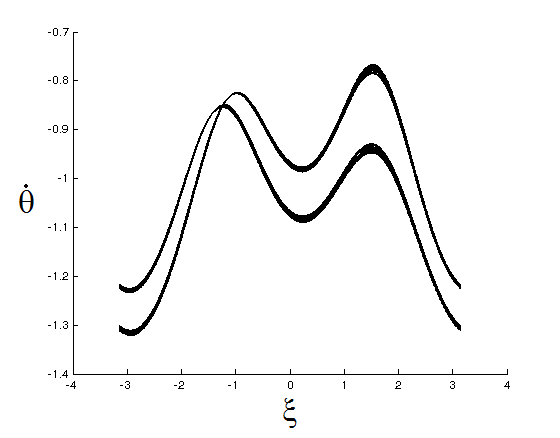}}
\subfloat[]{\includegraphics[width=0.25\textwidth]{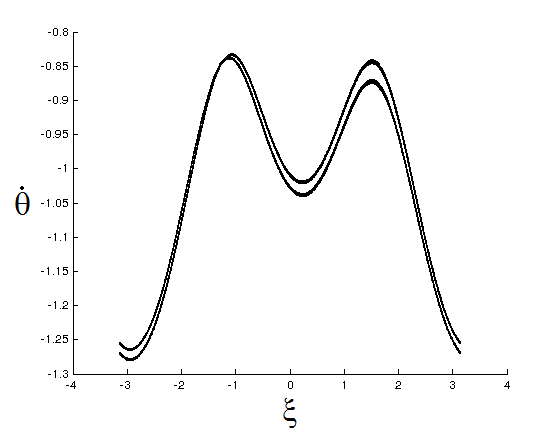}}
\subfloat[]{\includegraphics[width=0.25\textwidth]{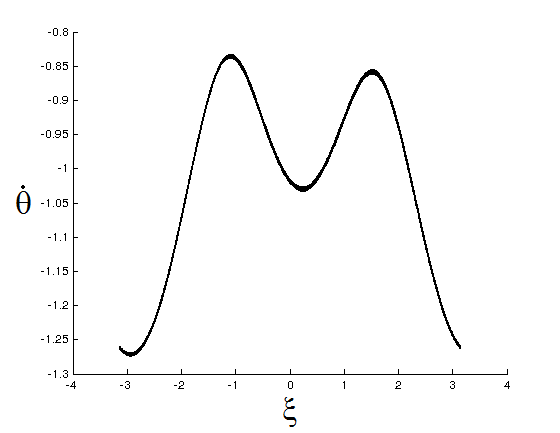}}
\caption{\small{The transition from the period 2 rotating solutions to the period 1 rotating
solutions. $\gamma = 0.09$, 0.094, 0.096 and 0.098 from (a) to (d) respectively.}}
\label{InvRotTransP2ToP1}
\end{figure}
%%%%%%%%%%%%%%%%%%%%%%%%%%%%%%%%%%%%%%%%%%%%%%%%%%%%%%%%%%%%%%%%

Numerical simulations can be used to find estimates for the values of $\gamma$ at which
solutions appear/disappear. This is done by starting with initial conditions on a
solution, then allowing the parameter $\gamma$ to be varied to see for which value
that solution vanishes. In Figure \ref{InvRotTransP2ToP1} the transition from the period 2
rotating solution to the period 1 rotating solution can be observed: by moving
towards smaller values of $\gamma$, this corresponds to a period doubling bifurcation
\cite{bifurcation1,bifurcation2} (period halving, if we think of $\gamma$ as increasing).
Similarly, starting on the period 1 rotating attractor and increasing further $\gamma$,
the solution disappears at $\gamma \approx 0.2694$.

The same analysis can be done for the oscillatory solutions.
We find that the downwards oscillatory solution labeled DO2
persists for $\gamma$ in the interval $[0.224,0.46]$, approximately.
However such a solution is really a period 2 solution only
for $\gamma$ greater than $\gamma\approx0.24$. In the interval $[0.224,0.24]$
the trajectory is ``thick'', see Figure \ref{DO2Fuzzyness}:
only due to its similarity to the solution DO2 and to prevent Table \ref{TAPCF}
having yet more columns, the basin of attraction of these solutions in that range
has also been listed under that of DO2. Nevertheless, by moving $\gamma$ backwards
starting from $0.24$ we have a sequence of period doubling bifurcations,
corresponding to values of $\gamma$ closer and closer to each other.
A period doubling cascade is expected to lead to a chaotic attractor,
which, however, may survive only for a tiny window of values of $\gamma$
(at $\gamma=0.223$ it has already definitely disappeared) and has a very small
basin of attraction (for $\gamma$ getting closer to the value 0.223
its relative area goes to zero).
The appearance of chaotic attractors for small sets
of parameters and with small basins of attraction has been observed
in similar contexts of multistable dissipative systems
close to the conservative limit \cite{FeudelGrebogi}.
For the value $\gamma = 0.223$ numerical
simulations find that trajectories remain in the region of phase space occupied
by DO2 for a long time, before eventually moving onto the fixed point.
As $\gamma$ increases further towards $\gamma\approx0.46$,
the amplitude of the period 2 oscillatory solution gradually decreases and taking
$\gamma$ larger causes a slow spiral into the downwards fixed point,
which now becomes stable.

%%%%%%%%%%%%%%%%%%%%%%%%%%%%%%%%%%%%%%%%%%%%%%%%%%%%%%%%%%%%%%%%
\begin{figure}[H]
\centering
\subfloat[]{\includegraphics[width=0.25\textwidth]{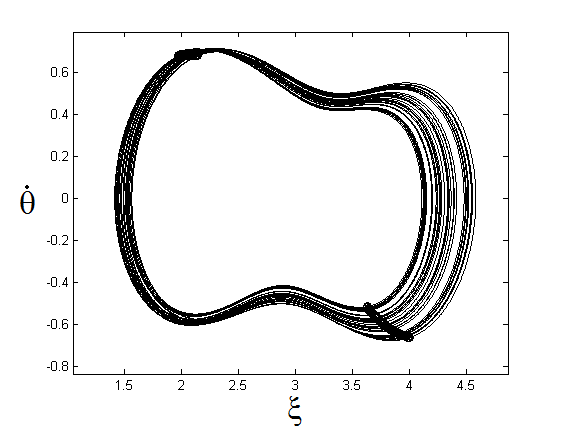}}
\subfloat[]{\includegraphics[width=0.25\textwidth]{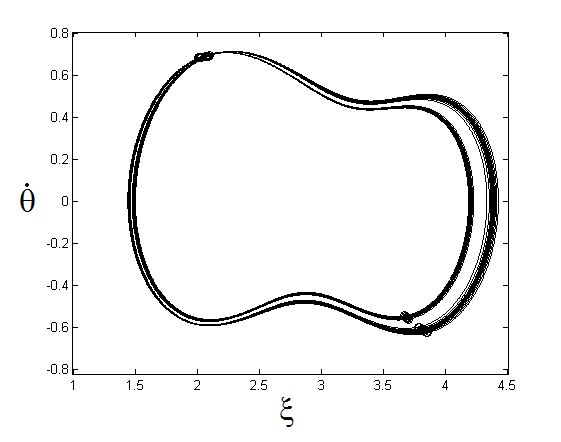}}
\subfloat[]{\includegraphics[width=0.25\textwidth]{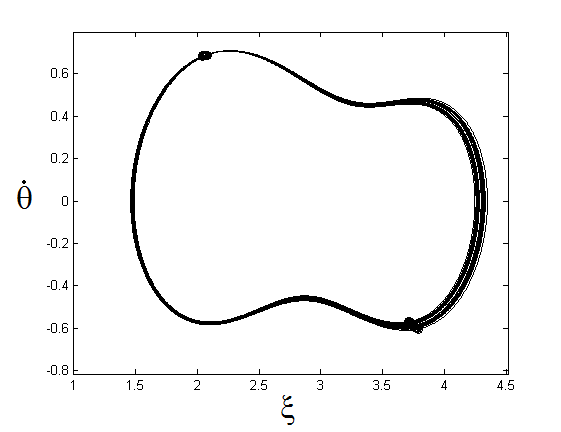}}
\subfloat[]{\includegraphics[width=0.25\textwidth]{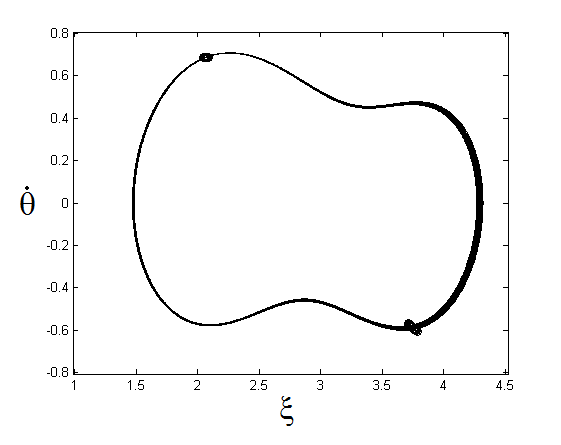}}
\caption{\small{The solution DO2 for different values of time independent $\gamma$. As the
damping coefficient is increased, the solution becomes more clearly defined:
this is due to a period halving bifurcation which stops when the period becomes $2$.
The damping coefficient is $\gamma = 0.23$, 0.235, 0.239 and 0.24 from (a) to (d),
respectively. The position of the trajectory at every $2\pi$, i.e the period of the
forcing, is shown by circles.}}
\label{DO2Fuzzyness}
\end{figure}
%%%%%%%%%%%%%%%%%%%%%%%%%%%%%%%%%%%%%%%%%%%%%%%%%%%%%%%%%%%%%%%%

%%%%%%%%%%%%%%%%%%%%%%%%%%%%%%%%%%%%%%%%%%%%%%%%%%%%%%%%%%%%%%%%
\subsection{Increasing dissipation}
\label{IncDisInvSect}
%%%%%%%%%%%%%%%%%%%%%%%%%%%%%%%%%%%%%%%%%%%%%%%%%%%%%%%%%%%%%%%%

As mentioned in Section \ref{IncDisSect}, as $\gamma$ increases from
$\gamma_0$ to $\gamma_1$ it is expected that taking $T_0$ larger causes the sizes of the
basins of attraction to tend towards the sizes of the corresponding basins when
$\gamma = \gamma_0$, when the set of attractors remains the same
for all values in between. If an attractor is replaced by a new attractor (by bifurcation),
then the new attractor inherits the basin of attraction of the old one.

We shall begin by fixing $\gamma_1 = 0.2$ and $\gamma_0\in[0.05,0.2]$,
as for $\gamma=\gamma_{1}$ the basins of attraction are not so
sensitive to initial conditions, see Figure \ref{BasinsConstant}(a),
and for $\gamma$ in that range the set of attractors consists only
of the the upwards fixed point and two rotating solutions;
moreover the profiles of the corresponding relative areas plotted in
Figure \ref{UpwardsConstantDissipFig} are rather smooth and
do not present any sharp jumps.

%%%%%%%%%%%%%%%%%%%%%%%%%%%%%%%%%%%%%%%%%%%%%%%%%%%%%%%%%%%%%%%%
\begin{figure}[H]
\begin{floatrow}
\capbtabbox{
\begin{tabular}{cc|c|c|}
\cline{3-4}
& &\multicolumn{2}{c|}{\multirow{1}{*}{Basin of Attraction \%}}\\
\cline{3-4}
& & FP & PR/NR \\
\hline
\multicolumn{1}{|c|}{\multirow{9}{*}{$T_0$}}
& 0 & 64.31 & 17.84 \\ \hhline{~---}
\multicolumn{1}{|c|}{} & 25 & 51.41 & 24.30  \\ \hhline{~---}
\multicolumn{1}{|c|}{} & 50 & 32.09 & 33.95  \\ \hhline{~---}
\multicolumn{1}{|c|}{} & 75 & 23.48 & 38.26  \\ \hhline{~---}
\multicolumn{1}{|c|}{} & 100 & 17.09 & 41.45 \\ \hhline{~---}
\multicolumn{1}{|c|}{} & 200 & 6.85 & 46.58  \\ \hhline{~---}
\multicolumn{1}{|c|}{} & 500 & 4.35 & 47.82  \\ \hhline{~---}
\multicolumn{1}{|c|}{} & 1000 & 4.15 & 47.92  \\ \hhline{~---}
\multicolumn{1}{|c|}{} & 1500 & 4.09 & 47.96 \\ \hline
\end{tabular}}{
\caption{\small{Relative areas of the basins of attraction with $\gamma_0 = 0.05$,
$\gamma_1 = 0.2$ and $T_0$ varying.}}
\label{TAP00502}}
\hskip.5truecm
\ffigbox{
\includegraphics[width=0.45\textwidth]{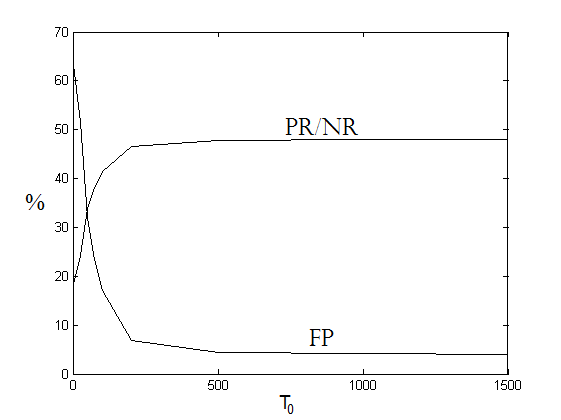}
}{
\caption{\small{Plot of the relative areas of the basins of attraction
as per Table \ref{TAP00502}.}}
\label{FAP00502}}
\end{floatrow}
\end{figure}
%%%%%%%%%%%%%%%%%%%%%%%%%%%%%%%%%%%%%%%%%%%%%%%%%%%%%%%%%%%%%%%%

%%%%%%%%%%%%%%%%%%%%%%%%%%%%%%%%%%%%%%%%%%%%%%%%%%%%%%%%%%%%%%%%
\begin{figure}[H]
\begin{floatrow}
\capbtabbox{
\begin{tabular}{cc|c|c|}
\cline{3-4}
& &\multicolumn{2}{c|}{\multirow{1}{*}{Basin of Attraction \%}}\\
\cline{3-4}
& & FP & PR/NR \\
\hline
\multicolumn{1}{|c|}{\multirow{9}{*}{$T_0$}}
& 0 & 64.31 & 17.84 \\ \hhline{~---}
\multicolumn{1}{|c|}{} & 25 & 49.06 & 25.47 \\ \hhline{~---}
\multicolumn{1}{|c|}{} & 50 & 38.54 & 30.73 \\ \hhline{~---}
\multicolumn{1}{|c|}{} & 75 & 31.25 & 34.37 \\ \hhline{~---}
\multicolumn{1}{|c|}{} & 100 & 25.67 & 37.16 \\ \hhline{~---}
\multicolumn{1}{|c|}{} & 150 & 18.12 & 40.94 \\ \hhline{~---}
\multicolumn{1}{|c|}{} & 200 & 14.14 & 42.93 \\ \hhline{~---}
\multicolumn{1}{|c|}{} & 500 & 9.81 & 45.09 \\ \hhline{~---}
\multicolumn{1}{|c|}{} & 1000 & 9.45 & 45.27 \\ \hline
\end{tabular}}{
\caption{\small{Relative areas of the basins of attraction
with $\gamma_0 = 0.1$, $\gamma_1 = 0.2$ and $T_0$ varying.}}
\label{TAP0102}}
\hskip.5truecm
\ffigbox{
\includegraphics[width=0.45\textwidth]{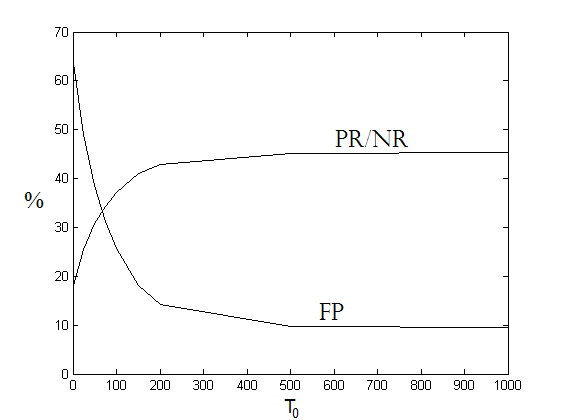}
}{
\caption{\small{Plot of the relative areas of the basins of attraction
as per Table \ref{TAP0102}.}}
\label{FAP0102}}
\end{floatrow}
\end{figure}
%%%%%%%%%%%%%%%%%%%%%%%%%%%%%%%%%%%%%%%%%%%%%%%%%%%%%%%%%%%%%%%%

Tables \ref{TAP00502}, \ref{TAP0102} and \ref{TAP01702}
show the relative area of each basin of attraction as $\gamma$ increases
from 0.05, 0.1 and 0.17, respectively, to 0.2 with varying $T_0$.
It can be seen from the results in Tables \ref{TAP0102} to \ref{TAP01702}
that the numerical simulations are in agreement
with the above expectation. With the exception of Table \ref{TAP00502} the relative
areas of the basins of attraction tend towards those when $\gamma$ is
kept constant at $\gamma_0$. The exception of the case of Table \ref{TAP00502}
is due to the fact that the set of attractors has changed as $\gamma$
passes from $0.05$ to $0.2$: the period 2 rotating solutions have been destroyed
and replaced by the period 1 rotating solutions. However, when the transition occurs,
the new attractors are located in phase space very close to the previous ones and
we find that the initial conditions which were heading towards or had indeed landed
on the period 2 rotating solutions move onto the now present period
1 rotating solutions. On the other hand, when the damping coefficient
crosses the value $\gamma \approx 0.1$, the attractor undergoes topological changes,
but, apart from that, the transition is rather smooth: the location in phase space
and the basin of attraction change continuously.
In conclusion, we find that the relative areas of the basins of attraction
for the two period 1 rotating attractors (PR/NR) tend towards those the
now destroyed period 2 rotating attractors (PR2/NR2) had at $\gamma=\gamma_0$.
As in Section \ref{PenNumericSect} we expect that the sizes of the basin of
attraction at $\gamma=\gamma_0$ are recovered asymptotically as $T_0\to\infty$.
Nevertheless, once more, the larger $T_0$ the smaller is the variation in the relative area:
for instance in Table \ref{TAP00502} for $T_0=100$ the relative area of the
basin of attraction of the fixed point has become nearly $1/4$
of the value for $\gamma=0.2$ constant,
while in order to have a further reduction by a factor of $4$ one has to take $T_0=1000$.

%%%%%%%%%%%%%%%%%%%%%%%%%%%%%%%%%%%%%%%%%%%%%%%%%%%%%%%%%%%%%%%%
\begin{figure}[H]
\begin{floatrow}
\capbtabbox{
\begin{tabular}{cc|c|c|}
\cline{3-4}
& &\multicolumn{2}{c|}{\multirow{1}{*}{Basin of Attraction \%}}\\
\cline{3-4}
& & FP & PR/NR \\
\hline
\multicolumn{1}{|c|}{\multirow{7}{*}{$T_0$}}
& 0 & 64.31 & 17.84 \\ \hhline{~---}
\multicolumn{1}{|c|}{} & 25 & 58.54& 20.73 \\ \hhline{~---}
\multicolumn{1}{|c|}{} & 50 & 56.78 & 21.62 \\ \hhline{~---}
\multicolumn{1}{|c|}{} & 100 & 55.41 & 22.29 \\ \hhline{~---}
\multicolumn{1}{|c|}{} & 200 & 53.69 & 23.15 \\ \hhline{~---}
\multicolumn{1}{|c|}{} & 500 & 51.93 & 24.04 \\ \hhline{~---}
\multicolumn{1}{|c|}{} & 1000 & 50.80 & 24.60 \\ \hhline{~---}
\multicolumn{1}{|c|}{} & 1500 & 50.62 & 24.69 \\ \hline
\end{tabular}}{
\caption{\small{Relative areas of the basins of attraction with $\gamma_0 = 0.17$,
$\gamma_1 = 0.2$ and $T_0$ varying.}}
\label{TAP01702}}
\hskip.5truecm
\ffigbox{
\includegraphics[width=0.45\textwidth]{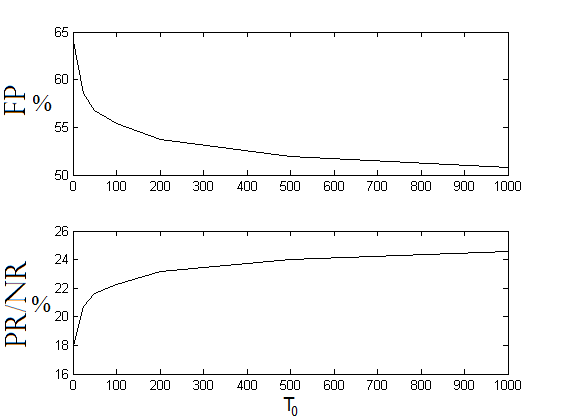}
}{
\caption{\small{Plot of the relative areas of the basins of attraction
as per Table \ref{TAP01702}.}}
\label{FAP01702}}
\end{floatrow}
\end{figure}
%%%%%%%%%%%%%%%%%%%%%%%%%%%%%%%%%%%%%%%%%%%%%%%%%%%%%%%%%%%%%%%%

%%%%%%%%%%%%%%%%%%%%%%%%%%%%%%%%%%%%%%%%%%%%%%%%%%%%%%%%%%%%%%%%
\begin{figure}[H]
\begin{floatrow}
\capbtabbox{
\begin{tabular}{cc|c|c|c|c|c|c|}
\cline{3-5}
& &\multicolumn{3}{c|}{\multirow{1}{*}{Basin of Attraction  \%}}\\
\cline{3-5}
& & FP & PR/NR & DO2 \\
\hline
\multicolumn{1}{|c|}{\multirow{16}{*}{$T_0$}}
& 0 & 25.00 & 12.68 & 49.61 \\ \hhline{~----}
\multicolumn{1}{|c|}{} & 10 & 24.86 & 13.81 & 47.52 \\ \hhline{~----}
\multicolumn{1}{|c|}{} & 15 & 24.34 & 14.84 & 45.98 \\ \hhline{~----}
\multicolumn{1}{|c|}{} & 20 & 24.43 & 15.57 & 44.43 \\ \hhline{~----}
\multicolumn{1}{|c|}{} & 25 & 24.77 & 16.04 & 43.15 \\ \hhline{~----}
\multicolumn{1}{|c|}{} & 50 & 27.60 & 16.98 & 38.45 \\ \hhline{~----}
\multicolumn{1}{|c|}{} & 75 & 30.12 & 17.28 & 35.32 \\ \hhline{~----}
\multicolumn{1}{|c|}{} & 100 & 33.08 & 17.42 & 32.08 \\ \hhline{~----}
\multicolumn{1}{|c|}{} & 150 & 38.29 & 17.56 & 26.58 \\ \hhline{~----}
\multicolumn{1}{|c|}{} & 200 & 42.60 & 17.66 & 22.08 \\ \hhline{~----}
\multicolumn{1}{|c|}{} & 300 & 49.36 & 17.73 & 15.18 \\ \hhline{~----}
\multicolumn{1}{|c|}{} & 400 & 54.20 & 17.75 & 10.30 \\ \hhline{~----}
\multicolumn{1}{|c|}{} & 500 & 57.37 & 17.77 & 7.08 \\ \hhline{~----}
\multicolumn{1}{|c|}{} & 1000 & 63.39 & 17.78 & 1.06 \\ \hhline{~----}
\multicolumn{1}{|c|}{} & 1500 & 64.26 & 17.79 & 0.16 \\ \hhline{~----}
\multicolumn{1}{|c|}{} & 2000 & 64.38 & 17.80 & 0.03 \\ \hline
\end{tabular}}{
\hskip-.5truecm
\caption{\small{Relative areas of the basins of attraction
with $\gamma_0 = 0.2$, $\gamma_1 = 0.23$ and $T_0$ varying.}}
\label{TAP02023}}
\hskip.5truecm
\ffigbox{
\includegraphics[width=0.45\textwidth]{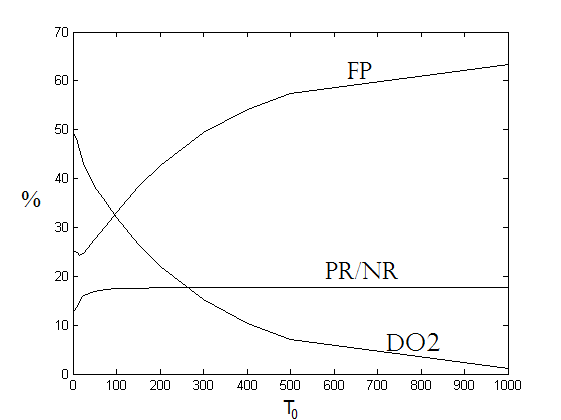}
}{
\caption{\small{Plot of the relative areas of the basins of attraction
as per Table \ref{TAP02023}.}}
\label{FAP02023}}
\end{floatrow}
\end{figure}
%%%%%%%%%%%%%%%%%%%%%%%%%%%%%%%%%%%%%%%%%%%%%%%%%%%%%%%%%%%%%%%%

We now consider the case where either $\gamma_0 =0.2$ and
$\gamma_{1}=0.23$ or $0.2725$ or $\gamma_0=0.23$ and $\gamma_{1}=0.2725$.
Such values for $\gamma$ offer more complexities as not only are there more
attractors to consider, but one may have attractors (PR and NR)
that are destroyed without leaving any trace. When this happens it is not obvious
which persisting attractor will inherit their basins of attraction.
The result of this could cause the final basins of attraction to be drastically
different from those for constant $\gamma$ and even not monotonically
increasing or decreasing as the value of $T_0$ is increased.

In Table \ref{TAP02023} we see that initially, for values of $T_0$ not too large,
the basin of attraction of FP slightly reduces in size,
while those of the rotating solutions PR/NR increase substantially.
Instead, for larger values of $T_0$, the basin of attraction of FP
increases appreciably, while those of PR/NR increase very slowly.
Apparently, the rotating solutions react more quickly as $\gamma$ is varied,
attracting phase space faster, so that the relative areas of their
basins of attraction tend towards the values at $\gamma = \gamma_0$
for shorter initial times $T_0$. It would be interesting to study further this phenomenon.

When $\gamma(t)$ varies from $\gamma=0.2$ to $\gamma=0.2725$ and from
$\gamma=0.23$ to $\gamma=0.2725$, the rotating solutions PR/NR disappear,
so that their basins of attractions are absorbed by the persisting attractors.
In Table \ref{TAP0202725} one sees a very slow movement towards
global attraction of the upwards fixed point, which is the only attractor
persisting for both $\gamma_0$ and $\gamma_1$. However even taking
$T_0 = 5000$ is not enough for the asymptotic behaviour to be approached.
The results in Table \ref{TAP02302725} show that, by taking $T_0$ larger and
larger, the relative areas of the basins of attraction of FP and DO4 both tend
to the values corresponding to $\gamma_0=0.23$ (in particular the basin of
attraction of DO4 becomes negligible).
Nearly all trajectories which were converging towards the rotating solutions
before the latter disappeared are attracted by the period two oscillations.
This could be due to the fact that DO2 is the closest attractor in phase space
which persists at both $\gamma_0$ and $\gamma_1$.

%%%%%%%%%%%%%%%%%%%%%%%%%%%%%%%%%%%%%%%%%%%%%%%%%%%%%%%%%%%%%%%%
\begin{figure}[H]
\begin{floatrow}
\capbtabbox{
\begin{tabular}{cc|c|c|c|c|c|}
\cline{3-5}
& &\multicolumn{3}{c|}{\multirow{1}{*}{Basin of Attraction \%}}\\
\cline{3-5}
& & FP & DO2 & DO4 \\
\hline
\multicolumn{1}{|c|}{\multirow{11}{*}{$T_0$}}
& 0 & 17.21 & 79.44 & 3.35 \\ \hhline{~----}
\multicolumn{1}{|c|}{} & 25 & 18.63 & 78.33 & 3.04   \\ \hhline{~----}
\multicolumn{1}{|c|}{} & 50 & 20.32 & 79.29 & 0.39   \\ \hhline{~----}
\multicolumn{1}{|c|}{} & 75 & 23.38 & 76.52 & 0.12   \\ \hhline{~----}
\multicolumn{1}{|c|}{} & 100 & 25.71 & 74.27 & 0.02 \\ \hhline{~----}
\multicolumn{1}{|c|}{} & 150 & 28.45 & 71.55 & 0.01  \\ \hhline{~----}
\multicolumn{1}{|c|}{} & 200 & 30.92 & 69.08 & 0.00 \\ \hhline{~----}
\multicolumn{1}{|c|}{} & 300 & 35.70 & 64.30 & 0.00 \\ \hhline{~----}
\multicolumn{1}{|c|}{} & 400 & 39.82 & 60.18 & 0.00  \\ \hhline{~----}
\multicolumn{1}{|c|}{} & 500 & 42.76 & 57.24 & 0.00  \\ \hhline{~----}
\multicolumn{1}{|c|}{} & 980 & 54.19 & 45.81 & 0.00  \\ \hhline{~----}
\multicolumn{1}{|c|}{} & 990 & 54.39 & 45.81 & 0.00  \\ \hhline{~----}
\multicolumn{1}{|c|}{} & 995 & 54.30 & 45.70 & 0.00  \\ \hhline{~----}
\multicolumn{1}{|c|}{} & 1000 & 90.11 & 9.89 & 0.00  \\ \hhline{~----}
\multicolumn{1}{|c|}{} & 1005 & 54.58 & 45.43 & 0.00  \\ \hhline{~----}
\multicolumn{1}{|c|}{} & 1010 & 54.52 & 45.48 & 0.00  \\ \hhline{~----}
\multicolumn{1}{|c|}{} & 1020 & 90.34 & 9.66 & 0.00  \\ \hhline{~----}
\multicolumn{1}{|c|}{} & 1030 & 54.81 & 45.19 & 0.00  \\ \hhline{~----}
\multicolumn{1}{|c|}{} & 1050 & 55.13 & 44.87 & 0.00  \\ \hhline{~----}
\multicolumn{1}{|c|}{} & 1500 & 59.78 & 40.22 & 0.00  \\ \hhline{~----}
\multicolumn{1}{|c|}{} & 2000 & 62.12 & 37.88 & 0.00  \\ \hhline{~----}
\multicolumn{1}{|c|}{} & 3000 & 63.89 & 36.11 & 0.00  \\ \hhline{~----}
\multicolumn{1}{|c|}{} & 5000 & 64.29 & 35.71 & 0.00  \\ \hline
\end{tabular}}{
\caption{\small{Relative areas of the basins of attraction with $\gamma_0 = 0.2$,
$\gamma_1 = 0.2725$ and $T_0$ varying.}}
\label{TAP0202725}}
\hskip.5truecm
\ffigbox{
\includegraphics[width=0.45\textwidth]{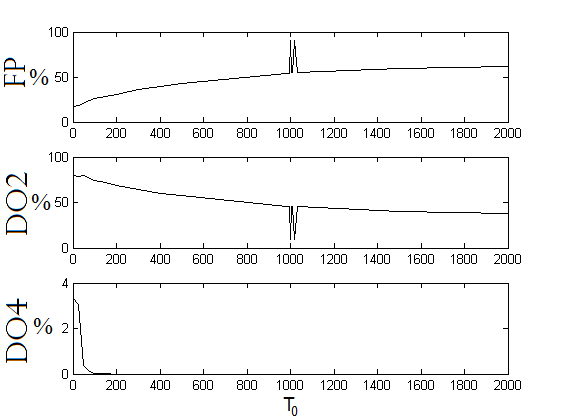}
}{
\caption{\small{Plot of the relative areas of the basins of attraction
as per Table \ref{TAP0202725}.}}
\label{FAP0202725}}
\end{floatrow}
\end{figure}
%%%%%%%%%%%%%%%%%%%%%%%%%%%%%%%%%%%%%%%%%%%%%%%%%%%%%%%%%%%%%%%%

However, the more striking feature of Figures \ref{FAP0202725} and \ref{FAP02302725}
are the jumps corresponding $T_0 = 1000$ in the prior and $T_{0}=100$ and $T_0=500$
in the latter. Moreover such jumps are very localised: for instance in Figure \ref{FAP02302725}
for $T_0=99$ and $T_0=100$ the basins of attraction of FP and DO2 are found to be about $44\%$
and $56\%$, respectively, whereas by slightly increasing or decreasing $T_0$ they settle around
$20\%$ and $80\%$. The quantity of phase space exchanged in these instances is roughly equal
to that attracted to the rotating solutions for $\gamma_0$.  For particular
values of $T_0$ when the rotating attractors disappear their trajectories move to the upwards fixed
point rather than the period 2 oscillations. The reason for this to happen is not clear.
Moreover, note that in principle there could be other jumps, corresponding to
values of $T_{0}$ which have not been investigated: however, it seems hard to make any prediction
as far as it remains unclear how the disappearing basins of attractions are absorbed by the persisting ones.

\vspace{.6cm}

%%%%%%%%%%%%%%%%%%%%%%%%%%%%%%%%%%%%%%%%%%%%%%%%%%%%%%%%%%%%%%%%
\begin{figure}[H]
\begin{floatrow}
\capbtabbox{
\begin{tabular}{cc|c|c|c|c|c|}
\cline{3-5}
& &\multicolumn{3}{c|}{\multirow{1}{*}{Basin of Attraction \%}}\\
\cline{3-5}
& & FP & DO2 & DO4 \\
\hline
\multicolumn{1}{|c|}{\multirow{11}{*}{$T_0$}}
& 0 & 17.21 & 79.44 & 3.35  \\ \hhline{~----}
\multicolumn{1}{|c|}{} & 25 & 17.96 & 79.45 & 2.60   \\ \hhline{~----}
\multicolumn{1}{|c|}{} & 50 & 16.04 & 83.36 & 0.60   \\ \hhline{~----}
\multicolumn{1}{|c|}{} & 75 & 18.39 & 80.27 & 1.34   \\ \hhline{~----}
\multicolumn{1}{|c|}{} & 90 & 19.56 & 80.38 & 0.05 \\ \hhline{~----}
\multicolumn{1}{|c|}{} & 95 & 20.05 & 79.91 & 0.04 \\ \hhline{~----}
\multicolumn{1}{|c|}{} & 97 & 20.08 & 79.89 & 0.04 \\ \hhline{~----}
\multicolumn{1}{|c|}{} & 98 & 20.01 & 79.96 & 0.03 \\ \hhline{~----}
\multicolumn{1}{|c|}{} & 99 & 44.14 & 55.83 & 0.03 \\ \hhline{~----}
\multicolumn{1}{|c|}{} & 100 & 43.96 & 56.01 & 0.03 \\ \hhline{~----}
\multicolumn{1}{|c|}{} & 101 & 20.54 & 79.42 & 0.03 \\ \hhline{~----}
\multicolumn{1}{|c|}{} & 105 & 20.40 & 79.57 & 0.03 \\ \hhline{~----}
\multicolumn{1}{|c|}{} & 110 & 20.79 & 79.19 & 0.02 \\ \hhline{~----}
\multicolumn{1}{|c|}{} & 125 & 21.55 & 78.45 & 0.01 \\ \hhline{~----}
\multicolumn{1}{|c|}{} & 150 & 22.46 & 77.54 & 0.00  \\ \hhline{~----}
\multicolumn{1}{|c|}{} & 200 & 23.33 & 76.67 & 0.00  \\ \hhline{~----}
\multicolumn{1}{|c|}{} & 300 & 24.06 & 75.94 & 0.00 \\ \hhline{~----}
\multicolumn{1}{|c|}{} & 400 & 24.36 & 75.64 & 0.00  \\ \hhline{~----}
\multicolumn{1}{|c|}{} & 490 & 24.54 & 75.46 & 0.00  \\ \hhline{~----}
\multicolumn{1}{|c|}{} & 500 & 49.45 & 50.55 & 0.00  \\ \hhline{~----}
\multicolumn{1}{|c|}{} & 510 & 24.42 & 75.58 & 0.00  \\ \hhline{~----}
\multicolumn{1}{|c|}{} & 1000 & 24.81 & 75.19 & 0.00  \\ \hline
\end{tabular}}{
\caption{\small{Relative areas of the basins of attraction
with $\gamma_0 = 0.23$, $\gamma_1 = 0.2725$ and $T_0$ varying.}}
\label{TAP02302725}}
\hskip.5truecm
\ffigbox{
\includegraphics[width=0.4\textwidth]{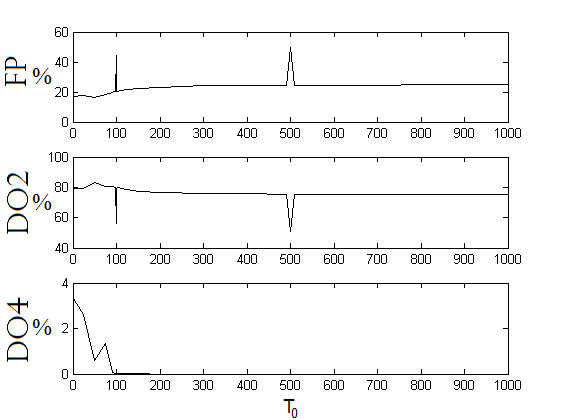}
}{
\caption{\small{Plot of the relative areas of the basins of attraction
as per Table \ref{TAP02302725}.}}
\label{FAP02302725}}
\end{floatrow}
\end{figure}
%%%%%%%%%%%%%%%%%%%%%%%%%%%%%%%%%%%%%%%%%%%%%%%%%%%%%%%%%%%%%%%%
%%%%%%%%%%%%%%%%%%%%%%%%%%%%%%%%%%%%%%%%%%%%%%%%%%%%%%%%%%%%%%%%
\subsection{Decreasing dissipation}
%%%%%%%%%%%%%%%%%%%%%%%%%%%%%%%%%%%%%%%%%%%%%%%%%%%%%%%%%%%%%%%%

Tables \ref{TAP02302}, \ref{TAP0272502}, \ref{TAP02725023} and \ref{TAP0302725}
and the corresponding Figures
\ref{FAP02302}, \ref{FAP0272502}, \ref{FAP02725023} and \ref{FAP0302725}
illustrate the cases when dissipation decreases over an initial period of time $T_0$.
We have considered the cases with $\gamma_0=0.23$, $0.02725$ and $\gamma_1=0.2$, with
$\gamma_0=0.2725$ and $\gamma_1=0.23$ and with $\gamma_0=0.3$ and $\gamma_1=0.2725$.

In particular they show that if the set of attractors
at $\gamma = \gamma_1$ is a proper subset of the set of attractors
which exist at $\gamma = \gamma_0$, then, as $T_0 \to\infty$,
the basin of attraction of each attractor which exists
at $\gamma_1$ turns out to have a relative area which tend to be
greater than or equal to that found for $\gamma = \gamma_0$.
In Table \ref{TAP02302} we consider the situation in which the attractor DO2,
which has a large basin of attraction for $\gamma_0=0.23$, is no longer
present when $\gamma(t)$ has reached the final value $\gamma_1=0.2$:
as a consequence the trajectories which would be attracted by DO2
at $\gamma=\gamma_0$ end up onto the other attractors: in fact most of them
are attracted by the fixed point.

%%%%%%%%%%%%%%%%%%%%%%%%%%%%%%%%%%%%%%%%%%%%%%%%%%%%%%%%%%%%%%%%
\begin{figure}[H]
\begin{floatrow}
\capbtabbox{
\begin{tabular}{cc|c|c|c|c|}
\cline{3-4}
& &\multicolumn{2}{c|}{\multirow{1}{*}{Basin of Attraction \%}}\\
\cline{3-4}
& & FP & PR/NR \\
\hline
\multicolumn{1}{|c|}{\multirow{7}{*}{$T_0$}}
& 0 & 64.31 & 17.84 \\ \hhline{~---}
\multicolumn{1}{|c|}{} & 25 & 70.69 & 14.65 \\ \hhline{~---}
\multicolumn{1}{|c|}{} & 50 & 72.57 & 13.72 \\ \hhline{~---}
\multicolumn{1}{|c|}{} & 75 & 73.24 & 13.38 \\ \hhline{~---}
\multicolumn{1}{|c|}{} & 100 & 73.57 & 13.22 \\ \hhline{~---}
\multicolumn{1}{|c|}{} & 200 & 74.10 & 12.95 \\ \hhline{~---}
\multicolumn{1}{|c|}{} & 500 & 74.42 & 12.79 \\ \hline
\end{tabular}}{
\caption{\small{Relative areas of the basins of attraction with $\gamma_0 = 0.23$,
$\gamma_1 = 0.2$ and $T_0$ varying.}}
\label{TAP02302}}
\hskip.5truecm
\ffigbox{
\includegraphics[width=0.4\textwidth]{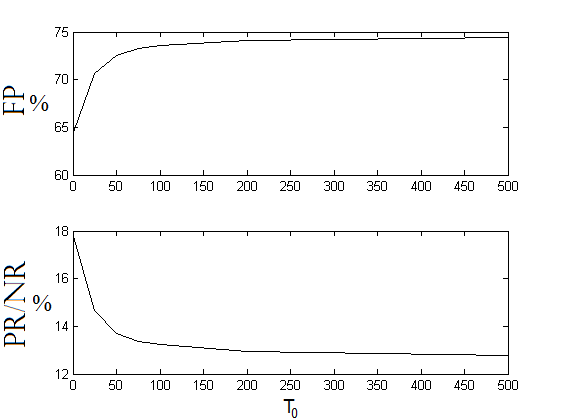}
}{
\caption{\small{Plot of the relative areas of the basins of attraction
as per Table \ref{TAP02302}.}}
\label{FAP02302}}
\end{floatrow}
\end{figure}
%%%%%%%%%%%%%%%%%%%%%%%%%%%%%%%%%%%%%%%%%%%%%%%%%%%%%%%%%%%%%%%%

%%%%%%%%%%%%%%%%%%%%%%%%%%%%%%%%%%%%%%%%%%%%%%%%%%%%%%%%%%%%%%%%
\begin{figure}[H]
\begin{floatrow}
\capbtabbox{
\begin{tabular}{cc|c|c|c|c|}
\cline{3-4}
& &\multicolumn{2}{c|}{\multirow{1}{*}{Basin of Attraction \%}}\\
\cline{3-4}
& & FP & PR/NR \\
\hline
\multicolumn{1}{|c|}{\multirow{13}{*}{$T_0$}}
& 0 & 64.31 & 17.84 \\ \hhline{~---}
\multicolumn{1}{|c|}{} & 5 & 65.80 & 17.10 \\ \hhline{~---}
\multicolumn{1}{|c|}{} & 10 & 69.52 & 15.24 \\ \hhline{~---}
\multicolumn{1}{|c|}{} & 15 & 74.40 & 12.80 \\ \hhline{~---}
\multicolumn{1}{|c|}{} & 20 & 77.90 & 11.05 \\ \hhline{~---}
\multicolumn{1}{|c|}{} & 25 & 80.38 & 9.81 \\ \hhline{~---}
\multicolumn{1}{|c|}{} & 50 & 86.22 & 6.89 \\ \hhline{~---}
\multicolumn{1}{|c|}{} & 75 & 88.64 & 5.68 \\ \hhline{~---}
\multicolumn{1}{|c|}{} & 100 & 90.07 & 4.97 \\ \hhline{~---}
\multicolumn{1}{|c|}{} & 200 & 92.79 & 3.61 \\ \hhline{~---}
\multicolumn{1}{|c|}{} & 500 & 95.92 & 2.04 \\ \hhline{~---}
\multicolumn{1}{|c|}{} & 1000 & 99.34 & 0.33 \\ \hhline{~---}
\multicolumn{1}{|c|}{} & 1500 & 99.35 & 0.32 \\ \hline
\end{tabular}}{
\caption{\small{Relative areas of the basins of attraction
with $\gamma_0 = 0.2725$, $\gamma_1 = 0.2$ and $T_0$ varying.}}
\label{TAP0272502}}
\hskip.5truecm
\ffigbox{
\includegraphics[width=0.4\textwidth]{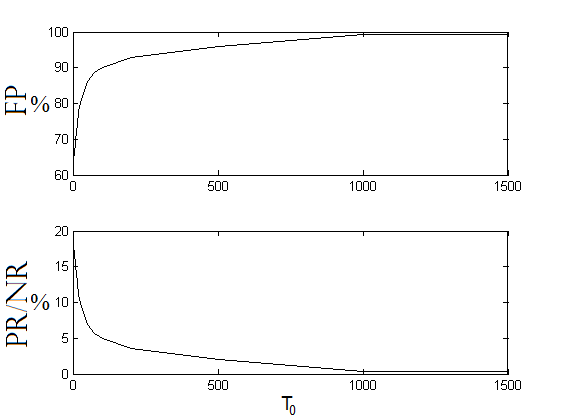}
}{
\caption{\small{Plot of the relative areas of the basins of attraction
as per Table \ref{TAP0272502}.}}
\label{FAP0272502}}
\end{floatrow}
\end{figure}
%%%%%%%%%%%%%%%%%%%%%%%%%%%%%%%%%%%%%%%%%%%%%%%%%%%%%%%%%%%%%%%%

%%%%%%%%%%%%%%%%%%%%%%%%%%%%%%%%%%%%%%%%%%%%%%%%%%%%%%%%%%%%%%%%
\begin{figure}[H]
\begin{floatrow}
\capbtabbox{
\begin{tabular}{cc|c|c|c|c|}
\cline{3-5}
& &\multicolumn{3}{c|}{\multirow{1}{*}{Basin of Attraction \%}}\\
\cline{3-5}
& & FP & PR/NR & DO2 \\
\hline
\multicolumn{1}{|c|}{\multirow{7}{*}{$T_0$}}
& 0 & 25.00 & 12.68 & 49.61 \\ \hhline{~----}
\multicolumn{1}{|c|}{} & 25 & 24.73 & 7.44 & 60.40 \\ \hhline{~----}
\multicolumn{1}{|c|}{} & 50 & 19.67 & 5.35 & 69.63 \\ \hhline{~----}
\multicolumn{1}{|c|}{} & 75 & 17.13 & 4.44 & 73.99 \\ \hhline{~----}
\multicolumn{1}{|c|}{} & 100 & 16.42 & 3.88 & 75.83 \\ \hhline{~----}
\multicolumn{1}{|c|}{} & 200 & 16.29 & 2.72 & 78.28 \\ \hhline{~----}
\multicolumn{1}{|c|}{} & 500 & 17.03 & 0.76 & 81.45 \\ \hline
\end{tabular}}{
\caption{\small{Relative areas of the basins of attraction
with $\gamma_0 = 0.2725$, $\gamma_1 = 0.23$ and $T_0$ varying.}}
\label{TAP02725023}}
\hskip.5truecm
\ffigbox{
\includegraphics[width=0.4\textwidth]{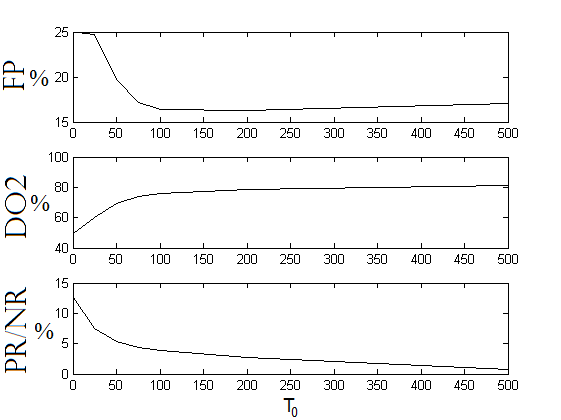}
}{
\caption{\small{Plot of the relative areas of the basins of attraction
as per Table \ref{TAP02725023}.}}
\label{FAP02725023}}
\end{floatrow}
\end{figure}
%%%%%%%%%%%%%%%%%%%%%%%%%%%%%%%%%%%%%%%%%%%%%%%%%%%%%%%%%%%%%%%%

We also notice the interesting features in Table \ref{TAP0272502}:
as the fixed point is the only attractor which exists
for both $\gamma_0$ and $\gamma_1$, we find that as $T_0$
increases its basin of attraction tends towards 100\%,
which corresponds to attraction of the entire phase space,
up to a zero-measure set.
This happens despite the fact that $\gamma(t)$ does not pass through any value
for which global attraction to the fixed point is satisfied.
It also suggests that it is possible to provide conditions
on the intersection of the two sets $\mathcal{A}_0$ and $\mathcal{A}_1$
of the attractors corresponding to $\gamma_0$ and $\gamma_1$, respectively,
in order to obtain that all trajectories move towards the same attractor
when the time $T_0$ over which $\gamma(t)$ is varied is sufficiently large.
In particular, it is remarkable that it is possible to create an attractor for
almost all trajectories by suitably tuning the damping coefficient as a function of time.

In Table \ref{TAP02725023},
the relative areas of the rotating solutions,
which are absent at $\gamma=\gamma_0$, tend to become negligible
when $T_0$ is large. Similarly, in Table \ref{TAP0302725},
the basin of attraction of the period 4 oscillating attractor,
which exists only for the final value $\gamma_1$ of the damping coefficient,
tends to disappear when $T_0$ is taken large enough.
This confirms the general expectation:
the basin of attraction of the disappearing attractor is absorbed by the
closer attractor, that is the solution DO2 in this case.

%%%%%%%%%%%%%%%%%%%%%%%%%%%%%%%%%%%%%%%%%%%%%%%%%%%%%%%%%%%%%%%%
\begin{figure}[H]
\begin{floatrow}
\capbtabbox{
\begin{tabular}{cc|c|c|c|c|}
\cline{3-5}
& &\multicolumn{3}{c|}{\multirow{1}{*}{Basin of Attraction \%}}\\
\cline{3-5}
& & FP & DO2 & DO4 \\
\hline
\multicolumn{1}{|c|}{\multirow{7}{*}{$T_0$}}
& 0 & 17.21 & 79.44 & 3.35 \\ \hhline{~----}
\multicolumn{1}{|c|}{} & 25 & 17.15 & 80.99 & 1.86 \\ \hhline{~----}
\multicolumn{1}{|c|}{} & 50 & 16.83 & 83.03 & 0.14 \\ \hhline{~----}
\multicolumn{1}{|c|}{} & 75 & 16.75 & 83.24 & 0.01 \\ \hhline{~----}
\multicolumn{1}{|c|}{} & 100 & 16.79 & 83.21 & 0.00 \\ \hhline{~----}
\multicolumn{1}{|c|}{} & 200 & 16.81 & 83.19 & 0.00 \\ \hhline{~----}
\multicolumn{1}{|c|}{} & 500 & 16.80 & 83.20 & 0.00 \\ \hline
\end{tabular}}{
\caption{\small{Relative areas of the basins of attraction with $\gamma_0 = 0.3$,
$\gamma_1 = 0.2725$ and $T_0$ varying.}}
\label{TAP0302725}}
\hskip.5truecm
\ffigbox{
\includegraphics[width=0.4\textwidth]{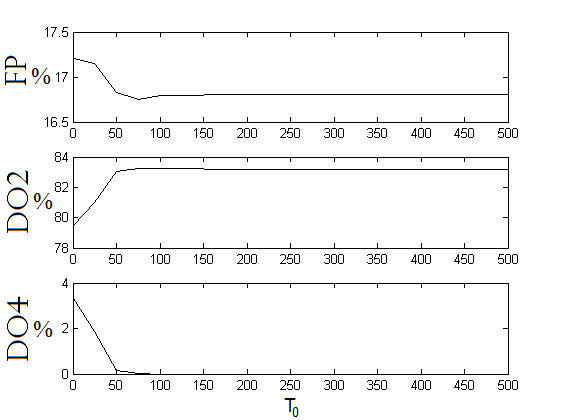}
}{
\caption{\small{Plot of the relative areas of the basins of attraction
as per Table \ref{TAP0302725}.}}
\label{FAP0302725}}
\end{floatrow}
\end{figure}
%%%%%%%%%%%%%%%%%%%%%%%%%%%%%%%%%%%%%%%%%%%%%%%%%%%%%%%%%%%%%%%%

%%%%%%%%%%%%%%%%%%%%%%%%%%%%%%%%%%%%%%%%%%%%%%%%%%%%%%%%%%%%%%%%
\section{Numerical Methods}
\label{NumericsSection}
%%%%%%%%%%%%%%%%%%%%%%%%%%%%%%%%%%%%%%%%%%%%%%%%%%%%%%%%%%%%%%%%

The two main numerical methods implemented for the simulations throughout were
a variable order Adams-Bashforth-Moulton method and the method of analytic
continuation \cite{KBrown,AnalyticContBook1,ProbAnalCont,QuantAnalCont,AnalyticContBook2,AnalyticContPen};
the latter consists in a numerical implementation of the Frobenius method.
Also used to check the results was a Runge-Kutta method. The Adams-Bashforth-Moulton
integration scheme used is the built in integrator found in {\sc matlab},
ODE113, whereas the programs based on the method of analytic continuation and Runge-Kutta scheme were written in C.
Of the three methods, the slowest was the Adams-Bashforth-Moulton method,
however it was found that the method worked well for the system
with the chosen parameter values and the results produced were reliable.

Both the Adams-Bashforth-Moulton and Runge-Kutta
methods are standard methods for solving ODEs of this type.
When implementing these two integrators to calculate the basins of attraction,
two different methods for both choosing initial conditions in phase space
and classifying attractors were used.
The first method for picking initial conditions was to take a mesh of equally
spaced points in the phase space: this method ensures uniform coverage
of the phase space. When taking this approach a mesh of
either 321 141 or 503 289 points was used depending on the desired accuracy.
The second method was to take random initial conditions: this can be done
by choosing initial conditions from a stream of random points,
which allows the user to use the same random initial conditions in each simulation if required.
This method is often preferred as the accuracy of the estimates of the relative areas
of the basins of attraction compared to the number of initial conditions used
can be calculated \cite{ErrorStats}; also it is easier to run additional simulations
for extra random points to improve estimates later on, when needed.
When using random points, the number of points used to calculate the basins of attraction
was 300 000 or 400 000 depending on the expected complexity of the system
under given parameters. In some cases where extra accuracy was required
due to some attractors having particularly
small basins of attraction, additional 200 000 or 300 000 random points were used.
This allowed us to obtain an error less than
0.20 on the relative areas of the basins of attraction.
In the most delicate cases, where more precise estimates were needed to distinguish between
values very close to each other (for instance in Table \ref{table-cubic}), the error
was made smaller by increasing the number of points. We decided to
express in all cases the relative areas up to the second decimal digit because often further
increasing the number of points did not alter appreciably that digit.

To detect and classify solutions, two methods can be used.
The first method consists in finding all attractors as a first step,
before computing the corresponding basins of attraction:
this required a complete characterisation of both the period and
the location in phase space of the attractors.
In principle, this works very well, but has the downside of
having to find initially all the attractors, and differentiate between
those that occupy the same region in phase space.
The second method for classifying the solutions was to create a library of solutions.
This was created as the program ran and built up as new solutions were found.
The solution of each integration was then checked against the library and,
if not already known, was added. In this way, the program finds solutions as it goes,
so has the advantage of the user not having to know the existing solutions
in the system prior to calculating the basins of attraction.

The method of analytic continuation was implemented and produced results
very similar to those of the Adams-Bashforth-Moulton method,
but in general was much quicker to run.
When implementing this method of integration, we only used random initial conditions in
the phase space and the library method for classifying solutions. The reason for using
the Adams-Bashforth-Moulton method, despite being the slowest of the three integrators,
was for comparison with the method of analytic continuation. Analytic continuation has not
previously been used for numerically integrating an ODE of the form \eqref{PenEqn},
that is an equation with an infinite polynomial nonlinearity that satisfies
an addition formula, thus it was good to have
a reliable method to check results with. The method for using
analytic continuation for integrating ODE's that have infinite polynomial nonlinearity that
satisfy an addition formula will be described more extensively in \cite{AnalyticContPen}. The similarity in results
of these completely different methods for integration, initial condition selection and solution
classification provides reassurance and confidence that the results produced are accurate.

%%%%%%%%%%%%%%%%%%%%%%%%%%%%%%%%%%%%%%%%%%%%%%%%%%%%%%%%%%%%%%%%
\section{Concluding remarks}
\label{Conclusions}
%%%%%%%%%%%%%%%%%%%%%%%%%%%%%%%%%%%%%%%%%%%%%%%%%%%%%%%%%%%%%%%%

In this work we have numerically shown the importance of not only
the final value of dissipation but its entire time evolution,
for understanding the long time behaviour of the pendulum with oscillating support.
This extends the work done in \cite{CP} to a system which, even for values
of the parameters in the perturbation regime, exhibits richer and more varied dynamics,
due to the presence of the separatrix in the phase space of the unperturbed system.
In addition we have considered also values of the parameters beyond the perturbation regime
(the inverted pendulum), where the system cannot be considered a perturbation
of an integrable one. In particular this results in a more complicated scenario,
with bifurcation phenomena and the appearance of attractors which exist only for values
of the damping coefficient $\gamma$ in finite intervals away from zero,
say $\gamma\in [\gamma_1 , \gamma_2]$, with $\gamma_1 > 0$.

We have preliminarily studied the behaviour of the system in the case of constant
dissipation. Firstly, in the perturbation regime, we have analytically computed to first order
the threshold values below which the periodic attractors exist.
We have also discussed why this approach fails due to intrinsic perturbation theory
limitations, in particular why the method cannot be applied to stable cases of
the upwards configuration or to solutions too close to the unperturbed separatrix.
Next, we have studied numerically the dependence of the sizes of the
basins of attraction on the damping coefficient.

Then we have explicitly considered the case of damping coefficient varying monotonically
between two values and outlined a few expectations for the way in which
the basins of attraction accordingly change with respect to the
case of constant dissipation. These expectations were later illustrated
and backed up with numerical simulations: in particular the relevance of the
study of the dynamics at constant dissipation was argued at length.
While the expectations account for many features observed numerically,
there are still some facts which are difficult to explain, even at an heuristic level,
and which would deserve further investigation, such as the relationship
between the fixed point solution and the oscillating attractors, to better understand
why in some cases they exchange large areas of their basins of attraction
when the damping coefficient varies in time.
More generally, an in-depth numerical study of the system with constant dissipation,
also for other parameter values, would be worthwhile. In particular it would be interesting
to perform a more detailed bifurcation analysis with respect to the parameter $\gamma$ and
to study the system for very small values of the damping coefficient
$\gamma$ (which relates to the spin-orbit model in celestial mechanics),
both in the perturbation regime and for large values of the forcing amplitude.
We think that, in order to study cases with very small dissipation, the method of analytical
continuation briefly described in Section \ref{NumericsSection} could be particularly fruitful.

\vspace{.3truecm}

Some interesting features appeared in our analysis which would deserve further consideration are:
\begin{itemize}
\item the increase in the basin of attraction of the fixed point observed
in Figure \ref{FAPCFD} when the damping coefficient becomes small enough;
\item the appearance of the period 4 solution for a thin interval of values
of the damping coefficient, as emerges in Table \ref{TAPCF};
\item the rate at which the values of the relative areas of the basins attraction corresponding to
the initial value $\gamma_0$ of $\gamma(t)$ are approached as the variation time $T_0$ increases;
\item the way in which the basins of attraction of the disappearing attractors
distribute among the persisting ones;
\item the oscillations through which the relative areas of the basins of attraction
approach the asymptotic value when taking larger and larger values of the
variation time $T_0$, as observed for instance in Tables
\ref{TAP0202725} and \ref{TAP02302725};
\item the jumps corresponding to $T_0 = 1000$ in Figure \ref{FAP0202725} and $T_0=100$,
$T_0=500$ in Figure \ref{FAP02302725};
\item the computation of the threshold values to second order, so as to include the
rotating solutions found numerically in the perturbation regime investigated in Section \ref{PenNumericSect}.
\end{itemize}

Finally, investigating analogous systems such as the pendulum
with periodically varying length to see if similar dynamics occur would also
be fascinating in its own right. Another interesting model to investigate further,
especially in the case of very small dissipation, is the spin-orbit model
already considered in \cite{CP}, which is expected to be of relevance to
understand the locking into the resonance $3\;$:$\,2$ of the system Mercury-Sun.

%%%%%%%%%%%%%%%%%%%%%%%%%%%%%%%%%%%%%%%%%%%%%%%%%%%%%%%%%%%%%%%%
\section*{Acknowledgements}
%%%%%%%%%%%%%%%%%%%%%%%%%%%%%%%%%%%%%%%%%%%%%%%%%%%%%%%%%%%%%%%%

The Adams-Bashforth-Moulton method used
was {\sc MATLAB}'s  ODE113.
We thank Jonathan Deane for helpful conversations on analytic continuation
and support with coding in C.
This research was completed as part of an EPSRC funded PhD.

%%%%%%%%%%%%%%%%%%%%%%%%%%%%%%%%%%%%%%%%%%%%%%%%%%%%%%%%%%%%%%%%
\appendix
%%%%%%%%%%%%%%%%%%%%%%%%%%%%%%%%%%%%%%%%%%%%%%%%%%%%%%%%%%%%%%%%

%%%%%%%%%%%%%%%%%%%%%%%%%%%%%%%%%%%%%%%%%%%%%%%%%%%%%%%%%%%%%%%%
\section{Global attraction to the two fixed points}
\label{LinearisedApp}
%%%%%%%%%%%%%%%%%%%%%%%%%%%%%%%%%%%%%%%%%%%%%%%%%%%%%%%%%%%%%%%%
To compute the conditions for attraction to the origin we use
the method outlined in \cite{GAtO}; see also \cite{BDG}. We define $f(\tau)$
as in \eqref{PenEqnTau} and require $f(\tau) > 0$: the consequences of this restriction
are that the method can only be applied to the downwards pointing pendulum
when $\alpha > \beta$. Then we apply the Liouville transformation
\begin{equation}
\tilde{\tau} = \int_0^\tau \sqrt{f(s)}\der s
\end{equation}
and write our equation \eqref{PenEqnTau} in terms of the new time $\tilde{\tau}$ as
\begin{equation}
\theta_{\tilde{\tau}\tilde{\tau}} + \left(\frac{\tilde{f}(\tilde{\tau})_{\tilde{\tau}}}{2\tilde{f}(\tilde{\tau})} +
\frac{\gamma}{\sqrt{\tilde{f}(\tilde{\tau})}}\right)\theta_{\tilde{\tau}} + \sin\theta = 0,
\end{equation}
where the subscript $\tilde{\tau}$ represents derivative with respect to the
new time $\tilde{\tau}$ and $\tilde{f}(\tilde{\tau}) := f(\tau)$.
This can be represented as the two-dimensional system on $\mathbb{T}\times\mathbb{R}$,
by setting $x(\tilde{\tau})=\theta(\tilde{\tau})$ and writing
\begin{equation}
x_{\tilde{\tau}} = y, \qquad
y_{\tilde{\tau}} = -\frac{y}{\sqrt{\tilde{f}}}
\left(\frac{\tilde{f}_{\tilde{\tau}}}{2\sqrt{\tilde{f}}} + \gamma\right) - \sin{x},
\end{equation}
for which we have the energy $E(x,y) = 1 -\cos{x} + y^2/2$. By setting $H(\tilde\tau)=E(x(\tilde\tau),y(\tilde\tau))$, one finds
\begin{equation}
H_{\tilde{\tau}} = -\frac{y^2}{\sqrt{\tilde{f}}}\left(\frac{\tilde{f}_{\tilde{\tau}}}{2\sqrt{\tilde{f}}}
+ \gamma\right),
\label{AppendixDecreasingEnergy}
\end{equation}
thus $H_{\tilde{\tau}} \le 0$, i.e $x$, $y$ are bounded given that $\gamma$ satisfies
\begin{equation} \label{largegamma}
\gamma > - \min_{\tilde{\tau} \ge 0}
\frac{\tilde{f}_{\tilde{\tau}}}{2\sqrt{\tilde{f}}} = - \min_{\tau \ge 0} \frac{f'}{2f}.
\end{equation}
Moreover we have that for all $\tilde{\tau} > 0$
\begin{equation}
H(\tilde{\tau}) + \int_0^{\tilde{\tau}} \frac{y^2}{\sqrt{\tilde{f}}}
\left(\frac{\tilde{f}'}{2\sqrt{\tilde{f}}} + \gamma\right) \der s = H(0),
\end{equation}
so that, as $\tilde{\tau} \rightarrow \infty$, using the properties above we can arrive at
\begin{equation}
\min_{s \ge 0} \left[ \frac{1}{\sqrt{\tilde{f}}}\left(
\frac{\tilde{f}_{\tilde{\tau}}}{2\sqrt{\tilde{f}}}
+ \gamma\right)\right]\int_0^\infty y^2(s) \der s < \infty.
\end{equation}
Hence $y \rightarrow 0$ as time tends to infinity. There are two regions of phase space to consider.
Any level curve of $H$ strictly inside the separatrix of the unperturbed pendulum is the boundary of
a positively invariant set $D$ containing the origin: since
$S = \{(x(\tilde\tau),y(\tilde\tau)) : H_{\tilde\tau} = 0\} \cup D$
consists purely of the origin, we can apply the local Barbashin-Krasovsky-La Salle
theorem \cite{Kr} to conclude that every trajectory that begins strictly inside
the separatrix will converge to the origin as $\tilde\tau \rightarrow +\infty$.

Outside of the separatrix we may use equation \eqref{AppendixDecreasingEnergy}, which shows the energy
to be strictly decreasing while $y \neq 0$, provided $\gamma$ is chosen large enough, coupled with
$y \rightarrow 0$ as time tends to infinity. The result is that all
trajectories tend to the invariant points on the $x$-axis as time tends to infinity. One of two cases
must occur: either the trajectory moves inside the separatrix or it does not. In the first instance we
have already shown that the limiting solution is the origin. In the latter there is only one possibility.
As all points on the $x$-axis are contained within the separatrix other than the unstable
fixed point, the trajectory must move onto such a fixed point and hence belongs to its stable manifold,
which is a zero-measure set.
Therefore we conclude that a full measure set of initial conditions are attracted by the origin.
Reverting back to the original system with time $\tau$, we conclude that for that system too
the basin of attraction of the origin has full measure,
provided $\beta<\alpha$ and $\gamma$ satisfies \eqref{largegamma}.

%%%%%%%%%%%%%%%%%%%%%%%%%%%%%%%%%%%%%%%%%%%%%%%%%%%%%%%%%%%%%%%%
\section{Action-angle variables}
\label{ActionAngleApp}
%%%%%%%%%%%%%%%%%%%%%%%%%%%%%%%%%%%%%%%%%%%%%%%%%%%%%%%%%%%%%%%%

In this section we detail the calculation of the action-angle variables for the simple pendulum in time $\tau$.
More details on calculating action-angle variables can be found in \cite{AnalMech,EngAAV,Brizard}.
The simple pendulum has equation of motion given by
\begin{equation}
\theta''  + \alpha\sin{\theta} = 0,
\end{equation}
where the dashes represent derivative with respect to the scaled time $\tau$.
The Hamiltonian for the simple pendulum in this notation is
\begin{equation}\label{Hamil_Theta}
E = H(\theta,\theta') = \frac{1}{2} (\theta')^2 - \alpha \cos{\theta}
\end{equation}
or, in terms of the usual notation for Hamiltonian dynamics,
\begin{equation}\label{SimplePenpqHamil}
E = H(p,q) = \frac{1}{2} p^2 - \alpha \cos{q},
\end{equation}
where $q = \theta$ and $p= q'=\theta'$. Rearranging this for $p$ we obtain
$p=\pm p(E,q)$, with
\begin{equation}
p(E,q) = \sqrt{2(E + \alpha \cos{q})} = \sqrt{2\alpha (E_{0} + \cos{q})},
\end{equation}
where $E_{0} = E/\alpha$. It is clear that there are two types of dynamics,
oscillatory dynamics when $E_{0} < 1$ and rotational dynamics when $E_{0} > 1$,
separated at a separatrix when $E_{0} = 1$, for which no action-angle variables exist.

%%%%%%%%%%%%%%%%%%%%%%%%%%%%%%%%%%%%%%%%%%%%%%%%%%%%%%%%%%%%%%%%
\subsection{Librations}
%%%%%%%%%%%%%%%%%%%%%%%%%%%%%%%%%%%%%%%%%%%%%%%%%%%%%%%%%%%%%%%%

We first consider the case $E_{0} < 1$. The action variable is
\begin{equation} \label{LibAct}
\begin{split}
I = \frac{1}{2\pi} \oint p \der q
= \frac{2}{\pi}\sqrt{2\alpha} \int_{0}^{q_1} \sqrt{E_{0} + \cos{q}} \thickspace \der q
= \frac{8}{\pi}\sqrt{\alpha}\Bigl[(k_1^2 - 1){\bf K}(k_1) + {\bf E}(k_1)\Bigr],
\end{split}
\end{equation}
where $k_1^2 = (E_{0}+1)/2$ and $q_{1}=\arccos(-E_{0})$. The functions ${\bf K}(k)$ and ${\bf E}(k)$
are the complete elliptic integrals of the first and second kinds respectively.

The angle variable $\varphi$ can be found as follows
\begin{equation}\label{dotPhiEqn}
\varphi' = \frac{\partial H}{\partial I } = \frac{\der E}{\der I} =
\left(\frac{\der I}{\der E}\right)^{-1},
\end{equation}
so that
\begin{equation}
\frac{\der I}{\der E} = \frac{\der }{\der E} \frac{2}{\pi}
\int_{0}^{q_1} \sqrt{E + \alpha \cos{q}} \thickspace \der q \\
= \frac{2}{\pi \sqrt{\alpha}} {\bf K}(k_1).
\end{equation}
Hence we have
\begin{equation} \label{varphi-osc}
\varphi (\tau) = \frac{\pi}{2{\bf K}(k_1)} \sqrt{\alpha}(\tau - \tau_0).
\end{equation}
Take $s = \sin{(q/2)}$; then using equation \eqref{Hamil_Theta} it is easy to show that
\begin{equation}\label{sdotsq}
(s')^2 = \frac{g}{l}(1-s^2)\left(k_1^2 - s^2\right).
\end{equation}
Integrating using the Jacobi elliptic functions
\begin{equation}
s(\tau) = k_1 \sn{\left(\sqrt{\frac{g}{l}}(\tau - \tau_0),k_1\right)},
\end{equation}
the expresssion can then be rearranged to achieve the following result:
\begin{equation}\begin{split}
q = 2\arcsin{\left[k_1 \sn{\left(\frac{2{\bf K}(k_1)}{\pi}\varphi,k_1\right)}\right]}, \qquad
p = 2 k_1 \sqrt{\alpha} \cn{\left(\frac{2{\bf K}(k_1)}{\pi}\varphi,k_1\right)},
\end{split}
\end{equation}
which coincide with equations \eqref{action-angle-librations}. By using \eqref{EKDer}
in Appendix \ref{AppUsefulIntsAndExps}, one obtains from \eqref{LibAct}
\begin{equation} \label{LibDIDK}
\frac{\partial I}{\partial k_{1}} = \frac{8}{\pi} k_{1} {\bf K}(k_{1}) \, \sqrt{\alpha} ,
\end{equation}
a relation which has been used to derive  \eqref{OmegaDer}.

%%%%%%%%%%%%%%%%%%%%%%%%%%%%%%%%%%%%%%%%%%%%%%%%%%%%%%%%%%%%%%%%
\subsection{Rotations}
%%%%%%%%%%%%%%%%%%%%%%%%%%%%%%%%%%%%%%%%%%%%%%%%%%%%%%%%%%%%%%%%

In the case of rotational dynamics we have
\begin{equation} \label{RotAct}
\begin{split}
I = \frac{1}{2\pi} \int_{0}^{2\pi} p \, \der q
= \frac{1}{2\pi}\sqrt{\alpha} \int_{0}^{2\pi} \sqrt{E_{0} + \cos{q}} \thickspace \der q
= \frac{4}{k_2\pi}\sqrt{\alpha} \, {\bf E}(k_2),
\end{split}
\end{equation}
where this time we let $k_2^2 = 2/(E_{0} + 1) = 1/k_1^2$. The angle variable $\varphi$
can similarly be found using \eqref{dotPhiEqn}, where
$\der I/\der E$ can be similarly calculated as
\begin{equation}
\frac{\der I}{\der E} = \frac{\der }{\der E} \frac{1}{2\pi}
\int_0^{2\pi} \sqrt{E +\alpha \cos{q}} \thickspace \der q \\
= \frac{k_2}{\pi \sqrt{\alpha}} {\bf K}(k_2),
\end{equation}
which hence gives
\begin{equation} \label{varphi-rot}
\varphi(\tau)  = \frac{\pi}{{\bf K}(k_2)}\sqrt{\alpha}\frac{(\tau - \tau_0)}{k_2}.
\end{equation}
Using \eqref{sdotsq} and the definition of $k_2$, for the rotating solutions we find that
\begin{equation}
s(\tau) = \sn{\left(\sqrt{\alpha} \frac{(\tau - \tau_0)}{k_2},k_2\right)} ,
\end{equation}
and similarly, by simple rearrangement, we find that
\begin{equation}\begin{split}
q = 2 \arcsin{\left[ \sn{\left(\frac{{\bf K}(k_2)}{\pi}\varphi,k_2\right)}\right]}, \qquad
p = \frac{2}{k_2} \sqrt{\alpha} \, \dn{\left(\frac{{\bf K}(k_2)}{\pi}\varphi,k_2\right)},
\end{split}
\end{equation}
which again yields equations \eqref{action-angle-rotations}
By using \eqref{EKDer} in Appendix \ref{AppUsefulIntsAndExps},
one obtains from \eqref{RotAct}
\begin{equation} \label{RotDIDK}
\frac{\partial I}{\partial k_{2}} = - \frac{4}{\pi k_{2}^{2}} {\bf K}(k_{2}) \,  \sqrt{\alpha} .
\end{equation}
which has been used to derive  \eqref{OmegaDerRot}.

%%%%%%%%%%%%%%%%%%%%%%%%%%%%%%%%%%%%%%%%%%%%%%%%%%%%%%%%%%%%%%%%
\section{Jacobian determinant}
\label{Jacobian}
%%%%%%%%%%%%%%%%%%%%%%%%%%%%%%%%%%%%%%%%%%%%%%%%%%%%%%%%%%%%%%%%

Here we compute the entries of the Jacobian matrix $J$ of the transformation to action-angle variables,
which will be used in the next Appendix. As a by-product
we check that $J$ determinant equal to 1, that is
\begin{equation}\label{JacobiEqn}
\frac{\partial q}{\partial \varphi}\frac{\partial p}{\partial I} - \frac{\partial q}{\partial I}\frac{\partial p}{\partial \varphi}  = 1.
\end{equation}
For further details on the proof of \eqref{JacobiEqn} we refer the reader to \cite{Brizard},
where the calculations are given in great detail.
The derivative with respect to $\varphi$ is straightforward in both the libration and rotation case,
however the dependence of $p$ and $q$ on the action $I$ is less obvious. That said, the dependence of $p$
and $q$ on $k_1$ in the oscillating case and $k_2$ in the rotating case is clear and we know the relationship
between $I$ and $k$ in both cases, hence the derivative of the
Jacobi elliptic functions can be calculated by using that
\begin{equation}\label{DwrtIEqn}
\frac{\partial}{\partial I} = \frac{\partial k}{\partial I}\frac{\partial}{\partial k} +
\frac{\partial u}{\partial I}\frac{\partial}{\partial u} = \frac{\partial k}{\partial I}\left(\frac{\partial}{\partial k} +
\frac{\partial u}{\partial k}\frac{\partial}{\partial u}\right),
\end{equation}
where $u$ is the first argument of the functions, i.e $\sn(u,k)$, etc.
Then for the oscillations we have
\begin{equation}\begin{split}
\frac{\partial q}{\partial I} & = \frac{\pi}{4k_1{\bf K}(k_1)\sqrt{\alpha}}\left[\frac{\sn(\cdot)}{\dn(\cdot)} +
\frac{2{\bf E}(k_1)\varphi\cn(\cdot)}{\pi k_{1}'^2} + \frac{k_1^2\sn(\cdot)\cn^2(\cdot)}{k_1'^2\dn(\cdot)} -
\frac{{\bf E}(\cdot)\cn(\cdot)}{k_1'^2} \right], \\
\frac{\partial p}{\partial I} & = \frac{\pi}{4k_1{\bf K}(k_1)}\left[\cn(\cdot) -
\frac{2{\bf E}(k_1)\varphi\sn(\cdot)\dn(\cdot)}{\pi k_{1}'^2} -
\frac{k_1^2\sn^2(\cdot)\cn(\cdot)}{k_1'^2} +
\frac{{\bf E}(\cdot)\sn(\cdot)\dn(\cdot)}{k_1'^2}\right],\\
\frac{\partial q}{\partial \varphi} & = \frac{4k_1{\bf K}(k_1)\cn(\cdot)}{\pi} ,\\
\frac{\partial p}{\partial \varphi} & = -\sqrt{\alpha} \, \frac{4k_1{\bf K}(k_1)\sn(\cdot)\dn(\cdot)}{\pi},
\end{split}
\end{equation}
where $(\cdot) = \left(\frac{2{\bf K}(k_1)\varphi}{\pi},k_1\right)$ and $k_1' = \sqrt{1 - k_1^2}$.
From the above it is easy to check that equation \eqref{JacobiEqn} is satisfied.
Similarly, for the rotations we have
\begin{equation}\begin{split}
\frac{\partial q}{\partial I} & = - \frac{\pi k_2^2}{2 {\bf K}(k_2) \, \sqrt{\alpha}}
\left[ \frac{\varphi \, {\bf E}(k_{2})\,\dn(\cdot)}{\pi k_{2}k_{2}'^2} +
\frac{k_2\sn(\cdot)\cn(\cdot)}{k_2'^2} - \frac{{\bf E}(\cdot)\dn(\cdot)}{k_2k_2'^2}\right], \\
\frac{\partial p}{\partial I} & = \frac{\pi k_2^2}{2 {\bf K}(k_2)}
\left[ \frac{\dn(\cdot)}{k_2^2} + \frac{\varphi \, {\bf E}(k_{2})\,\sn(\cdot) \, \cn(\cdot)}{\pi k_{2}'^2} +
\frac{\sn^2(\cdot)\dn(\cdot)}{k_2'^2} - \frac{{\bf E}(\cdot)\sn(\cdot)\cn(\cdot)}{k_2'^2} \right], \\
\frac{\partial q}{\partial \varphi} & = \frac{2{\bf K}(k_2)\dn(\cdot)}{\pi},\\
\frac{\partial p}{\partial \varphi} & = -\sqrt{\alpha} \, \frac{2k_2{\bf K}(k_2)\sn(\cdot)\cn(\cdot)}{\pi} ,
\end{split}
\end{equation}
where $(\cdot) = \left(\frac{{\bf K}(k_2)}{\pi}\varphi,k_2\right)$ and $k_2' = \sqrt{1 - k_2^2}$.
It is once again easily checked from the above that \eqref{JacobiEqn} is satisfied.

%%%%%%%%%%%%%%%%%%%%%%%%%%%%%%%%%%%%%%%%%%%%%%%%%%%%%%%%%%%%%%%%
\section{Equations of motion for the perturbed system}
\label{AppPerturbedSys}
%%%%%%%%%%%%%%%%%%%%%%%%%%%%%%%%%%%%%%%%%%%%%%%%%%%%%%%%%%%%%%%%

By \eqref{JacobiEqn} one has
\begin{equation}
\begin{pmatrix} \partial \varphi/\partial q & \partial \varphi / \partial p \\
\partial I /\partial q & \partial I / \partial p
\end{pmatrix}
=
\begin{pmatrix} \partial p / \partial I & - \partial q / \partial I \\
- \partial p/ \partial \varphi & \partial q / \partial \varphi .
\end{pmatrix}
\end{equation}
We rewrite the equation \eqref{PenEqnTau} in the action-angle coordinates
introduced in Appendix \ref{ActionAngleApp} as follows.

%%%%%%%%%%%%%%%%%%%%%%%%%%%%%%%%%%%%%%%%%%%%%%%%%%%%%%%%%%%%%%%%
\subsection{Librations}
%%%%%%%%%%%%%%%%%%%%%%%%%%%%%%%%%%%%%%%%%%%%%%%%%%%%%%%%%%%%%%%%

In this section we want to write \eqref{PenEqnTau} in terms of the action-angle introduced in
Appendix \ref{ActionAngleApp}.
By taking into account the forcing term $-\beta \cos\tau \cos\theta$ in \eqref{PenEqnTau} one finds
\begin{equation}
\begin{split}
I' & = \frac{\partial I}{\partial q}q' + \frac{\partial I}{\partial p}p' = -
\frac{\partial p}{\partial \varphi}q' + \frac{\partial q}{\partial \varphi}p' \\
& = \frac{8 \beta k_1^2 {\bf K}(k_1)}{\pi}\cos(\tau - \tau_0)\sn(\cdot)\cn(\cdot)\dn(\cdot), \\
\varphi' & = \frac{\partial \varphi}{\partial q}q' + \frac{\partial \varphi}{\partial p}p' =
\frac{\partial p}{\partial I}q' - \frac{\partial q}{\partial I}p' \\
& = \frac{\pi \sqrt{\alpha}}{2{\bf K}(k_1)}  - \frac{\pi \beta}{2{\bf K}(k_1)\,\sqrt{\alpha}}
\left[ \sn^2(\cdot) + \frac{2{\bf E}(k_1)\varphi \sn(\cdot)\cn(\cdot)\dn(\cdot)}{\pi k_{1}'^2} \right. \\
& \qquad \qquad + \left. \frac{k_1^2\sn^2(\cdot)\cn^2(\cdot)}{1-k_1^2} -
\frac{{\bf E}(\cdot) \sn(\cdot)\cn(\cdot)\dn(\cdot)}{1 - k_1^2}\right] \cos(\tau - \tau_0) ,
\end{split}
\end{equation}
where we have used the properties of the Jacobi elliptic functions in
Appendix \ref{AppUsefulIntsAndExps}. As in Appendix \ref{Jacobian},
we are shortening $(\cdot) = \left(\frac{2{\bf K}(k_1)\varphi}{\pi},k_1\right)$.

We then wish to add the dissipative term $\gamma \theta'$.
This results in the following equations:
\begin{equation}
\begin{split}
I' & = \frac{8 \beta k_1^2 {\bf K}(k_1)}{\pi} \cos(\tau - \tau_0)\sn(\cdot)\cn(\cdot)\dn(\cdot) -
\frac{8 \gamma k_1^2 \sqrt{\alpha} \, {\bf K}(k_1)}{\pi}\cn^2(\cdot), \\
\varphi' & = \frac{\pi \sqrt{\alpha}}{2{\bf K}(k_1)} - \frac{\pi \beta}{2{\bf K}(k_1)\,\sqrt{\alpha}}
\left[ \sn^2(\cdot) + \frac{2{\bf E}(k_1)\varphi \sn(\cdot)\cn(\cdot)\dn(\cdot)}{\pi (1-k_{1}^2)}\right. \\
& \qquad \qquad \left. +
\frac{k_1^2\sn^2(\cdot)\cn^2(\cdot)}{1-k_1^2} - \frac{{\bf E}(\cdot)
\sn(\cdot)\cn(\cdot)\dn(\cdot)}{1 - k_1^2}\right]\cos(\tau - \tau_0) \\
& \qquad \qquad + \frac{\gamma \pi \cn(\cdot)}{2{\bf K}(k_1)}\left[ \frac{\sn(\cdot)}{\dn(\cdot)} \right.
\left. + \frac{2 {\bf E}(k_1)\varphi \cn(\cdot)}{\pi (1-k_{1}^2)} +
\frac{k_1^2\sn(\cdot)\cn^2(\cdot)}{(1-k_1^2)\dn(\cdot)} -
\frac{ {\bf E}(\cdot)\cn(\cdot)}{1 - k_1^2}\right].
\end{split}
\end{equation}
Using the property that, see \cite{DFLaw}, ${\bf E}(u,k) = {\bf E}(k)u/{\bf K}(k) + {\bf Z}(u,k)$
we arrive at equations \eqref{PerturbedAAvariablesOscTau}.
The function ${\bf Z}(u,k)$ is the Jacobi zeta function, which is
periodic with period $2{\bf K}(k)$ in $u$.

%%%%%%%%%%%%%%%%%%%%%%%%%%%%%%%%%%%%%%%%%%%%%%%%%%%%%%%%%%%%%%%%
\subsection{Rotations}
%%%%%%%%%%%%%%%%%%%%%%%%%%%%%%%%%%%%%%%%%%%%%%%%%%%%%%%%%%%%%%%%

The presence of the forcing term leads to te equations
\begin{equation}\begin{split}
I' & = - \frac{\partial p}{\partial \varphi}q' + \frac{\partial q}{\partial \varphi}p'
% \\ &
= \frac{4 \beta {\bf K}(k_2)}{\pi}\cos(\tau - \tau_0) \sn(\cdot)\cn(\cdot)\dn(\cdot), \\
\varphi' & = \frac{\partial p}{\partial I}\dot{q} - \frac{\partial q}{\partial I}\dot{p}
% \\ &
= \frac{\pi\sqrt{\alpha}}{k_2{\bf K}(k_2)} + \frac{\pi k_2 \beta}{\sqrt{\alpha}{\bf K}(k_2)}
\left[ \frac{ {\bf E}(k_{2}) \, \varphi \,\sn(\cdot)\,\cn(\cdot)\,\dn(\cdot)}{\pi(1-k_{2}^2)} \right. \\
& \qquad \qquad + \left.
\frac{k^2_2 \sn^2(\cdot)\cn^2(\cdot)}{1 - k_2^2} -
\frac{{\bf E}(\cdot)\sn(\cdot)\cn(\cdot)\dn(\cdot)}{1 - k_2^2} \right]\cos(\tau - \tau_0) .
\end{split}
\end{equation}
Again, if we wish to add a dissipative term, we arrive at the equations
\begin{equation}\begin{split}
I' & = \frac{4 \beta {\bf K}(k_2)}{\pi}\cos(\tau - \tau_0) \sn(\cdot)\cn(\cdot)\dn(\cdot) -
\frac{4\gamma \sqrt{\alpha} \, {\bf K}(k_2)}{\pi k_2}\dn^2(\cdot), \\
\varphi' & = \frac{\pi\sqrt{\alpha}}{k_2{\bf K}(k_2)} + \frac{\pi k_2 \beta}{\sqrt{\alpha}{\bf K}(k_2)}
\left[ \frac{ {\bf E}(k_{2}) \, \varphi \,\sn(\cdot)\,\cn(\cdot)\,\dn(\cdot)}{\pi(1-k_{2}^2)} \right. \\
& \qquad \qquad + \left. \frac{k^2_2 \sn^2(\cdot)\cn^2(\cdot)}{1 - k_2^2} - \frac{{\bf E}(\cdot)
\sn(\cdot)\cn(\cdot)\dn(\cdot)}{1 - k_2^2} \right]\cos(\tau - \tau_0) \\
& \qquad \qquad - \frac{\gamma \pi}{{\bf K}(k_2)}
\left[ \frac{ {\bf E}(k_{2}) \, \varphi \,\dn^{2} (\cdot)}{\pi(1-k_{2}^2)} +
\frac{k_2^2\sn(\cdot) \cn(\cdot)\dn(\cdot)}{1 - k_2^2} - \frac{ {\bf E}(\cdot)\dn^2(\cdot)}{1-k_2^2} \right] .
\end{split}
\end{equation}
Again using that ${\bf E}(u,k) = {\bf E}(k)u/{\bf K}(k) + {\bf Z}(u,k)$ we arrive at the equations
\eqref{PerturbedAAvariablesRotTau}.
%%%%%%%%%%%%%%%%%%%%%%%%%%%%%%%%%%%%%%%%%%%%%%%%%%%%%%%%%%%%%%%%
\section{Useful properties of the elliptic functions}
\label{AppUsefulIntsAndExps}
%%%%%%%%%%%%%%%%%%%%%%%%%%%%%%%%%%%%%%%%%%%%%%%%%%%%%%%%%%%%%%%%

The complete integrals of the first and second kind are, respectively,
\begin{equation} \label{CompInt}
{\bf K}(k) = \int_{0}^{\pi/2} \frac{\der \psi}{\sqrt{1-k^2 \sin^2 \psi}} , \qquad
{\bf E}(k) = \int_{0}^{\pi/2} \der \psi \, \sqrt{1-k^2 \sin^2 \psi} ,
\end{equation}
whereas the incomplete elliptic integral of the second kind is
\begin{equation} \label{Int}
{\bf E}(u,k) = \int_{0}^{\sn(u,k)} {\rm d}x \frac{\sqrt{1-k^2 x^{2}}}{\sqrt{1-x^2}}  .
\end{equation}
One has
\begin{equation} \label{EKDer}
\frac{\partial {\bf K}(k)}{\partial k} =
\frac{1}{k} \left( \frac{{\bf E}(k)}{1-k^2} - {\bf K}(k) \right), \qquad
\frac{\partial {\bf E}(k)}{\partial k} =
\frac{1}{k} \left( {\bf E}(k) - {\bf K}(k) \right) .
\end{equation}

The  following properties of the Jacobi elliptic functions have been used in the previous sections.
The derivatives with respect to the first arguments are
\begin{equation} \label{uDer}
\begin{split}
\frac{\partial}{\partial u} \sn(u,k) & = \cn(u,k) \, \dn (u,k) , \\
\frac{\partial}{\partial u} \cn(u,k) & = - \sn(u,k) \, \dn (u,k) , \\
\frac{\partial}{\partial u} \dn(u,k) & = - k^2 \sn(u,k) \, \cn (u,k) ,
\end{split}
\end{equation}
while the derivatives with respect to the elliptic modulus are
\begin{equation} \label{kDer}
\begin{split}
\frac{\partial}{\partial k} \sn(u,k) & = \frac{u}{k} \cn(u,k) \, \dn (u,k)
+ \frac{k}{k'^2} \, \sn(u,k) \, \cn^{2} (u,k) -
\frac{1}{k k'^2} {\bf E}(u,k)\,\cn(u,k) \, \dn(u,k) , \\
\frac{\partial}{\partial k} \cn(u,k) & = - \frac{u}{k} \sn(u,k) \, \dn (u,k)
- \frac{k}{k'^2} \, \sn^2(u,k) \, \cn (u,k) +
\frac{1}{k k'^2} {\bf E}(u,k)\,\sn(u,k) \, \dn(u,k) , \\
\frac{\partial}{\partial k} \dn(u,k) & = -ku \, \sn(u,k) \, \cn (u,k)
- \frac{k}{k'^2} \, \sn^{2}(u,k) \, \dn (u,k) +
\frac{k}{k'^2} {\bf E}(u,k)\,\sn(u,k) \, \cn(u,k) ,
\end{split}
\end{equation}
where $k'^2=1-k^2$.

Finding the value of $\Delta$ for rotations in Section \ref{ThresholdsSect} requires use of
\begin{equation}
\int_0^{x_1} \dn^2(x,k) \thinspace \der x = \int_0^{\sn(x_1,k)}
\frac{\sqrt{1 - k^2\hat{x}^2}}{\sqrt{1 - \hat{x}^2}} \thinspace \der \hat{x} = {\bf E}(x_1,k) .
\end{equation}
In the case of librations we also require the relation $k^2 \cn^2(\cdot) + (1 - k^2) = \dn^2(\cdot)$.

The integral for $\Gamma_1 (\tau_0;\mathpzc{p},\mathpzc{q})$
in equation \eqref{GammaEqnLib} is found by
\begin{equation}\begin{split}
\Gamma_1 (\tau_0;\mathpzc{p},\mathpzc{q}) = &
\frac{1}{T}\int_0^{T}\sn(\sqrt{\alpha}\tau)\cn(\sqrt{\alpha}\tau)
\dn(\sqrt{\alpha}\tau)\cos(\tau - \tau_0) \thinspace \der \tau \\
= & \frac{\cos(\tau_0)}{T}\int_0^{T}\sn(\sqrt{\alpha}\tau)
\cn(\sqrt{\alpha}\tau)\dn(\sqrt{\alpha}\tau)\cos(\tau) \thinspace \der \tau \\
&  + \frac{\sin(\tau_0)}{T}\int_0^{T}\sn(\sqrt{\alpha}\tau)
\cn(\sqrt{\alpha}\tau)\dn(\sqrt{\alpha}\tau)\sin(\tau) \thinspace \der \tau \\
= & \frac{\sin(\tau_0)}{T}\int_0^{T}\sn(\sqrt{\alpha}\tau)
\cn(\sqrt{\alpha}\tau)\dn(\sqrt{\alpha}\tau)\sin(\tau) \thinspace \der \tau ,\\
\end{split}
\end{equation}
where $T=2\pi \mathpzc{q}=4{\bf K}(k_{1})\mathpzc{p}$.

The Jacobi elliptic functions can be expanded in a Fourier series as
\begin{equation}\begin{split}
\sn(u,k) = \frac{2\pi}{k{\bf K}(k)}\sum_{n=1}^\infty \frac{\mathfrak{q}^{n - 1/2}}{1 -
\mathfrak{q}^{2n - 1}}\sin\left(\frac{(2n - 1)\pi u}{2{\bf K}(k)}\right),\\
\cn(u,k) = \frac{2\pi}{k{\bf K}(k)}\sum_{n=1}^\infty \frac{\mathfrak{q}^{n - 1/2}}{1 +
\mathfrak{q}^{2n - 1}}\cos\left(\frac{(2n - 1)\pi u}{2{\bf K}(k)}\right),\\
\dn(u,k) = \frac{\pi}{2{\bf K}(k)}+\frac{2\pi}{{\bf K}(k)}\sum_{n=1}^\infty
\frac{\mathfrak{q}^{n}}{1 - \mathfrak{q}^{2n}}\cos\left(\frac{2n\pi u}{2{\bf K}(k)}\right),\\
\end{split}\end{equation}
where $\mathfrak{q}$ is the nome, defined as
\[\mathfrak{q} = \exp \left( - \frac{ \pi{\bf K}(k')}{{\bf K}(k)}\right) ,\]
with $k'=\sqrt{1-k^2}$.

In the calculation of $\langle R^{(n)} \rangle$ for $n \ge 2$,
when the pendulum is in libration, we require the evaluation of the integrals
\begin{equation}
\frac{1}{T}\int_0^T \frac{2{\bf K}(k_1)}{\sqrt{\alpha}\pi}
\frac{\partial}{\partial \tau}\Bigl(\cn^2(\sqrt{\alpha}\tau)\Bigr)\thinspace \der \tau = 0,
\end{equation}
and
\begin{equation}\begin{split}
& \frac{1}{T}\int_0^T \frac{2{\bf K}(k_1)}{\sqrt{\alpha}\pi}\frac{\partial}{\partial \tau}
\Bigl(\sn(\sqrt{\alpha}\tau)\cn(\sqrt{\alpha}\tau)\dn(\sqrt{\alpha}\tau)
\Bigr)\cos(\tau - \tau_0)\thinspace \der \tau \\
& \qquad = \frac{1}{T}\int_0^T \frac{2{\bf K}(k_1)}{\sqrt{\alpha}\pi}\sn(\sqrt{\alpha}\tau)
\cn(\sqrt{\alpha}\tau)\dn(\sqrt{\alpha}\tau)\sin(\tau - \tau_0)\thinspace \der \tau \\
& \qquad = \frac{\cos(\tau_0)}{T}\int_0^T \frac{2{\bf K}(k_1)}{\sqrt{\alpha}\pi}
\sn(\sqrt{\alpha}\tau)\cn(\sqrt{\alpha}\tau)\dn(\sqrt{\alpha}\tau)\sin(\tau)\thinspace \der \tau \\
& \qquad \qquad - \frac{\sin(\tau_0)}{T}\int_0^T \frac{2{\bf K}(k_1)}{\sqrt{\alpha}\pi}
\sn(\sqrt{\alpha}\tau)\cn(\sqrt{\alpha}\tau)\dn(\sqrt{\alpha}\tau)\cos(\tau)\thinspace \der \tau ,
\end{split}
\end{equation}
where $T=2\pi \mathpzc{q}=4{\bf K}(k_{1})\mathpzc{p}$.
The integral multiplying $\sin(\tau_0)$ vanishes due to parity and hence
\begin{equation}
\frac{2{\bf K}(k_1)}{\sqrt{\alpha}\pi T}\int_0^T \frac{\partial}{\partial \tau}
\Bigl(\sn(\sqrt{\alpha}\tau)\cn(\sqrt{\alpha}\tau)\dn(\sqrt{\alpha}\tau)\Bigr)
\cos(\tau - \tau_0)\thinspace \der \tau = \frac{2{\bf K}(k_1)}{\sqrt{\alpha}\pi}
\cos(\tau_0)G_1(\mathpzc{p},\mathpzc{q}).
\end{equation}

%%%%%%%%%%%%%%%%%%%%%%%%%%%%%%%%%%%%%%%%%%%%%%%%%%%%%%%%%%%%%%%%

%%%%%%%%%%%%%%%%%%%%%%%%%%%%%%%%%%%%%%%%%%%%%%%%%%%%%%%%%%%%%%%%

\end{document}